\documentclass[reqno,10pt]{amsart}
\usepackage{amsfonts, latexsym, amssymb, amsthm, amsmath, mathrsfs, esint, color, comment, fullpage}
\usepackage[left=2.8cm,right=2.8cm,top=2.8cm,bottom=2.8cm]{geometry}
\usepackage{amscd}
\usepackage{enumitem} 
\usepackage{mathtools}
\usepackage{graphicx}
\graphicspath{  }
\usepackage{relsize}
\usepackage{url} 
\usepackage[colorlinks=true, pdfstartview=FitV, linkcolor=blue, 
            citecolor=blue, urlcolor=blue]{hyperref}
\usepackage{bbm}
\usepackage{cleveref}
\numberwithin{equation}{section}

\newtheorem{theorem}{Theorem}[section]
\newtheorem{lemma}[theorem]{Lemma}
\newtheorem{proposition}[theorem]{Proposition}
\newtheorem{corollary}[theorem]{Corollary}
\newtheorem{assumption}{Assumption}
\newtheorem{definition}[theorem]{Definition}
\newtheorem{remark}[theorem]{Remark}
\theoremstyle{remark}

\newcommand{\EEE}{\color{black}}
\newcommand{\III}{\color{black}}

\newcommand{\BBB}{\color{black}}
\newcommand{\CCC}{\color{black}}

\newcommand{\ep}{\varepsilon}

\newcommand{\RED}{\color{black}}
\newcommand{\BLACK}{\color{black}}

\renewcommand{\div}{{\rm div}}

\newcommand{\Laplace}{\triangle}
\DeclareMathOperator{\dev}{dev}

\DeclareMathOperator{\dist}{dist}

\DeclareMathOperator{\supp}{supp}
\DeclareMathOperator{\sym}{sym}
\DeclareMathOperator{\tr}{tr}

\newcommand{\A}{\mathbb{A}}
\newcommand{\calA}{\mathcal{A}}

\newcommand{\C}{\mathbb{C}}

\newcommand{\calD}{\mathcal{D}}
\newcommand{\calE}{\mathcal{E}}
\newcommand{\calH}{\mathcal{H}}
\newcommand{\calI}{\mathcal{I}}
\newcommand{\calJ}{\mathcal{J}}
\newcommand{\calK}{\mathcal{K}}
\newcommand{\calL}{\mathcal{L}}
\newcommand{\M}{\mathbb{M}}

\newcommand{\Mb}{\mathcal{M}_b}
\newcommand{\N}{\mathbb{N}}

\newcommand{\calQ}{\mathcal{Q}}
\newcommand{\R}{\mathbb{R}}

\newcommand{\calS}{\mathcal{S}}

\newcommand{\U}{U}
\newcommand{\calU}{\mathcal{U}}

\newcommand{\calV}{\mathcal{V}}

\newcommand{\calX}{\mathcal{X}}
\newcommand{\calY}{\mathcal{Y}}
\newcommand{\Z}{\mathbb{Z}}

\newcommand{\calZ}{\mathcal{Z}}

\newcommand{\Dir}{D}

\newcommand{\DiffOpA}{\mathcal{A}}
\newcommand{\LinOpA}{A}
\newcommand{\DiffOp}{\mathcal{A}_{x_2}}
\newcommand{\CorrSpace}[3]{\calX^{{#1}}({#2})}

\newcommand{\calXzero}[1]{\calX_0({#1})}
\newcommand{\calYzero}[1]{\Upsilon_0({#1})}
\newcommand{\calXinf}[1]{\calX_{\infty}({#1})}
\newcommand{\calYinf}[1]{\Upsilon_{\infty}({#1})}
\newcommand{\Corrector}{\widetilde{E}}

\newcommand{\eps}{\varepsilon}
\newcommand{\epsh}{{\varepsilon_h}}

\newcommand{\strong}{\to}
\newcommand{\weak}{\,\xrightharpoonup{}\,}
\newcommand{\weakstar}{\xrightharpoonup{*}}
\newcommand{\strongtwoscale}{\xrightarrow{\,2\,}}
\newcommand{\weaktwoscale}{\xrightharpoonup{2}}

\newcommand{\weakstartwoscale}{\xrightharpoonup{2-*}}

\newcommand{\ext}[1]{\widetilde{#1}}
\newcommand{\genprod}{\stackrel{\text{gen.}}{\otimes}}
\newcommand{\mres}{\lfloor}
\newcommand{\charfun}[1]{\mathbbm{1}_{#1}}

\newcommand{\proj}{proj}
\newcommand{\closure}[1]{\overline{#1}}
\newcommand{\averageI}[1]{\overline{#1}}
\newcommand{\meanI}[1]{\widehat{#1}}
\newcommand{\tangential}{_\nu^\perp}
\newcommand{\tangentiali}{_{\nu^i}^\perp}
\newcommand{\tangentiallll}{_{\iota(\nu^i)}^\perp}
\newcommand{\normal}{_\nu}
\newcommand{\expfun}[1]{\exp\left({#1}\right)}

\usepackage[
maxbibnames=99,
    backend=biber,
    citestyle=numeric, doi=false,isbn=false,url=false,eprint=false, giveninits=true,
]{biblatex}
\title [Periodic homogenization of elastoplastic plates]
{Effective quasistatic evolution models for perfectly plastic plates with periodic microstructure\III: the limiting regimes.}
\author[M. Bu\v{z}an\v{c}i\'{c}] {Marin Bu\v{z}an\v{c}i\'{c}} 
\address[M. Bu\v{z}an\v{c}i\'{c}]{Faculty of Chemical Engineering and Technology, University of Zagreb, Trg Marka Maruli\'{c}a 19,
10000 Zagreb, Croatia}
\email{buzancic@fkit.hr}
\author[E. Davoli] {Elisa Davoli} 
\address[E. Davoli]{Institute of Analysis and Scientific Computing, TU Wien,
Wiedner Hauptstrasse 8-10, A-1040 Vienna, Austria}
\email{elisa.davoli@tuwien.ac.at}
\author[I. Vel\v{c}i\'{c}] {Igor Vel\v{c}i\'{c}} 
\address[I. Vel\v{c}i\'{c}]{Faculty of Electrical Engineering and Computing, University of Zagreb, Unska 3, 10000 Zagreb, Croatia}
\email{igor.velcic@fer.hr}

\begin{document}
\maketitle
\vspace{-\baselineskip}
\begin{abstract}
We identify effective models for thin, linearly elastic and perfectly plastic plates exhibiting a microstructure resulting from the periodic alternation of two elastoplastic phases.  We study here both the case in which the thickness of the plate converges to zero on a much faster scale than the periodicity parameter and the opposite scenario in which homogenization occurs on a much finer scale than dimension reduction. After performing a static analysis of the problem, we show convergence of the corresponding quasistatic evolutions. The methodology relies on two-scale convergence and periodic unfolding, combined with suitable measure-disintegration results and evolutionary $\Gamma$-convergence. 
\end{abstract}
\section{Introduction}
The main goal of this paper is to complete the study of limiting models stemming from the interplay of homogenization and dimension reduction in perfect plasticity which we have initiated in \cite{BDV}\EEE, as well as to show how the stress-strain approach introduced in \cite{Francfort.Giacomini.2014} for the homogenization of elasto-perfect plasticity can be used to identify effective theories for composite plates. \BBB In our previous contribution, we considered a composite thin plate whose thickness $h$ and microstructure-width $\ep_h$ were asymptotically comparable, namely, we assumed  
$$\lim_{h\to 0}\frac{h}{\ep_h}=:\gamma\in (0,+\infty).$$ 
In this work, instead, we analyze the two limiting regimes corresponding to the settings $\gamma=0$ and $\gamma=+\infty$. These can be seen, roughly speaking, as situations in which homogenization and dimension reduction happen on different scales, so that the behavior of the composite plate should ideally approach either that obtained via homogenization of the lower-dimensional model or the opposite one in which dimension reduction is performed on the homogenized material. 

To the Authors knowledge, apart from \cite{BDV} there has been no other study of simultaneous homogenization and dimension reduction for inelastic materials. In the purview of elasticity, we single out the works  \cite{cailleriethin, damlamianvogelius} (see also the book \cite{panasenkobook}) where first results were obtained in the case of linearized elasticity and under isotropy or additional material symmetry assumptions, as well as  \cite{bukal2017simultaneous} for the study of the general case without further constitutive restrictions and for an extension to some nonlinear models. 
 A $\Gamma$-convergence analysis in the nonlinear case has been provided in \cite{cherdantsev2015,neukammvelcic2013,hornung2014derivation,bufford2015multiscale,velcic2015}, whereas the case of high-contrast elastic plates is the subject of 
 \cite{BCVZ}. 

We shortly review below the literature on dimension reduction in plasticity and that on the study of composite elastoplastic materials. Reduced models for homogeneous perfectly plastic plates have been characterized in \cite{Davoli.Mora.2013,Davoli.Mora2015,Maggiani.Mora2016, Gidoni.Maggiani.Scala2018} in the quasistatic and dynamic settings, respectively, whereas the case of shallow shells is the focus of \cite{Maggiani.Mora2017}. In the presence of hardening, an analogous study has been undertaken in \cite{liero2011evolutionary,Liero.Roche2012}. 
Further results in finite plasticity are the subject of \cite{DavoliCOCV,DavoliM3AS}.

Homogenization of the elastoplastic equations in the small strain regime has been studied in \cite{Schweizer.Veneroni2015, Heida.Schweizer2016, Heida.Schweizer2018}.
We also refer to \cite{Francfort.Giacomini.Musesti,Giacomini.Musesti} for a study of the Fleck and Willis model, and to \cite{Hanke} for the case of gradient plasticity.
 Static and partial evolutionary results for large-strain
stratified composites in crystal plasticity have been obtained in  \cite{Cristowiak.Kreisbeck, Cristowiak.Kreisbeck2, Davoli.Ferreira.Kreisbeck, Davoli.Kreisbeck}, whereas static results in finite plasticity are the subject of   \cite{Davoli.Gavioli.Pagliari, Davoli.Gavioli.Pagliari2}.  Inhomogeneous perfectly plastic materials have been fully characterized in \cite{Francfort.Giacomini.2012}, an associated study of periodic homogenization is the focus of \cite{Francfort.Giacomini.2014}.

\EEE  
The main result of the paper, Theorem \ref{main result} is rooted in the theory of evolutionary \III$\Gamma$\BBB-convergence (see \cite{mielke.roubicek.stefanelli}) and consists in showing that rescaled three-dimensional quasistatic evolutions associated to the original composite plates converge, as the thickness and periodicity simultaneously go to zero, to the quasistatic evolution corresponding to suitable reduced effective elastic energies (identified by static \III$\Gamma$\BBB-convergence) and dissipation potentials, cf. Subsection \ref{principleofmximumpl}. As one might expect, for $\gamma=0$ the limiting driving energy and dissipation potential are homogenized versions of those identified in \cite{Davoli.Mora.2013} where only dimension reduction was considered. In the $\gamma=\infty$ setting, instead, the key functionals are obtained by averaging the original ones in the periodicity cell.

Essential ingredients to identify the limiting models are to establish a characterization of two-scale limits of rescaled linearized strains, as well as to prove variants of the principle of maximal work in each of the two regimes. These are the content of Theorem \ref{two-scale weak limit of scaled strains}, as well as \Cref{two-scale Hill's principle - regime zero} for the case $\gamma=0$, and of  \Cref{two-scale Hill's principle - regime inf} for $\gamma=+\infty$, respectively. A very delicate point consists in the identification of the limiting space of  elastoplastic  variables, for a fine characterization of the correctors arising in the two-scale limit passage needs to be established by delicate measure-theoretic disintegration arguments, cf. Section \ref{compactness}.

We finally mention that, for the regimes analyzed in this contribution, we obtain more restrictive results than in \cite{BDV}, for an additional assumption on the ordering of the phases on the interface, cf. \Cref{secphasedec} needs to be imposed in order to ensure lower semicontinuity of the dissipation potential, cf. \Cref{Reshetnyak remark 1}. 
\BBB

We briefly outline the structure of the paper. 
In \Cref{prel} we introduce our notation and recall some preliminary results. \Cref{setting} is devoted to the mathematical formulation of the problem, whereas  \Cref{compactness} tackles compactness properties of sequences with equibounded energy and dissipation. In \Cref{statics} we characterize the limiting model, we introduce the set of limiting deformations and stresses, and we discuss duality between stress and strain. Eventually, in \Cref{dynamics} we prove the main result of the paper, i.e., \Cref{main result}, where we show convergence of the quasistatic evolution of $3d$ composite thin plates to the quasistatic evolution associated to the limiting model. Similarly as in \cite{Francfort.Giacomini.2014,BDV}, in the limiting model a decoupling of macroscopic and microscopic variables is not possible and both scales contribute to the description of the limiting evolution. 
\BBB

 .

\section{Notation and preliminary results}
\label{prel}
Points $x \in \R^3$ will be expressed as pairs $(x',x_3)$, with $x' \in \R^2$ and $x_3 \in \R$, whereas we will write $y \in \calY$ to identify points on a flat 2-dimensional torus.
 We will denote by $I$ the open interval $I := \left(-\frac{1}{2}, \frac{1}{2}\right)$. 
The notation $\nabla_{x'}$ will describe the gradient with respect to $x'$.
Scaled gradients and symmetrized scaled gradients will be similarly denoted as follows:
\begin{gather} \label{defsymmscgrad}
    \nabla_h v := \Big[\; \nabla_{x'} v \;\Big|\Big.\; \tfrac{1}{h}\,\partial_{x_3} v \;\Big], \quad E_h v := \sym \nabla_h v.
\end{gather}
for $h>0$, and for maps $v$ defined on suitable subsets of $\mathbb{R}^3$. For $N=2,3$, we use the notation $\M^{N \times N}$ to identify the set of real $N \times N$ matrices. We will always implicitly assume this set to be endowed with the classical Frobenius scalar product $A : B := \sum_{i,j}A_{ij}\,B_{ij}$ and the associated norm $|A| := \sqrt{A:A}$, for $A,B\in \mathbb{M}^{N\times N}$. The subspaces of symmetric and deviatoric matrices, will be denoted by $\M^{N \times N}_{\sym}$ and $\M^{N \times N}_{\dev}$, respectively. For the trace and deviatoric part of a matrix $A\in \mathbb{M}^{N\times N}$ we will adopt the notation ${\rm tr}{A}$, and 
\[
    A_{\dev} = A - \frac{1}{N}{\rm tr}{A}.
\]

Given two vectors $a, b \in \R^N$, we will adopt standard notation for their scalar product and Euclidean norm, namely $a \cdot b$ and $|a|$. The dyadic (or tensor) product of $a$ and $b$ will be identified as 
  by $a \otimes b$; correspondingly,  the {\em symmetrized tensor product} $a \odot b$  will be the symmetric matrix with entries $(a \odot b)_{ij} := \frac{a_i b_j + a_j b_i}{2}$. 
We recall that ${\rm tr}{\left(a \odot b\right)} = a \cdot b$, and $|a \odot b|^2 = \frac{1}{2}|a|^2|b|^2 + \frac{1}{2}(a \cdot b)^2$, so that
\begin{equation*}
    \frac{1}{\sqrt{2}}|a||b| \leq |a \odot b| \leq |a||b|.
\end{equation*}
Given a vector $v \in \R^3$, we will use the notation $v^{\prime}$ to denote the two-dimensional vector having its same first two components
\begin{equation*}
    v^{\prime} := \begin{pmatrix} v_1 \\ v_2 \end{pmatrix}.
\end{equation*}
In the same way, for every $A \in \M^{3 \times 3}$, we will use the notation $A^{\prime\prime}$ to identify the minor
\begin{equation*}
    A^{\prime\prime} := \begin{pmatrix} A_{11} & A_{12} \\ A_{21} & A_{22} \end{pmatrix}.
\end{equation*}
The natural embedding of $\R^2$ into $\R^3$ will be given by $\iota:\R^2 \to \R^3$ 
 defined as 
\begin{equation*}
    \iota(v) := \begin{pmatrix} v_1 \\ v_2 \\0  \end{pmatrix}.
\end{equation*}
We will adopt standard notation for the Lebesgue and Hausdorff measure, as well as for Lebesgue and Sobolev spaces, and for spaces of continuously differentiable functions.
  Given a set $U \subset \R^N$, we will denote its closure by $\closure{U}$ and its 
 characteristic function by $\charfun{U}$.

We will distinguish between the spaces $C_c^k(\U;E)$ ($C^k$ functions with compact support contained in $\U$) and $C_0^k(\U;E)$ ($C^k$ functions ``vanishing on $\partial{\U}$"). 
The notation $C(\calY;E)$ will indicate the space of all continuous functions which are $[0,1]^2$-periodic. Analogously, we will define $C^k(\calY;E) := C^k(\R^2;E) \cap C(\calY;E)$. With a slight abuse of notation,
 $C^k(\calY;E)$ will be identified with the space of all $C^k$ functions on the 2-dimensional torus. 

We will frequently make use of the \emph{standard mollifier} $\rho \in C^{\infty}(\R^N)$, defined by
\begin{equation*}
    \rho(x)
    :=
    \begin{cases}
    C\,\expfun{\frac{1}{|x|^2-1}} & \ \text{ if } |x| < 1,\\
    0 & \ \text{ otherwise},
    \end{cases}
\end{equation*}
where the constant $C > 0$ is selected so that $\int_{\R^N} \rho(x) \,dx = 1$, as well as of the associated family $\{\rho_\epsilon\}_{\epsilon>0} \subset C^{\infty}(\R^N)$ with
\begin{equation*}
    \rho_\epsilon(x) := \frac{1}{\epsilon^N} \rho\left(\frac{x}{\epsilon}\right).
\end{equation*}

Throughout the text, the letter $C$ stands for generic positive constants whose value may vary from line to line.

 A collection of all preliminary results which will be used throughout the paper can be found in
 \cite[Section 2]{BDV}. For an overview on basic notions in measure theory, $BV$ functions, as well as $BD$ and $BH$ maps, we refer the reader to, e.g., \cite{fonseca2007modern}, \cite{ambrosio2000functions}, \cite{braides1998approximation}, to the monograph \cite{Temam.1985}, as well as to \cite{Demengel.1984}.

\subsection{Convex functions of measures} \label{Convex functions of measures}
Let $U$ be an open set of $\R^N$. 
For every $\mu \in \Mb(U;X)$ let $\frac{d\mu}{d|\mu|}$ be the Radon-Nikodym derivative of $\mu$ with respect to its variation $|\mu|$. 
Let $H : X \to [0,+\infty)$ be a convex and positively one-homogeneous function such that
\begin{equation} \label{coercivity of H}
    r |\xi| \leq H(\xi) \leq R |\xi| \quad \text{for every}\, \xi \in X,
\end{equation}
where $r$ and $R$ are two constants, with $0 < r \leq R$. 

Using the theory of convex functions of measures (see \cite{goffman1964sublinear} and \cite{demengel1984convex}) it is possible to define a nonnegative Radon measure $H(\mu) \in \Mb^+(U)$ as
\[
    H(\mu)(A) := \int_{A} H\left(\frac{d\mu}{d|\mu|}\right) \,d|\mu|,
\]
for every Borel set $A \subset U$, as well as an associated
 functional $\calH: \Mb(U;X) \to [0,+\infty)$ given by
\[
    \calH(\mu) := H(\mu)(\U) = \int_{\U} H\left(\frac{d\mu}{d|\mu|}\right) \,d|\mu|.
\]
and being lower semicontinuous on $\Mb(U;X)$ with respect to weak* convergence, cf. \cite[Theorem 2.38]{ambrosio2000functions}).

Let $a,\, b \in [0,T]$ with $a \leq b$. 
The \emph{total variation} of a function $\mu : [0,T] \to \Mb(U;X)$ on $[a,b]$ is defined as
\begin{equation*}
    \calV(\mu; a, b) := \sup\left\{ \sum_{i = 1}^{n-1} \left\|\mu(t_{i+1}) - \mu(t_i)\right\|_{\Mb(U;X)} : a = t_1 < t_2 < \ldots < t_n = b,\ n \in \N \right\}.
\end{equation*}
Analogously,  the \emph{$\calH$-variation} of a function $\mu : [0,T] \to \Mb(U;X)$ on $[a,b]$ is given by
\begin{equation*}
    \calD_{\calH}(\mu; a, b) := \sup\left\{ \sum_{i = 1}^{n-1} \calH\left(\mu(t_{i+1}) - \mu(t_i)\right) : a = t_1 < t_2 < \ldots < t_n = b,\ n \in \N \right\}.
\end{equation*}
From \eqref{coercivity of H} it follows that
\begin{equation} \label{equivalence of variations}
    r \calV(\mu; a, b) \leq \calD_{\calH}(\mu; a, b) \leq R \calV(\mu; a, b).
\end{equation}

\subsection{Disintegration of a measure}
Let $S$ and $T$ be measurable spaces and let $\mu$ be a measure on $S$. 
Given a measurable function $f : S \to T$, we denote by $f_{\#}\mu$ the \emph{push-forward} of $\mu$ under the map $f$, defined by
\[
    f_{\#}\mu(B) := \mu\left(f^{-1}(B)\right), \quad \text{ for every measurable set $B \subseteq T$}.
\]
In particular, for any measurable function $g : T \to \overline{\R}$ we have
\[
    \int_{S} g \circ f \,d\mu = \int_{T} g \,d(f_{\#}\mu).
\]
Note that in the previous formula $S = f^{-1}(T)$. 

Let \BBB $S_1 \subset \R^{N_1}$, $S_2 \subset \R^{N_2}$, \BBB for some $N_1,N_2 \in \N$,  \BBB be open sets, and let $\eta \in \Mb^+(S_1)$. 
We say that a function $x_1 \in S_1 \to \mu_{x_1} \in \Mb(S_2;\BBB \R^M \BBB)$ is $\eta$-measurable if $x_1 \in S_1 \to \mu_{x_1}(B)$ is $\eta$-measurable for every Borel set $B \subseteq S_2$.

Given a $\eta$-measurable function $x_1 \to \mu_{x_1}$ such that $\int_{S_1}|\mu_{x_1}|\,d\eta<+\infty$, then the \emph{generalized product} $\eta \genprod \mu_{x_1}$ satisfies $\BBB \eta \genprod \mu_{x_1}\BBB \in \Mb(S_1 \times S_2;\BBB \R^M \BBB)$ and is such that 
\[
	\langle \eta \genprod \mu_{x_1}, \varphi \rangle := \int_{S_1} \left( \int_{S_2} \varphi(x_1,x_2) \,d\mu_{x_1}(x_2) \right) \,d\eta(x_1),
\]
for every bounded Borel function $\varphi : S_1 \times S_2 \to \R$.

\subsection{Traces of stress tensors}
\label{sub:traces}
In this last subsection we collect some properties of classes of maps which will include our elastoplastic stress tensors. 

We suppose here that $\U$ is an open bounded set of class $C^2$ in $\R^N$.
If $\sigma \in L^2(\U;\M^{N \times N}_{\sym})$ and $\div\sigma \in L^2(\U;\R^N)$, then we can define a distribution $[ \sigma \nu ]$ on $\partial{\U}$ by
\begin{equation} \label{traces of the stress}
    [ \sigma \nu ](\psi) := \int_{\U} \psi \cdot \div\sigma \,dx + \int_{\U} \sigma : E\psi \,dx,
\end{equation}
for every $\psi \in H^1(\U;\R^N)$. 
It \BBB follows \BBB that $[ \sigma \nu ] \in H^{-1/2}(\partial{\U};\R^N)$ (see, e.g., \cite[Chapter 1, Theorem 1.2]{temam2001navier}). 
If, in addition, $\sigma \in L^{\infty}(\U;\M^{N \times N}_{\sym})$ and $\div\sigma \in L^N(\U;\R^N)$, then \eqref{traces of the stress} holds for $\psi \in W^{1,1}(\U;\R^N)$. 
By Gagliardo’s extension theorem, in this case we have $[ \sigma \nu ] \in L^{\infty}(\partial{\U};\R^N)$, and 
\begin{equation*}
    [ \sigma_k \nu ] \weakstar [ \sigma \nu ] \quad \text{weakly* in $L^{\infty}(\partial{\U};\R^N)$},
\end{equation*}
whenever $\sigma_k \weakstar \sigma$ weakly* in $L^{\infty}(\U;\M^{N \times N}_{\sym})$ and $\div\sigma_k \weak \div\sigma$ weakly in $L^N(\U;\R^N)$.

We will consider the normal and tangential parts of $[ \sigma \nu ]$, defined by
\begin{equation*}
    [ \sigma \nu ]\normal := ([ \sigma \nu ] \cdot \nu) \nu, \quad
    [ \sigma \nu ]\tangential := [ \sigma \nu ]-([ \sigma \nu ] \cdot \nu) \nu.
\end{equation*}
Since $\nu \in C^1(\partial{\U};\R^N)$, we have that $[ \sigma \nu ]\normal,\, [ \sigma \nu ]\tangential \in H^{-1/2}(\partial{\U};\R^N)$. If, in addition, $\sigma_{\dev} \in L^{\infty}(\U;\M^{N \times N}_{\dev})$, then it was proved in \cite[Lemma 2.4]{Kohn.Temam.1983} that $[ \sigma \nu ]\tangential \in L^{\infty}(\partial{\U};\R^N)$ and
\begin{equation*}
    \|[ \sigma \nu ]\tangential\|_{L^{\infty}(\partial{\U};\R^N)} \leq \frac{1}{\sqrt2} \|\sigma_{\dev}\|_{L^{\infty}(\U;\M^{N \times N}_{\dev})}.
\end{equation*}

More generally, if $\U$ has Lipschitz boundary and is such that there exists a compact set $S \subset \partial{U}$ with $\calH^{N-1}(S) = 0$ such that $\partial{U} \setminus S$ is a $C^2$-hypersurface, then arguing as in \cite[Section 1.2]{Francfort.Giacomini.2012} we can uniquely determine $[ \sigma \nu ]\tangential$ as an element of $L^{\infty}(\partial{U};\R^N)$ through any approximating sequence $\{\sigma_n\} \subset C^{\infty}(\closure{\U};\M^{N \times N}_{\sym})$ such that 
\begin{align*}
    &\sigma_n \strong \sigma \quad \text{strongly in } L^2(\U;\M^{N \times N}_{\sym}),\\
    &\div\sigma_n \strong \div\sigma \quad \text{strongly in } L^2(\U;\R^N),\\
    &\| (\sigma_n)_{\dev} \|_{L^{\infty}(\U;\M^{N \times N}_{\dev})} \leq \| \sigma_{\dev} \|_{L^{\infty}(\U;\M^{N \times N}_{\dev})}.
\end{align*}



\section{Setting of the problem}
\label{setting}
\BBB We describe here our modeling assumptions and recall a few associated instrumental results. \BBB
\BBB Unless otherwise stated,  $\omega \subset \R^2$ is a bounded, connected, and open set with $C^2$   boundary. \BBB
Given a small positive number $h > 0$, we assume 
\begin{equation*}
    \Omega^h := \omega \times (h I),
\end{equation*}
to be the reference configuration of a linearly elastic and perfectly plastic plate.

We consider a non-zero Dirichlet boundary condition on the whole lateral surface, i.e. the Dirichlet boundary of $\Omega^h$ is given by \CCC $\Gamma_\Dir^h := \EEE \partial\omega \BBB \times (h I)$. \BBB

We work under the assumption that the body is only submitted to a hard device on $\Gamma_\Dir^h$ and that there are no applied loads, i.e. the evolution is only driven by time-dependent boundary conditions. 
More general boundary conditions, together with volume and surfaces forces have been considered, e.g., in \cite{DalMaso.DeSimone.Mora.2006, Francfort.Giacomini.2012, Davoli.Mora.2013} but  for simplicity of exposition will be neglected in this analysis.

\subsection{Phase decomposition}\label{secphasedec} 
We recall here some basic notation and assumptions from \cite{Francfort.Giacomini.2014}. 

\BBB Recall that  $\calY = \R^2/\Z^2$ is \BBB the $2$-dimensional torus, let $Y := [0, 1)^2$ be its associated periodicity cell, and denote by $\calI : \calY \to Y$ their canonical identification.
For any $\calZ \subset \calY$, we define
\begin{equation} \label{periodic set notation}
    \calZ_\eps := \left\{ x \in \R^2 : \frac{x}{\eps} \in \Z^2+\calI(\calZ) \right\},
\end{equation}
and to every function $F : \calY \to X$ we associate the $\eps$-periodic function $F_\eps : \R^2 \to X$, given by
\begin{equation*}
    F_\eps(x) := F\left(y_\eps\right), \;\text{ for }\; \frac{x}{\eps}-\left\lfloor \frac{x}{\eps} \right\rfloor = \calI(y_\eps) \in Y.
\end{equation*}
With a slight abuse of notation we will also write $F_\eps(x) = F\left(\frac{x}{\eps}\right)$.

The torus $\calY$ is assumed to be made up of finitely many phases $\calY_i$ together with their interfaces. 
We assume that those phases are pairwise disjoint open sets with Lipschitz boundary.
Then we have $\calY = \bigcup_{i} \closure{\calY}_i$ and we denote the interfaces by 
\begin{equation*}
    \Gamma := \bigcup_{i,j} \partial\calY_i \cap \partial\calY_j.
\end{equation*}

We will write
\begin{equation*}
    \Gamma := \bigcup_{i \neq j} \Gamma_{ij},
\end{equation*}
where $\Gamma_{ij}$ stands for the interface between $\calY_i$ and $\calY_j$.

\EEE Correspondingly, \BBB $\omega$ is composed of finitely many phases $(\calY_i)_{\eps}$ \CCC and that $\eps$ is chosen \EEE small enough so \BBB that  $\mathcal{H}^1\left(\cup_i(\partial \calY_i)_{\eps} \cap \EEE\partial\omega \BBB\right)=0$.  \BBB
Additionally, we assume that $\Omega^h$ is a specimen of \EEE a linearly elastic - \BBB perfectly plastic material having periodic elasticity tensor and dissipation potential.

We are interested in the situation when the period $\eps$ is a function of the thickness $h$, i.e. $\eps = \epsh$, and we assume that the limit 
\begin{equation*}
   \gamma := \lim_{h \to 0} \frac{h}{\epsh}.
\end{equation*}
exists in \RED $\{0, +\infty\}$ \BLACK. 
\III We additionally impose the following condition: 
  there exists a compact set  $\mathcal{S} \subset \Gamma$  with  $\calH^1(\mathcal{S}) = 0$  such that
     each connected component of  $\Gamma \setminus \mathcal{S}$  is either a closed curve of class $C^2$ or an open curve with endpoints  $\{a,b\}$ which is of class $C^2$  (excluding the endpoints).
\BBB

We say that a multi-phase torus $\calY$ is \emph{geometrically admissible} if it satisfies the above assumptions.

\CCC 
\begin{remark} \label{remnak1} 
Notice that under the above assumptions, $\mathcal{H}^1$-almost every $y \in \Gamma$ is at the intersection of the boundaries of exactly two phases. 
\end{remark}
\BBB

\begin{remark}
We point out that \CCC we \BBB assume greater regularity than that in \cite{Francfort.Giacomini.2014}, where the interface $\Gamma \setminus \calS$ was allowed to be a $C^1$-hypersurface. 
Under such weaker assumptions, in fact, the tangential part of the trace of an admissible stress $[ \sigma \nu ]\tangential$ at a point $x$ on $\Gamma \setminus \calS$ would not be defined independently of the considered approximating sequence\EEE, cf. Subsection \ref{sub:traces}. \BBB By requiring a higher regularity of $\Gamma\setminus \calS$, we will avoid dealing with this \CCC situation. 

\end{remark}

\medskip
\noindent{\bf The set of admissible stresses.}

We assume that there exist convex compact sets $K_i \in \M^{3 \times 3}_{\dev}$ associated to each phase $\calY_i$ \CCC which will provide restrictions on the deviatoric part of the stress\BBB.
We work under the assumption that there exist two constants $r_K$ and $R_K$, with $0 < r_K \leq R_K$, such that for every $i$
\begin{equation*}
    \{ \xi \in \M^{3 \times 3}_{\sym} : |\xi| \leq r_K \} \subseteq K_i \subseteq\{ \xi \in \M^{3 \times 3}_{\sym}: |\xi| \leq R_K \}.
\end{equation*}
Finally, we define
\begin{equation*}
    K(y) := K_i, \quad \text{ for } y \in \calY_i.
\end{equation*}

\EEE We will require an ordering between the phases at the interface. 
Namely, we assume that 
\III at the point $y \in \Gamma$  where exactly two phases $\calY_i$ and $\calY_j$ meet we have that either $K_i \subset K_j$ or $K_j \subset K_i$. \BBB

\CCC
We will call this \EEE requirement \BBB the assumption on the ordering of the phases. \BBB
\begin{remark} \label{Reshetnyak remark 1}
The restrictive assumption \CCC on the ordering between the phases \BBB will allow us to use Reshetnyak's lower semicontinuity theorem to obtain lower semicontinuity of the dissipation functional, \CCC cf. the proof of \Cref{main result}\BBB. \CCC Notice that in the regime $\gamma \in (0,+\infty)$, see \cite{BDV}, we did not \EEE rely on \BBB such assumption (see also \cite{Francfort.Giacomini.2012,Francfort.Giacomini.2014}) and thus were able to prove the convergence to the limit model in the general case. In the regimes $\gamma \in \{0,\infty\}$ \EEE the general geometrical setting where no ordering between the phases is assumed \BBB remains an open problem.  \BBB
\end{remark}

\noindent{\bf The elasticity tensor.}

For every $i$, let $(\C_{\rm dev})_i$ and $k_i$ be a symmetric positive definite tensor on $\M^{3\times 3}_{\rm dev}$ and a positive constant, respectively, such that there exist two constants $r_c$ and $R_c$, with $0 < r_c \leq R_c$, satisfying
\begin{align} \label{tensorassumption}
    &r_c |\xi|^2 \leq (\C_{\dev})_i \xi : \xi \leq R_c |\xi|^2 \quad \text{ for every }\xi \in \M^{3 \times 3}_{\dev},\\ \label{tensorassumption2}
    &r_c \leq k_i \leq R_c.
\end{align}

Let $\C$ be the {\em elasticity tensor}, considered as a map from $\calY$ taking values in the set of symmetric positive definite linear operators, $\C : \calY \times \M^{3 \times 3}_{\sym} \to \M^{3 \times 3}_{\sym}$, defined as
\begin{equation*}
    \C(y) \xi := \C_{\dev}(y)\,\xi_{\dev} + \left(k(y)\,{\rm tr}{\xi}\right)\,I_{3 \times 3} \quad \text{ for every } y \in \calY \text{ and } \xi \in \M^{3 \times 3},
\end{equation*}
where $\C_{\dev}(y)=(\C_{\dev})_i$ and $k(y) = k_i$ for every $y \in \calY_i$.

Let $Q:\calY \times \M^{3 \times 3}_{\sym}\to[0,+\infty)$ be the quadratic form associated with $\C$, and given by
\begin{equation*}
    Q(y, \xi) := \frac{1}{2} \C(y) \xi : \xi \quad \text{ for every } y \in \calY \text{ and } \xi \in\M^{3 \times 3}_{\sym}.
\end{equation*}
It follows that $Q$ satisfies
\begin{equation} \label{coercivity of Q}
    r_c |\xi|^2 \leq Q(y, \xi) \leq R_c |\xi|^2 \quad \text{for every }y\in \calY\text{ and }\xi \in \M^{3 \times 3}_{\sym}.
\end{equation}

\medskip
\noindent{\bf The dissipation potential.}

For each $i$, let $H_i : \M^{3 \times 3}_{\dev} \to [0,+\infty)$ be the support function of the set $K_i$, i.e
\begin{equation*}
    H_i(\xi) = \sup\limits_{\tau \in K_i} \tau : \xi.
\end{equation*}
It follows that $H_i$ is convex, positively 1-homogeneous, and satisfies
\begin{equation} \label{coercivity of H_i}
    r_k |\xi| \leq H_i(\xi) \leq R_k |\xi| \quad \text{for every}\, \xi \in \M^{3 \times 3}_{\dev}.
\end{equation}

The dissipation potential $H : \calY \times \M^{3 \times 3}_{\dev} \to [0,+\infty]$ is defined as follows:
\begin{enumerate}[label=(\roman*)]
    \item For every $y \in \calY_i$, 
    \begin{equation*}
        H(y, \xi) := H_i(\xi).
    \end{equation*}
    \item 
   \III
   For a point $y \in \Gamma$ that is at interface of exactly two phases $\calY_i$ and $\calY_j$ we define 
    $$ H(y,\xi)=\min_{i,j} \{H_i(y,\xi),H_j (y,\xi)\}. $$
    \item For all other points we take 
      $$ H(y,\xi)=\min_{i} H_i (y,\xi). $$
      \end{enumerate} 
      \BBB
\begin{remark} \label{Reshetnyak remark 2}
\CCC We point out that $H$ is a Borel, lower semicontinuous function on $\calY \times \M^{3 \times 3}_{\dev}$. 
Furthermore, for each $y \in \calY$, the function $\xi \mapsto H(y, \xi)$ is positively 1-homogeneous and convex. \BBB
\end{remark}

\medskip
\noindent{\bf Admissible triples and energy.}

On $\Gamma_\Dir^h$ we prescribe a boundary datum being the trace of a map $w^h \in H^1(\Omega^h;\R^3)$ with the following \CCC Kirchhoff-Love \BBB structure:
\begin{equation} \label{MGform0}
    w^h(z) := \left( \bar{w}_1(z') - \frac{z_3}{h}\partial_1 \bar{w}_3(z'),\, \bar{w}_2(z') - \frac{z_3}{h}\partial_2 \bar{w}_3(z'),\, \frac{1}{h}\bar{w}_3(z') \right) \,\text{ for a.e. }z=(z',z_3) \in \Omega^h,
\end{equation}
where $\bar{w}_\alpha \in H^1(\omega)$, $\alpha=1,2$, and $\bar{w}_3 \in H^2(\omega)$.
The {\em set of admissible displacements and strains} for the boundary datum $w^h$ is denoted by 
$\calA(\Omega^h, w^h)$ and is defined as the class of all triples 
$(v,f,q) \in BD(\Omega^h) \times L^2(\Omega^h;\M^{3 \times 3}_{\sym}) \times \Mb(\Omega^h;\M^{3 \times 3}_{\dev})$ satisfying
\begin{eqnarray*}
& Ev=f+q \quad\text{in }\Omega^h,
\\
& q=(w^h-v)\odot\nu_{\partial\Omega^h}\calH^2 \quad\text{on }\Gamma_\Dir^h.
\end{eqnarray*} 
The function $v$ represents the {\em displacement} of the plate, while $f$ and $q$ are called the {\em elastic} and {\em plastic strain}, respectively.

For every admissible triple $(v,f,q) \in \calA(\Omega^h, w^h)$ we define the associated {\em energy} as
\begin{equation*}
    \calE_{h}(v,f,q) := 
    \int_{\Omega^h} Q\left(\frac{z'}{\epsh}, f(z)\right) \,dz + 
    \int_{\Omega^h \cup \Gamma_\Dir^h} H\left(\frac{z'}{\epsh}, \frac{dq}{d|q|} \right) \,d|q|.
\end{equation*}
The first term represents the elastic energy, while the second term accounts for plastic dissipation.

\subsection{The rescaled problem} \label{rescaled}
As usual in dimension reduction problems, it is convenient to perform a change of variables in such a way to rewrite the system on a fixed domain independent of $h$. 
To this purpose, 
we consider the open interval $I = \left(-\frac{1}{2}, \frac{1}{2}\right)$ and set
\begin{equation*}
    \Omega \,:=\, \omega \times I, \qquad 
    \Gamma_\Dir \,:=\, \partial\omega \times I. 
\end{equation*}
We consider the change of variables $\psi_h : \closure{\Omega} \to \closure{\Omega^h}$, defined as
\begin{equation} \label{eq:def-psih}
    \psi_h(x',x_3) := (x', hx_3) \quad \text{for every}\, (x',x_3) \in \closure{\Omega},
\end{equation}
and the linear operator $\Lambda_h : \M_{\sym}^{3 \times 3} \to \M_{\sym}^{3 \times 3}$ given by 
\begin{equation} \label{definition Lambda_h}
    \Lambda_h \xi:=\begin{pmatrix}
    \xi_{11} & \xi_{12} & \frac{1}{h}\xi_{13}
    \vspace{0.1 cm}\\
    \xi_{21} & \xi_{22} & \frac{1}{h}\xi_{23}
    \vspace{0.1 cm}\\
    \frac{1}{h}\xi_{31} & \frac{1}{h}\xi_{32} & \frac{1}{h^2}\xi_{33}
    \end{pmatrix}
    \quad\text{for every }\xi\in\M^{3 \times 3}_{\sym}.
\end{equation}
To any triple $(v,f,q) \in \calA(\Omega^h, w^h)$ we associate a triple $(u,e,p) \in BD(\Omega) \times L^2(\Omega;\M^{3 \times 3}_{\sym}) \times \allowbreak\Mb(\Omega \cup \Gamma_\Dir;\M^{3 \times 3}_{\sym})$ defined as follows:
\begin{equation*}
    u := (v_1, v_2, h v_3) \circ \psi_h, \qquad
    e := \Lambda_{h}^{-1}f\circ\psi_h, \qquad 
    p := \tfrac{1}{h}\Lambda_h^{-1}\psi_h^{\#}(q).
\end{equation*}
Here the measure $\psi_h^{\#}(q) \in \Mb(\Omega;\M^{3 \times 3})$ is the pull-back measure of $q$, satisfying 
\begin{equation*}
    \int_{\Omega \cup \Gamma_\Dir} \varphi : d \psi_h^{\#}(q) = \int_{\Omega^h \cup \Gamma_\Dir^h} (\varphi\circ\psi_h^{-1}) : dq \quad \text{ for every } \varphi \in C_0(\Omega \cup \Gamma_\Dir;\M^{3 \times 3}).
\end{equation*}
According to this change of variable we have
\begin{equation*}
    \calE_h(v,f,q) = h \calQ_h(\Lambda_h e) + h \calH_h(\Lambda_h p),
\end{equation*}
where
\begin{equation} \label{definition Q_h} 
    \calQ_h(\Lambda_h e) = \int_{\Omega} Q\left(\frac{x'}{\epsh}, \Lambda_h e\right) \,dx
\end{equation}
and
\begin{equation} \label{definition H_h} 
    \calH_h(\Lambda_h p) = \int_{\Omega \cup \Gamma_\Dir} H\left(\frac{x'}{\epsh}, \frac{d\Lambda_h p}{d|\Lambda_h p|}\right) \,d|\Lambda_h p|.
\end{equation}

We also introduce the scaled Dirichlet boundary datum $w \in H^1(\Omega;\R^3)$, given by
\begin{equation*}
    w(x):=(\bar{w}_1(x')-{x_3}\partial_1 w_3(x'), \bar{w}_2(x')-x_3\partial_2 w_3(x'), w_3(x'))\quad\text{for a.e.\ }x\in\Omega.
\end{equation*}
By the definition of the class $\calA(\Omega^h, w^h)$ it follows that the scaled triple $(u,e,p)$ satisfies 
\begin{eqnarray}
& Eu=e+p \quad\text{in }\Omega,
 \label{straindec*}
\\
& p=(w-u)\odot\nu_{\partial\Omega}\calH^2 \quad\text{on }\Gamma_\Dir,
 \label{boundcondp*}
\\
& p_{11}+p_{22}+\frac{1}{h^2}p_{33}=0\quad\text{in }\Omega \cup \Gamma_\Dir.
 \label{trace*}
\end{eqnarray} 
We are thus led to introduce the class $\calA_h(w)$ of all triples $(u,e,p) \in BD(\Omega) \times L^2(\Omega;\M^{3 \times 3}_{\sym}) \times \allowbreak\Mb(\Omega \cup \Gamma_\Dir;\M^{3 \times 3}_{\sym})$ satisfying \eqref{straindec*}--\eqref{trace*}, and to define the functional
\begin{equation} \label{jep}
    \calJ_{h}(u,e,p) := \calQ_h(\Lambda_h e)+\calH_h(\Lambda_h p)
\end{equation}
for every $(u,e,p) \in \mathcal A_h(w)$. 
In the following we will study the asymptotic behaviour \BBB of  the quasistatic evolution \BBB associated with $\calJ_{h}$, as $h \to 0$ and $\eps_h \to 0$.

\BBB Notice that if  $\bar{w}_\alpha \in H^1(\ext{\omega})$, $\alpha=1,2$, and $\bar{w}_3 \in H^2(\ext{\omega})$, where $\omega \subset \ext{\omega}$, then we can trivially extend the triple $(u,e,p)$ to $\ext{\Omega}:= \ext{\omega} \times I$ by 
$$u=w, \qquad e=Ew, \qquad p=0 \qquad \text{ on } \ext{\Omega} \setminus \closure{\Omega}. $$
In the following, \EEE with a slight abuse of notation, we will still denote this extension by $(u,e,p)$, whenever such an extension procedure will be needed. \BBB
\BBB 
\medskip

\noindent{\bf Kirchhoff-Love admissible triples and limit energy.} 

We consider the set of {\em Kirchhoff-Love displacements}, defined as
\begin{equation*}
    KL(\Omega):= \big\{u \in BD(\Omega) : \ (Eu)_{i3}=0 \quad\text{for } i=1,2,3\big\}.
\end{equation*}
We note that $u \in KL(\Omega)$ if and only if $u_3\in BH(\omega)$ and there exists $\bar{u}\in BD(\omega)$ such that 
\begin{equation} \label{ualfa}
    u_{\alpha}=\bar{u}_{\alpha}-x_3\partial_{x_\alpha}u_3, \quad \alpha=1,2.
\end{equation}
In particular, if $u \in KL(\Omega)$, then 
\begin{equation} \label{symmetric gradient of KL functions}
    Eu = \begin{pmatrix} \begin{matrix} E\bar{u} - x_3 D^2u_3 \end{matrix} & \begin{matrix} 0 \\ 0 \end{matrix} \\ \begin{matrix} 0 & 0 \end{matrix} & 0 \end{pmatrix}.
\end{equation}
If, in addition, $u \in W^{1,p}(\Omega;\R^3)$ for some $1 \leq p \leq \infty$, then $\bar{u}\in W^{1,p}(\omega;\R^2)$ and $u_3\in W^{2,p}(\omega)$. 
We call $\bar{u}, u_3$ the {\em Kirchhoff-Love components} of $u$. 

For every $w\in H^1(\Omega;\R^3) \cap KL(\Omega)$ we define the class $\calA_{KL}(w)$ of {\em Kirchhoff-Love admissible triples} for the boundary datum~$w$ as the set of all triples $(u,e,p) \in KL(\Omega) \times L^2(\Omega;\M^{3 \times 3}_{\sym}) \times \Mb(\Omega \cup \Gamma_\Dir;\M^{3 \times 3}_{\sym})$ satisfying
\begin{eqnarray}
& Eu=e+p \quad\text{in }\Omega, \qquad  
p=(w-u)\odot\nu_{\partial\Omega}\calH^2 \quad\text{on }\Gamma_\Dir, \label{AKL1}\\ 
& e_{i3}=0 \quad\text{in }\Omega, \quad p_{i3}=0 \quad\text{in }\Omega \cup \Gamma_\Dir,  \quad i=1,2,3.
 \label{AKL2}
\end{eqnarray}
Note that the space
\begin{equation*}
    \big\{\xi\in \M^{3 \times 3}_{\sym}: \ \xi_{i3}=0 \text{ for }i=1,2,3\big\}
\end{equation*}
is canonically isomorphic to $\M^{2 \times 2}_{\sym}$. Therefore, in the following, given a triple $(u,e,p) \in \calA_{KL}(w)$ we will usually identify $e$ with a function in $L^2(\Omega;\M^{2 \times 2}_{\sym})$ and $p$ with a measure in $\Mb(\Omega \cup \Gamma_\Dir;\M^{2 \times 2}_{\sym})$. Note also that the class $\calA_{KL}(w)$ is always nonempty as it contains the triple $(w,Ew,0)$. 

To provide a useful characterisation of admissible triplets in $\calA_{KL}(w)$, let us first recall the definition of \CCC zero-th \BBB and first order moments of functions. 

\begin{definition} \label{moments of functions}
For $f \in L^2(\Omega;\M^{2 \times 2}_{\sym})$ we denote by $\bar{f}$, $\hat{f} \in L^2(\omega;\M^{2 \times 2}_{\sym})$ and $f^\perp \in L^2(\Omega;\M^{2 \times 2}_{\sym})$ the following orthogonal components (with respect to the scalar product of $L^2(\Omega;\M^{2 \times 2}_{\sym})$) of $f$:
\begin{equation} \label{moments1} 
    \bar{f}(x') := \int_{-\frac{1}{2}}^{\frac{1}{2}} f(x',x_3)\, dx_3, \qquad
    \hat{f}(x') := 12 \int_{-\frac{1}{2}}^{\frac{1}{2}}x_3 f(x',x_3)\,dx_3
\end{equation}
for a.e.\ $x' \in \omega$, and
\begin{equation*}
    f^\perp(x) := f(x) - \bar{f}(x') - x_3 \hat{f}(x')    
\end{equation*}
for a.e.\ $x \in \Omega$.
We name $\bar{f}$ the \emph{zero-th order moment} of $f$ and $\hat{f}$ the \emph{first order moment} of $f$. \CCC More generally, we will also use the expressions \eqref{moments1} for any integrable function over $I$.  \BBB
\end{definition}

The coefficient in the definition of $\hat{f}$ is chosen from the computation $\int_{I} x_3^2 \,dx_3 = \frac{1}{12}$. 
It ensures that if $f$ is of the form $f(x) = x_3 g(x')$, for some $g \in L^2(\omega;\M^{2 \times 2}_{\sym})$, then $\hat{f} = g$.

Analogously, we have the following definition of zero-th and first order moments of measures. 

\begin{definition} \label{moments of measures}
For $\mu \in M_b(\Omega \cup \Gamma_\Dir;\M^{2 \times 2}_{\sym})$ we define $\bar{\mu}$, $\hat{\mu} \in M_b(\omega \cup \gamma_\Dir;\M^{2 \times 2}_{\sym})$ and $\mu^\perp \in M_b(\Omega \cup \Gamma_\Dir;\M^{2 \times 2}_{\sym})$ as follows:
\begin{equation*}
    \int_{\omega \cup \gamma_\Dir} \varphi: d\bar{\mu}:= \int_{\Omega \cup \Gamma_\Dir} \varphi: d\mu , \qquad
    \int_{\omega \cup \gamma_\Dir} \varphi: d\hat{\mu} := 12\int_{\Omega \cup \Gamma_\Dir} x_3 \varphi: d\mu 
\end{equation*}
for every $\varphi \in C_0(\omega \cup \gamma_\Dir;\M^{2 \times 2}_{\sym})$, and
\begin{equation*}
    \mu^\perp := \mu - \bar{\mu} \otimes \calL^{1}_{x_3} - \hat{\mu} \otimes x_3 \calL^{1}_{x_3},
\end{equation*}
where $\otimes$ is the usual product of measures, and $\calL^{1}_{x_3}$ is the Lebesgue measure restricted to the third component of $\R^3$. 
We call $\bar{\mu}$ the \emph{zero-th order moment} of $\mu$ and $\hat{\mu}$ the \emph{first order moment} of $\mu$.
\end{definition}

We are now ready to recall the following characterisation of $\calA_{KL}(w)$, given in \cite[Proposition 4.3]{Davoli.Mora.2013}.

\begin{proposition} \label{A_KL characherization}
Let $w \in H^1(\Omega;\R^3) \cap KL(\Omega)$ and let $(u,e,p) \in  {KL}(\Omega) \times L^2(\Omega;\M^{3 \times 3}_{\sym}) \times \Mb(\Omega \cup \Gamma_\Dir;\M^{3 \times 3}_{\dev})$. 
Then $(u,e,p) \in \calA_{KL}(w)$ if and only if the following three conditions are satisfied:
\begin{enumerate}[label=(\roman*)]
    \item $E\bar{u} = \bar{e}+\bar{p}$ in $\omega$ and $\bar{p} = (\bar{w}-\bar{u}) \odot \nu_{\partial\omega} \calH^1$ on $\gamma_\Dir$;
    \item $D^2u_3 = - (\hat{e}+\hat{p})$ in $\omega$, $u_3 = w_3$ on $\gamma_\Dir$, and $\hat{p} = (\nabla u_3-\nabla w_3) \odot \nu_{\partial\omega} \calH^1$ on $\gamma_\Dir$;
    \item $p^\perp = - e^\perp$ in $\Omega$ and $p^\perp = 0$ on $\Gamma_\Dir$.
\end{enumerate}
\end{proposition}

\subsection{The reduced problem} \label{reduced}

For a fixed $y \in \calY$, let $\A_y : \M^{2 \times 2}_{\sym} \to \M^{3 \times 3}_{\sym}$ be the operator given by
\[
    \A_y \xi
    :=
    \begin{pmatrix} \begin{matrix} \xi \end{matrix} & \begin{matrix} \lambda^y_1(\xi) \\ \lambda^y_2(\xi) \end{matrix} \\ \begin{matrix} \lambda^y_1(\xi) & \lambda^y_2(\xi) \end{matrix} & \lambda^y_3(\xi) \end{pmatrix}
    \quad \text{for every }\xi\in\M^{2 \times 2}_{\sym},
\]
where for every $\xi\in\M^{2 \times 2}_{\sym}$ the triple $(\lambda^y_1(\xi),\lambda^y_2(\xi),\lambda^y_3(\xi))$ is the unique solution to the minimum problem
\begin{equation} \label{Q_r minimum problem}
    \min_{\lambda_i \in \R} Q\left(y, \begin{pmatrix} \begin{matrix} \xi \end{matrix} & \begin{matrix} \lambda_1 \\ \lambda_2 \end{matrix} \\ \begin{matrix} \lambda_1 & \lambda_2 \end{matrix} & \lambda_3 \end{pmatrix} \right).
\end{equation}
We observe that for every $\xi\in\M^{2 \times 2}_{\sym}$, the matrix $\A_y\xi$ is given by the unique solution of the linear system
\[
    \C(y) \A_y \xi : \begin{pmatrix} 0 & 0 & \lambda^y_1 \\ 0 & 0 & \lambda^y_2 \\ \lambda^y_1 & \lambda^y_2 & \lambda^y_3 \end{pmatrix} = 0 \quad \text{for every } \lambda_1, \lambda_2, \lambda_3 \in \R.
\]
This implies, in particular, for every $y \in \calY$ that $\A_y$ is a linear map.

Let $Q_r : \calY \times \M^{2 \times 2}_{\sym}\to [0,+\infty)$ be the map
\[
    Q_r(y, \xi) := Q(y, \A_y \xi) \quad \text{for every } \xi \in \M^{2 \times 2}_{\sym}.
\]
By the properties of $Q$, we have that $Q_r(y,\cdot)$ is positive definite on symmetric matrices.

We also define the tensor $\C_r : \calY \times \M^{2 \times 2}_{\sym} \to \M^{3 \times 3}_{\sym}$, given by
\[
    \C_r(y) \xi := \C(y) \A_y \xi \quad \text{for every } \xi \in \M^{2 \times 2}_{\sym}.
\]
We remark that by \eqref{Q_r minimum problem} \CCC it \BBB holds
\[
    \C_r(y) \xi : \zeta = \C(y) \A_y \xi : \begin{pmatrix} \zeta^{\prime\prime} & 0 \\ 0 & 0 \end{pmatrix} \quad\text{for every } \xi\in \M^{2 \times 2}_{\sym},\,\zeta\in\M^{3 \times 3}_{\sym},
\]
and
\[
    Q_r(y, \xi) = \frac{1}{2} \C_r(y) \xi : \begin{pmatrix} \xi & 0 \\ 0 & 0 \end{pmatrix} \quad \text{for every } \xi \in \M^{2 \times 2}_{\sym}.
\]

\medskip
\noindent{\bf The reduced dissipation potential.}

The set $K_r(y) \subset \M^{2 \times 2}_{\sym}$ represents the set of admissible stresses in the reduced problem and can be characterised as follows (see \cite[Section 3.2]{Davoli.Mora.2013}):
\begin{equation} \label{K_r characherization}
    \xi \in K_r(y)
    \quad \iff \quad 
    \begin{pmatrix}
    \xi_{11} & \xi_{12} & 0 \\
    \xi_{12} & \xi_{22} & 0 \\
    0 & 0 & 0
    \end{pmatrix}
    - \frac{1}{3}({\rm tr}\xi)I_{3 \times 3} \in K(y),
\end{equation}
where $I_{3 \times 3}$ is the identity matrix in $\M^{3 \times 3}$.

The \emph{plastic dissipation potential} $H_r : \calY \times \M^{2 \times 2}_{\sym} \to [0,+\infty)$ is given by the support function of $K_r(y)$, i.e
\begin{equation*}
    H_r(y, \xi) := \sup\limits_{\sigma \in K_r(y)}\, \sigma : \xi \quad \text{for every}\, \xi \in \M^{2 \times 2}_{\sym}.
\end{equation*}
It follows that $H_r(y,\cdot)$ is convex and positively 1-homogeneous, and there are two constants $0 < r_H \leq R_H$ such that 
\begin{equation*}
    r_H |\xi| \leq H_r(y, \xi) \leq R_H |\xi| \quad \text{for every}\, \xi \in \M^{2 \times 2}_{\sym}.
\end{equation*}
Therefore $H_r(y,\cdot)$ satisfies the triangle inequality 
\begin{equation*}
H_r(y, \xi_1+\xi_2) \leq H_r(y, \xi_1) + H_r(y, \xi_2) \quad \text{for every}\, \xi_1,\, \xi_2 \in \M^{2 \times 2}_{\sym}.
\end{equation*}
Finally, for a fixed $y \in \calY$, we can deduce the property
\begin{equation*}
    K_r(y) = \partial H_r(y, 0).
\end{equation*}

\subsection{Definition of quasistatic evolutions}
The $\calH_h$-variation of a map $\mu : [0,T] \to  \Mb(\Omega \cup \Gamma_\Dir;\M^{3 \times 3}_{\dev})$ on $[a,b]$ is defined as
\begin{equation*}
    \calD_{\calH_h}(
    \mu; a, b) := \sup\left\{ \sum_{i = 1}^{n-1} \calH_h\left(\mu(t_{i+1}) - \mu(t_i)\right) : a = t_1 < t_2 < \ldots < t_n = b,\ n \in \N \right\}.
\end{equation*}

For every $t \in [0, T]$ we prescribe a boundary datum $w(t) \in H^1(\Omega;\R^3) \cap KL(\Omega)$ and we assume the map $t\mapsto w(t)$ to be absolutely continuous from $[0, T]$ into $H^1(\Omega;\R^3)$.

\begin{definition} \label{h-quasistatic evolution} 
Let $h > 0$. 
An \emph{$h$-quasistatic evolution} for the boundary datum $w(t)$ is a function $t \mapsto (u^h(t), e^h(t), p^h(t))$ from $[0,T]$ into $BD(\Omega) \times L^2(\Omega;\M^{3 \times 3}_{\sym}) \times \Mb(\Omega \cup \Gamma_\Dir;\M^{3 \times 3}_{\sym})$ that satisfies the following conditions:
\begin{enumerate}[label=(qs\arabic*)$_{h}$]
    \item \label{h-qs S} for every $t \in [0,T]$ we have $(u^h(t), e^h(t), p^h(t)) \in \calA_h(w(t))$ and
    \begin{equation*}
        \calQ_h(\Lambda_h e^h(t)) \leq \calQ_h(\Lambda_h \eta) + \calH_h(\Lambda_h \pi-\Lambda_h p^h(t)),
    \end{equation*}
    for every $(\upsilon,\eta,\pi) \in \calA_h(w(t))$.
    \item \label{h-qs E} the function $t \mapsto p^h(t)$ from $[0, T]$ into $\Mb(\Omega \cup \Gamma_\Dir;\M^{3 \times 3}_{\sym})$ has bounded variation and for every $t \in [0, T]$
    \begin{equation*}
        \calQ_h(\Lambda_h e^h(t)) + \calD_{\calH_h}(\Lambda_h p^h; 0, t) = \calQ_h(\Lambda_h e^h(0)) 
        + \int_0^t \int_{\Omega} \C\left(\tfrac{x'}{\epsh}\right) \Lambda_h e^h(s) : E\dot{w}(s) \,dx ds.
    \end{equation*}
\end{enumerate}
\end{definition}

The following existence result of a quasistatic evolution for a general multi-phase material can be found in \cite[Theorem 2.7]{Francfort.Giacomini.2012}.

\begin{theorem} \EEE Assume \eqref{tensorassumption}, \eqref{tensorassumption2}, and  \eqref{coercivity of H_i}. \BBB
Let $h > 0$ and let $(u^h_0, e^h_0, p^h_0) \in \calA_h(w(0))$ satisfy the global stability condition \ref{h-qs S}. 
Then, there exists a two-scale quasistatic evolution $t \mapsto (u^h(t), e^h(t), p^h(t))$ for the boundary datum $w(t)$ such that $u^h(0) = u_0$,\, $e^h(0) = e^h_0$, and $p^h(0) = p^h_0$.
\end{theorem}

\BBB Our goal is to study the asymptotics of the quasistatic evolution when $h$ goes to zero. The main result is given by Theorem \ref{main result}.\BBB 

\subsection{Two-scale convergence adapted to dimension reduction}

We briefly recall some results and definitions from \cite{Francfort.Giacomini.2014}. 

\begin{definition}
\label{def:2-scale-meas}
Let $\Omega \subset \R^3$ be an open set.
Let $\{\mu^h\}_{h>0}$ be a family in $\Mb(\Omega)$ and consider $\mu \in \Mb(\Omega \times \calY)$. 
We say that
\begin{equation*}
    \mu^h \weakstartwoscale \mu \quad \text{two-scale weakly* in }\Mb(\Omega \times \calY),
\end{equation*}
if for every $\chi \in C_0(\Omega \times \calY)$
\begin{equation*}
    \lim_{h \to 0} \int_{\Omega} \chi\left(x,\frac{x'}{\epsh}\right) \,d\mu^h(x) = \int_{\Omega \times \calY} \chi(x,y) \,d\mu(x,y).
\end{equation*}
The convergence above is called \emph{two-scale weak* convergence}.
\end{definition}
\BBB
\begin{remark} \label{transfertwoscale}
Notice that the family $\{\mu^h\}_{h>0}$ determines the family of measures $\{\nu^h\}_{h>0} \subset \Mb(\Omega \times \calY)$ obtained by setting
$$\int_{\Omega \times \calY} \chi(x,y)\,d\nu^h=\int_{\Omega} \chi \left(x,\frac{x'}{\CCC \eps_h \BBB}\right) \,d\mu^h (x)$$
for every $\chi \in C_0^0(\Omega \times \calY)$. Thus $\mu$ is simply the weak* limit in $\Mb(\Omega \times \calY)$ 
\CCC of \BBB $\{\nu^h\}_{h>0}$. 
\end{remark} 

We collect some basic properties \BBB of two-scale convergence \BBB below (\CCC the first one is a direct consequence of Remark \ref{transfertwoscale} and the second one follows from the definition). \III Before stating them recall \eqref{periodic set notation}. \BBB
\begin{proposition}
\begin{enumerate}[label=(\roman*)]
    \item
    Any sequence that is bounded in $\Mb(\Omega)$ admits a two-scale weakly* convergent subsequence.
    \item 
    Let $\calD \subset \calY$ and assume that 
    $\supp(\mu^h) \subset \Omega \cap (\calD_\epsh \times I)$.
    If $\mu^h \weakstartwoscale \mu$ two-scale weakly* in $\Mb(\Omega \times \calY)$, then $\supp(\mu) \subset \Omega \times \closure{\calD}$.
\end{enumerate}
\end{proposition}

\section{Compactness results}
\label{compactness}

In this section, we provide a characterization of two-scale limits of symmetrized scaled gradients. 
We will consider sequences of deformations $\{v^h\}$ such that $v^h \in BD(\Omega^h)$ for every $h > 0$, their $L^1$-norms are uniformly bounded \BBB (up to rescaling)\BBB, and their symmetrized gradients $E v^h$ form a sequence of uniformly bounded Radon measures \BBB(again, up to rescaling). \BBB 
\BBB As already explained in Section \ref{rescaled}, \BBB we associate to the sequence $\{v^h\}$ above a rescaled sequence of maps $\{u^h\} \subset BD(\Omega)$, defined as
\begin{equation*}
    u^h := (v^h_1, v^h_2, h v^h_3) \circ \psi_h,
\end{equation*}
where $\psi_h$ is defined in \eqref{eq:def-psih}.
The symmetric gradients of the maps $\{v^h\}$ and $\{u^h\}$ are related as follows
\begin{equation} \label{eq:scaled-gradient}
    \BBB \frac{1}{h} E v^h = (\psi_h)_{\#} (\Lambda_h Eu^h). \BBB
\end{equation}
\BBB The boundedness of $\frac{1}{h}\|Ev^h\|_{\mathcal{M}_b(\Omega^h;\mathbb M^{3 \times 3}_{sym}) }  $ is equivalent to the boundedness of $\|\Lambda_h Eu^h\|_{\mathcal{M}_b(\Omega;\mathbb M^{3 \times 3}_{sym}) } $.
\BBB We will express our compactness result with respect to the sequence $\{u^h\}_{h>0}$. 
\BBB 

We first recall a compactness result for sequences of non-oscillating fields (see \cite{Davoli.Mora.2013}).

\begin{proposition} \label{two-scale weak limit of scaled strains - 2x2 submatrix}
\BBB Let $\{u^h\}_{h>0} \subset BD(\Omega)$ be a sequence such that there exists a constant $C>0$ for which $$\|u^h\|_{L^1(\Omega;\R^3)}+\|\Lambda_h Eu^h\|_{\mathcal{M}_b(\Omega;\mathbb M^{3 \times 3}_{sym}) } \leq C. $$
\BBB 
Then, there exist functions $\bar{u} = (\bar{u}_1, \bar{u}_2) \in BD(\omega)$ and $u_3 \in BH(\omega)$ such that, up to subsequences, there holds 
\begin{align*}
    u^h_{\alpha} &\strong \bar{u}_{\alpha}-x_3 \partial_{x_\alpha}u_3, \quad \text{strongly in }L^1(\Omega), \quad \alpha \in \{1,2\},\\
    u^h_3 &\strong u_3, \quad \text{strongly in }L^1(\Omega),\\
     Eu^h &\weakstar  \begin{pmatrix} E \bar{u} - x_3 D^2u_3 & 0 \\ 0 & 0 \end{pmatrix}  \quad \text{weakly* in }\Mb(\Omega;\M^{3 \times 3}_{\sym}). \BLACK
\end{align*}
\end{proposition}

Now we turn to identifying the two-scale limits of the sequence $\Lambda_h E u^h$. 

\subsection{Corrector properties and duality results}
In order to define and analyze the space of measures which arise as two-scale limits of scaled symmetrized gradients of $BD$ functions, we will consider the following general framework \BBB (see also \cite{breit2020trace}). \BBB

Let $V$ and $W$ be finite-dimensional Euclidean spaces of dimensions $N$ and $M$, respectively.
We will consider $k$\textsuperscript{th} order linear homogeneous partial differential operators with constant coefficients $\DiffOpA : C_c^{\infty}(\R^n;V) \to C_c^{\infty}(\R^n;W)$. 
More precisely, the operator $\DiffOpA$ acts on functions $u : \R^n \to V$ as
\begin{equation*} \label{DiffOp definition}
    \DiffOpA u  \,:=\, \sum_{|\alpha| = k} \LinOpA_\alpha \partial^\alpha u.
\end{equation*}
where the coefficients $\LinOpA_\alpha \in W \otimes V^* \cong \mathrm{Lin}(V;W)$ are constant tensors, $\alpha = (\alpha_1, \dots, \alpha_n) \in \N_0^n$ is a multi-index and $\partial^\alpha := \partial_1^{\alpha_1} \cdots \partial_n^{\alpha_n}$ denotes the distributional partial derivative of order $|\alpha| = \alpha_1 + \cdots + \alpha_n$.

We define the space
\begin{equation*}
    BV^{\DiffOpA}(U) = \Big\{ u \in L^1(U;V) :\DiffOpA u \in \Mb(U;W) \Big\}
\end{equation*}
of \emph{functions with bounded $\DiffOpA$-variations} on an open subset $U$ of $\R^n$.
This is a Banach space endowed with the norm
\begin{equation*}
    \|u\|_{BV^{\DiffOpA}(U)} := \|u\|_{L^1(U;V)} + |\DiffOpA u|(U). 
\end{equation*}
Here, the distributional $\DiffOpA$-gradient is defined and extended to distributions via the duality
\begin{equation*}
    \int_{U} \varphi \cdot d\DiffOpA u := \int_{U} \DiffOpA^* \varphi \cdot u \,dx, \quad \varphi \in C_c^{\infty}(U;W^*),
\end{equation*}
where $\DiffOpA^* : C_c^{\infty}(\R^n;W^*) \to C_c^{\infty}(\R^n;V^*)$ is the formal $L^2$-adjoint operator of $\DiffOpA$
\begin{equation*} \label{DiffOp adjont definition}
    \DiffOpA^* \,:=\, (-1)^k \sum_{|\alpha| = k} \LinOpA_\alpha^* \partial^\alpha.
\end{equation*}
The \emph{total $\DiffOpA$-variation} of $u \in L^1_{loc}(U;V)$ is defined as
\begin{align*} \label{DiffOp variation}
    |\DiffOpA u|(U) := \sup\left\{\int_{U} \DiffOpA^*\varphi \cdot u \,dx : \varphi \in C_c^k(U;W^*), \; |\varphi| \leq 1 \right\}.
\end{align*}
Let $\{u_n\} \subset BV^{\DiffOpA}(U)$ and $u \in BV^{\DiffOpA}(U)$. We say that $\{u_n\}$ converges weakly* to $u$ in $BV^{\DiffOpA}$ if $u_n \strong u \;\text{ in } L^1(U;V)$ and $\DiffOpA u_n \weakstar \DiffOpA u \;\text{ in } \Mb(U;W)$.

In order to characterize the two-scale weak* limit of scaled symmetrized gradients, we will generally consider two domains  \BBB $\Omega_1 \subset \R^{N_1}$, $\Omega_2 \subset \R^{N_2}$, for some $N_1, N_2 \in \N$ \BBB and assume that the operator $\DiffOp$ is defined through partial derivatives only with respect to the entries of the $n_2$-tuple $x_2$.
In the spirit of \cite[Section 4.2]{Francfort.Giacomini.2014}, we will define the space
\begin{align*}
    \CorrSpace{\DiffOp}{\Omega_1}{\Omega_2} := \Big\{\mu \in \Mb(\Omega_1 \times \Omega_2&;V) : \DiffOp\mu \in \Mb(\Omega_1 \times \Omega_2;W),\\
    \mu(F \times \Omega_2&) = 0 \textrm{ for every Borel set } S \subseteq \Omega_1 \Big\}.
\end{align*}

We will assume that $BV^{\DiffOp}(\Omega_2)$ satisfies the following weak* compactness property:
\begin{assumption} \label{BV^A assumption 1}
If $\{u_n\} \subset BV^{\DiffOp}(\Omega_2)$ is uniformly bounded in the $BV^{\DiffOp}$-norm, then there exists a subsequence $\{u_m\} \subseteq \{u_n\}$ and a function $u \in BV^{\DiffOp}(\Omega_2)$ such that $\{u_m\}$ converges weakly* to $u$ in $BV^{\DiffOp}(\Omega_2)$, i.e.
\begin{equation*} \label{BV^A Poincare-Korn weak* compactness}
    u_m \strong u \;\text{ in } L^1(\Omega_2;V) \;\text{ and }\; \DiffOp u_m \weakstar \DiffOp u \;\text{ in } \Mb(\Omega_2;W).
\end{equation*}

Furthermore, there exists a countable collection $\{\U^k\}$ of open subsets of $\R^{n_2}$ that increases to $\Omega_2$ (i.e. $\closure{\U^k} \subset \U^{k+1}$ for every $k\in \N$, and $\Omega_2 = \bigcup_{k} \U^k$) such that $BV^{\DiffOp}(\U^k)$ satisfies the weak* compactness property above for every $k\in \N$.
\end{assumption}

The following theorem is our main disintegration result for measures in $\CorrSpace{\DiffOp}{\Omega_1}{\Omega_2}$, which will be instrumental to define a notion of duality for admissible two-scale configurations.
The proof is an adaptation of the arguments in \cite[Proposition 4.7]{Francfort.Giacomini.2014} \RED (see \cite[Proposition 4.2]{BDV}) \BLACK.

\begin{proposition} \label{BV^A main property}
\BBB Let Assumption \ref{BV^A assumption 1} be satisfied. \BBB
Let $\mu \in \CorrSpace{\DiffOp}{\Omega_1}{\Omega_2}$. 
Then there exist $\eta \in \Mb^+(\Omega_1)$ and a Borel map $(x_1,x_2) \in \Omega_1 \times \Omega_2 \mapsto \mu_{x_1}(x_2) \in V$ such that, for $\eta$-a.e. $x_1 \in \Omega_1$,
\begin{equation} \label{BV^A main property 1}
    \mu_{x_1} \in BV^{\DiffOp}(\Omega_2), \qquad \int_{\Omega_2} \mu_{x_1}(x_2) \,dx_2 = 0, \qquad |\DiffOp\mu_{x_1}|(\Omega_2) \neq 0,
\end{equation}
and
\begin{equation} \label{BV^A main property 2}
    \mu = \mu_{x_1}(x_2) \,\eta \otimes \calL^{n_2}_{x_2}.
\end{equation}
Moreover, the map $x_1 \mapsto \DiffOp\mu_{x_1} \in \Mb(\Omega_2;W)$ is $\eta$-measurable and
\begin{equation*}
    \DiffOp\mu = \eta \genprod \DiffOp\mu_{x_1}.
\end{equation*}
\end{proposition}

Lastly, we give a necessary and sufficient condition with which we can characterize the $\DiffOp$-gradient of a measure, under the following two assumptions.

\begin{assumption} \label{BV^A assumption 2}
For every $\chi \in C_0(\Omega_1 \times \Omega_2;W)$ with $\DiffOp^*\chi = 0$ (in the sense of distributions), there exists a sequence of smooth functions $\{\chi_n\} \subset C_c^{\infty}(\Omega_1 \times \Omega_2;W)$ such that $\DiffOp^*\chi_n = 0$ for every $n$, and $\chi_n \strong \chi$ in $L^{\infty}(\Omega_1 \times \Omega_2;W)$.
\end{assumption}

\begin{assumption} \label{BV^A assumption 3}
The following Poincar\'{e}-Korn type inequality holds in $BV^{\DiffOp}(\Omega_2)$:
\begin{equation*} \label{BV^A Poincare-Korn inequality}
    \left\|u - \int_{\Omega_2} u \,dx_2\right\|_{L^1(\Omega_2;V)} \leq C |\DiffOp u|(\Omega_2), \quad \forall u \in BV^{\DiffOp}(\Omega_2).
\end{equation*}
\end{assumption}

\RED The proof of the following result is given in \cite[Proposition 4.3]{BDV}. \BLACK

\begin{proposition} \label{BV^A duality lemma}
\BBB Let Assumptions \ref{BV^A assumption 1}, \ref{BV^A assumption 2} and \ref{BV^A assumption 3} be satisfied. \BBB
Let $\lambda \in \Mb(\Omega_1 \times \Omega_2;W)$. 
Then, the following items are equivalent:
\begin{enumerate}[label=(\roman*)]
	\item \label{BV^A duality lemma (i)}
	For every $\chi \in C_0(\Omega_1 \times \Omega_2;W)$ with $\DiffOp^*\chi = 0$ (in the sense of distributions) we have
	\begin{equation*}
	    \int_{\Omega_1 \times \Omega_2} \chi(x_1,x_2) \cdot d\lambda(x_1,x_2) = 0.
	\end{equation*}
	\item \label{BV^A duality lemma (ii)}
	There exists $\mu \in \CorrSpace{\DiffOp}{\Omega_1}{\Omega_2}$ such that $\lambda = \DiffOp\mu$. 
\end{enumerate}
\end{proposition}
\CCC Next we will apply these results to obtain auxiliary claims which we will use \EEE to characterize two-scale limits of scaled symmetrized gradients. \BBB
\subsubsection{Case \texorpdfstring{$\gamma = 0$}{γ = 0}}
\CCC We \BBB consider $\DiffOp = E_{y}$,\, $\DiffOp^* = \div_{y}$,\, $\Omega_1 = \omega$, and $\Omega_2 = \calY$ (\CCC it can be easily seen that Proposition \ref{BV^A main property} and Proposition \ref{BV^A duality lemma} are also valid if we take $\Omega_2=\calY$\BBB).  Then, 
$BV^{\DiffOp}(\Omega_2) = BD(\calY)$ and we denote the associated corrector space by
\begin{align*}
    \calXzero{\omega} := \Big\{\mu \in \Mb(\omega \times \calY;\R^2) : E_{y}\mu \in \Mb(\omega \times \calY;\M^{2 \times 2}_{\sym}),&\\
    \mu(F \times \calY) = 0 \textrm{ for every Borel set } F \subseteq \omega &\Big\}.
\end{align*}

\begin{remark}
We note that $\calXzero{\omega}$ is the $2$-dimensional variant of the set introduced in \cite[Section 4.2]{Francfort.Giacomini.2014}, where its main properties have been characterized. 
\end{remark}

\EEE Analogously, let \BBB $\DiffOp = D^2_{y}$,\, $\DiffOp^* = \div_{y}\div_{y}$,\, $\Omega_1 = \omega$, and $\Omega_2 = \calY$, then $BV^{\DiffOp}(\Omega_2) = BH(\calY)$ and we denote the associated corrector space by
\begin{align*}
    \calYzero{\omega} := \Big\{\kappa \in \Mb(\omega \times \calY) : D^2_{y}\kappa \in \Mb(\omega \times \calY;\M^{2 \times 2}_{\sym}),&\\
    \kappa(F \times \calY) = 0 \textrm{ for every Borel set } F \subseteq \omega &\Big\}.
\end{align*}

\begin{remark}
It is known that that \Cref{BV^A assumption 1} and \Cref{BV^A assumption 2} are satisfied in $BH(\calY)$, so we only need to justify \Cref{BV^A assumption 3}. 

Owing to \cite[Remarque 1.3]{Demengel.1984}, there exists a constant $C > 0$ such that
\begin{equation*}
    \| u-p(u) \|_{BH(\calY)} \leq C |D^2_{y}u|(\calY),
\end{equation*}
where 
$p(u)$ is given by
\begin{equation*}
    p(u) = \int_{\calY} \nabla_{y}u\,dy  \cdot y + \int_{\calY} u \,dy - \int_{\calY} \nabla_{y}u \,dy \cdot \int_{\calY} y \,dy.
\end{equation*}
However, since integrating first derivatives of periodic functions over the \EEE periodicity cell provides a zero contribution, \BBB we precisely obtain the desired Poincar\'{e}-Korn type inequality.
\end{remark}

As a consequence of \Cref{BV^A main property} and \Cref{BV^A duality lemma}, we infer the following results.

\begin{proposition} \label{corrector main property - regime zero}
Let $\mu \in \calXzero{\omega}$ and $\kappa \in \calYzero{\omega}$. 
Then there exist $\eta \in \Mb^+(\omega)$ and Borel maps $(x',y) \in \omega \times \calY \mapsto \mu_{x'}(y) \in \R^2$ and $(x',y) \in \omega \times \calY \mapsto \kappa_{x'}(y) \in \R$ such that, for $\eta$-a.e. $x' \in \omega$,
\begin{align*}
    \mu_{x'} \in BD(\calY), \qquad \int_{\calY} \mu_{x'}(y) \,dy = 0, \qquad |E_{y}\mu_{x'}|(\calY) \neq 0,\\
    \kappa_{x'} \in BH(\calY), \qquad \int_{\calY} \kappa_{x'}(y) \,dy = 0, \qquad |D^2_{y}\kappa_{x'}|(\calY) \neq 0,
\end{align*}
and
\begin{align*}
    \mu = \mu_{x'}(y) \,\eta \otimes \calL^{2}_{y}, \qquad
    \kappa = \kappa_{x'}(y) \,\eta \otimes \calL^{2}_{y}.
\end{align*}
Moreover, the maps $x' \mapsto E_{y}\mu_{x'} \in \Mb(\calY;\M^{2 \times 2}_{\sym})$ and $x' \mapsto D^2_{y}\kappa_{x'} \in \Mb(\calY;\M^{2 \times 2}_{\sym})$ are $\eta$-measurable and
\begin{align*}
    E_{y}\mu = \eta \genprod E_{y}\mu_{x'}, \qquad
    D^2_{y}\kappa = \eta \genprod D^2_{y}\kappa_{x'}.
\end{align*}
\end{proposition}

\begin{proposition} \label{duality lemma 1 - regime zero}
Let $\lambda \in \Mb(\omega \times \calY;\M^{2 \times 2}_{\sym})$. The following items are equivalent:
\begin{enumerate}[label=(\roman*)]
    \item For every $\chi \in C_0(\omega \times \calY;\M^{2 \times 2}_{\sym})$ with $\div_{y}\chi(x',y) = 0$ (in the sense of distributions) we have
    \begin{equation*}
        \int_{\omega \times \calY} \chi(x',y) : d\lambda(x',y) = 0.
    \end{equation*}
    \item There exists $\mu \in \calXzero{\omega}$ such that $\lambda = E_{y}\mu$.
\end{enumerate}
\end{proposition}

\begin{proposition} \label{duality lemma 2 - regime zero}
Let $\lambda \in \Mb(\omega \times \calY;\M^{2 \times 2}_{\sym})$. The following items are equivalent:
\begin{enumerate}[label=(\roman*)]
	\item For every $\chi \in C_0(\omega \times \calY;\M^{2 \times 2}_{\sym})$ with $\div_{y}\div_{y}\chi(x',y) = 0$ (in the sense of distributions) we have
	\begin{equation*}
	    \int_{\omega \times \calY} \chi(x',y) : d\lambda(x',y) = 0.
	\end{equation*}
	\item There exists $\kappa \in \calYzero{\omega}$ such that $\lambda = D^2_{y}\kappa$. 
\end{enumerate}
\end{proposition}

\subsubsection{Case \texorpdfstring{$\gamma = +\infty$}{γ = +∞}}
\EEE In this scaling regime, we \BBB consider $\DiffOp = E_{y}$,\, $\DiffOp^* = \div_{y}$,\, $\Omega_1 = \Omega$, and $\Omega_2 = \calY$. Then, 
$BV^{\DiffOp}(\Omega_2) = BD(\calY)$ and we denote the associated corrector space by
\begin{align*}
    \calXinf{\Omega} := \Big\{\mu \in \Mb(\Omega \times \calY&;\R^2) : E_{y}\mu \in \Mb(\Omega \times \calY;\M^{2 \times 2}_{\sym}),\\
    \mu(F \times \calY&) = 0 \textrm{ for every Borel set } F \subseteq \Omega \Big\},
\end{align*}
Further, we choose $\DiffOp = D_{y}$,\, $\DiffOp^* = \div_{y}$,\, $\Omega_1 = \Omega$, and $\Omega_2 = \calY$, so that 
$BV^{\DiffOp}(\Omega_2) = BV(\calY)$ and the associated corrector space is given by
\begin{align*}
    \calYinf{\Omega} := \Big\{\kappa \in \Mb(\Omega \times \calY) : D_{y}\kappa \in \Mb(\Omega \times \calY;\R^2),\quad \quad&\\
    \kappa(F \times \calY) = 0 \textrm{ for every Borel set } F \subseteq \Omega &\Big\}.
\end{align*}

Clearly \Cref{BV^A assumption 1}, \Cref{BV^A assumption 2} and \Cref{BV^A assumption 3} are satisfied in $BD(\calY)$ and $BV(\calY)$.
Thus, we can state the following propositions as consequences of \Cref{BV^A main property} and \Cref{BV^A duality lemma}.

\begin{proposition} \label{corrector main property - regime inf}
Let $\mu \in \calXinf{\Omega}$ and $\kappa \in \calYinf{\Omega}$. 
Then there exist $\eta \in \Mb^+(\Omega)$ and Borel maps $(x,y) \in \Omega \times \calY \mapsto \mu_{x}(y) \in \R^2$ and $(x,y) \in \Omega \times \calY \mapsto \kappa_{x}(y) \in \R^2$ such that, for $\eta$-a.e. $x \in \Omega$,
\begin{align*}
    \mu_{x} \in BD(\calY), \qquad \int_{\calY} \mu_{x}(y) \,dy = 0, \qquad |E_{y}\mu_{x}|(\calY) \neq 0,\\
    \kappa_{x} \in BV(\calY), \qquad \int_{\calY} \kappa_{x}(y) \,dy = 0, \qquad |D_{y}\kappa_{x}|(\calY) \neq 0,
\end{align*}
and
\begin{align*}
    \mu = \mu_{x}(y) \,\eta \otimes \calL^{2}_{y}, \qquad
    \kappa = \kappa_{x}(y) \,\eta \otimes \calL^{2}_{y}.
\end{align*}
Moreover, the maps $x \mapsto E_{y}\mu_{x} \in \Mb(\calY;\M^{2 \times 2}_{\sym})$ and $x \mapsto D_{y}\kappa_{x} \in \Mb(\calY;\R^2)$ are $\eta$-measurable and
\begin{align*}
    E_{y}\mu = \eta \genprod E_{y}\mu_{x}, \qquad
    D_{y}\kappa = \eta \genprod D_{y}\kappa_{x}.
\end{align*}
\end{proposition}

\begin{proposition} \label{duality lemma 1 - regime inf}
Let $\lambda \in \Mb(\Omega \times \calY;\M^{2 \times 2}_{\sym})$. The following items are equivalent:
\begin{enumerate}[label=(\roman*)]
    \item For every $\chi \in C_0(\Omega \times \calY;\M^{2 \times 2}_{\sym})$ with $\div_{y}\chi(y) = 0$ (in the sense of distributions) we have
    \begin{equation*}
        \int_{\calY} \chi(x,y) : d\lambda(x,y) = 0.
    \end{equation*}
    \item There exists $\mu \in \calXinf{\Omega}$ such that $\lambda = E_{y}\mu$.
\end{enumerate}
\end{proposition}

\begin{proposition} \label{duality lemma 2 - regime inf}
Let $\lambda \in \Mb(\Omega \times \calY;\R^2)$. The following items are equivalent:
\begin{enumerate}[label=(\roman*)]
    \item For every $\chi \in C_0(\Omega \times \calY;\R^2)$ with $\div_{y}\chi(y) = 0$ (in the sense of distributions) we have
    \begin{equation*}
        \int_{\calY} \chi(x,y) : d\lambda(x,y) = 0.
    \end{equation*}
    \item There exists $\kappa \in \calYinf{\Omega}$ such that $\lambda = D_{y}\kappa$.
\end{enumerate}
\end{proposition}

\subsection{\CCC Additional auxiliary results\BBB}

\subsubsection{Case \texorpdfstring{$\gamma = 0$}{γ = 0}}
In order to simplify the proof of the structure result for the two-scale limits of symmetrized scaled gradients, we will use the following lemma.

\begin{lemma} \label{auxiliary result - regime zero}
Let $\{\mu^h\}_{h>0}$ be a bounded family in $\Mb(\Omega;\M^{2 \times 2}_{\sym})$ such that
\begin{equation*}
    \mu^h \weakstartwoscale \mu \quad \text{two-scale weakly* in $\Mb(\Omega \times \calY;\M^{2 \times 2}_{\sym})$}.
\end{equation*}
for some $\mu \in \Mb(\Omega \times \calY;\M^{2 \times 2}_{\sym})$ as $h \to 0$.
Assume that
\begin{enumerate}[label=(\roman*)]
    \item $\bar{\mu}^h \weakstartwoscale \lambda_1$ two-scale weakly* in $\Mb(\omega \times \calY;\M^{2 \times 2}_{\sym})$, for some $\lambda_1 \in \Mb(\omega \times \calY;\M^{2 \times 2}_{\sym})$;
    \item For every $\chi \in C_c^{\infty}(\omega \times \calY;\M^{2 \times 2}_{\sym})$ such that $\div_{y}\div_{y}\chi(x',y) = 0$ we have
    \begin{equation*} 
        \lim_{h \to 0} \int_{\omega} \chi\hspace{-0.25em}\left(x',\tfrac{x'}{\epsh}\right) : d\hat{\mu}^h(x')  = \int_{\omega \times \calY} \chi(x',y) : d\lambda_2(x',y),
    \end{equation*}
    for some $\lambda_2 \in \Mb(\omega \times \calY;\M^{2 \times 2}_{\sym})$;
    \item There exists an open set $\ext{I} \supset I$ which compactly contains $I$ such $(\mu^h)^\perp \weakstartwoscale 0$ two-scale weakly* in $\Mb(\omega \times \ext{I} \times \calY;\M^{2 \times 2}_{\sym})$.
\end{enumerate}
Then, there exists $\kappa \in \calYzero{\omega}$ such that 
\begin{equation*}
    \mu = \lambda_1 \otimes \calL^{1}_{x_3}  + \left( \lambda_2 + D^2_{y}\kappa \right) \otimes x_3 \calL^{1}_{x_3}.
\end{equation*}
\end{lemma}

\begin{proof}
Every $\mu^h$ determines a measure $\nu^h$ on $\omega \times \ext{I} \times \calY$ with the relation
\begin{equation*}
    \nu^h(B) := \mu^h(B \cap (\Omega \times \calY))
\end{equation*}
for every Borel set $B \subseteq \omega \times \ext{I} \times \calY$. 
With a slight abuse of notation, we will still write $\mu^h$ instead of $\nu^h$.

Let $\nu$ be the measure such that
\begin{equation*}
    \mu^h \weakstartwoscale \nu \quad \text{two-scale weakly* in $\Mb(\omega \times \ext{I} \times \calY;\M^{2 \times 2}_{\sym})$}.
\end{equation*}
We first observe that, from the assumption (i) and (iii), it follows that $\bar{\nu} = \lambda_1$ and $\nu^\perp = 0$. 
Furthermore, $\mu^h \weakstartwoscale \nu$  two-scale weakly* in $\Mb(\Omega \times \calY;\M^{2 \times 2}_{\sym})$.

Let $\chi \in C_c^{\infty}(\Omega \times \calY;\M^{2 \times 2}_{\sym})$. 
If we consider the following orthogonal decomposition
\begin{equation*}
    \chi(x,y) = \bar{\chi}(x',y) + x_3 \hat{\chi}(x',y) + \chi^\perp(x,y),
\end{equation*}
then we have that
\begin{align*}
    &\int_{\Omega \times \calY} \chi(x,y) : d\nu(x,y) 
    = \lim_{h \to 0} \int_{\Omega} \chi\hspace{-0.25em}\left(x,\tfrac{x'}{\epsh}\right) : d\mu^h(x') \\
    &= \lim_{h \to 0} \int_{\omega} \bar{\chi}\hspace{-0.25em}\left(x',\tfrac{x'}{\epsh}\right) : d\bar{\mu}^h(x') + \frac{1}{12} \lim_{h \to 0} \int_{\omega} \hat{\chi}\hspace{-0.25em}\left(x',\tfrac{x'}{\epsh}\right) : d\hat{\mu}^h(x') + \lim_{h \to 0} \int_{\Omega} \chi^\perp\hspace{-0.25em}\left(x,\tfrac{x'}{\epsh}\right) : d(\mu^h)^\perp(x) \\
    &= \int_{\omega \times \calY} \bar{\chi}\left(x',y\right) : d\lambda_1(x',y) + \frac{1}{12} \lim_{h \to 0} \int_{\omega} \hat{\chi}\hspace{-0.25em}\left(x',\tfrac{x'}{\epsh}\right) : d\hat{\mu}^h(x').
\end{align*}
Suppose now that $\chi(x,y) = x_3 \hat{\chi}(x',y)$ with $\div_{y}\div_{y}\hat{\chi}(x',y) = 0$. 
Then the above equality yields
\begin{equation*}
    \int_{\omega \times \calY} \hat{\chi}(x',y) : d\hat{\nu}(x',y) 
    = \lim_{h \to 0} \int_{\omega} \hat{\chi}\hspace{-0.25em}\left(x',\tfrac{x'}{\epsh}\right) : d\hat{\mu}^h(x')
    = \int_{\omega \times \calY} \hat{\chi}(x',y) : d\lambda_2(x',y).
\end{equation*}
By a density argument, we infer that
\begin{equation*}
    \int_{\omega \times \calY} \hat{\chi}(x',y) : d\left(\hat{\nu}(x',y) - \lambda_2(x',y)\right) = 0,
\end{equation*}
for every $\hat{\chi} \in C_0(\omega \times \calY;\M^{2 \times 2}_{\sym})$ with $\div_{y}\div_{y}\hat{\chi}(x',y) = 0$ (in the sense of distributions). 
From this and \Cref{duality lemma 2 - regime zero} we conclude that there exists $\kappa \in \calYzero{\omega}$ such that
\begin{equation*}
    \hat{\nu} - \lambda_2 =  D^2_{y}\kappa.
\end{equation*}
Since $\mu = \nu$ on $\Omega \times \calY$, we obtain the claim.
\end{proof}

\subsubsection{Case \texorpdfstring{$\gamma = +\infty$}{γ = +∞}}
The following result will be in the proof of the structure result for the two-scale limits of symmetrized scaled gradients.
We note, however, that this result is independent of the limit value $\gamma$. 

\begin{proposition} \label{auxiliary result - regime inf}
Let $\{v^h\}_{h>0}$ be a bounded family in $BD(\Omega)$ such that
\begin{equation*}
    v^h \weakstar v \quad \text{weakly* in $BD(\Omega)$},
\end{equation*}
for some $v \in BD(\Omega)$. 
Then there exists $\mu \in \calXinf{\Omega}$ such that 
\begin{equation*}
    \left(Ev^h\right)'' \weakstartwoscale E_{x'}v' \otimes \calL^{2}_{y} + E_{y}\mu \quad \text{two-scale weakly* in $\Mb(\Omega \times \calY;\M^{2 \times 2}_{\sym})$}.
\end{equation*}
\end{proposition}

\begin{proof}
The proof follows closely that of \cite[Proposition 4.10]{Francfort.Giacomini.2014}.

By compactness, the exists $\lambda \in \Mb(\Omega \times \calY;\M^{3 \times 3}_{\sym})$ such that (up to a subsequence)
\begin{equation*}
    Ev^h \weakstartwoscale \lambda \quad \text{two-scale weakly* in $\Mb(\Omega \times \calY;\M^{3 \times 3}_{\sym})$}.
\end{equation*}
Since $v^h \strong v$ strongly in $L^{1}(\Omega;\R^3)$, we have componentwise
\begin{align*}
    v^h_i \weakstartwoscale v_i(x) \,\calL^{3}_{x} \otimes \calL^{2}_{y} \quad \text{two-scale weakly* in $\Mb(\Omega \times \calY)$}, \quad i=1,2,3.
\end{align*}
Consider $\chi \in C_c^{\infty}(\Omega \times \calY;\M^{2 \times 2}_{\sym})$ such that $\div_{y}\chi(x,y) = 0$. 
Then
\begin{align*}
    & \lim_{h\to0} \int_{\Omega} \chi\hspace{-0.25em}\left(x,\tfrac{x'}{\epsh}\right) : \,d\left(Ev^h\right)''(x) 
    = \lim_{h\to0} \int_{\Omega} \chi\hspace{-0.25em}\left(x,\tfrac{x'}{\epsh}\right) : \,dE_{x'}(v^h)'(x)\\
    &= - \lim_{h\to0} \int_{\Omega} (v^h)'(x) \cdot \div_{x'}\left(\chi\hspace{-0.25em}\left(x,\tfrac{x'}{\epsh}\right)\right) \,dx\\ 
    &= - \lim_{h\to0} \left( \int_{\Omega} (v^h)'(x) \cdot \div_{x'}\chi\hspace{-0.25em}\left(x,\tfrac{x'}{\epsh}\right) \,dx + \frac{1}{\epsh} \int_{\Omega} (v^h)'(x) \cdot \div_{y}\chi\hspace{-0.25em}\left(x,\tfrac{x'}{\epsh}\right) \,dx \right)\\ 
    &= - \lim_{h\to0} \int_{\Omega} (v^h)'(x) \cdot \div_{x}\chi\hspace{-0.25em}\left(x,\tfrac{x'}{\epsh}\right) \,dx\\
    &= - \int_{\Omega \times \calY} v'(x) \cdot \div_{x'}\chi\left(x,y\right) \,dx dy\\
    &= \int_{\Omega \times \calY} \chi(x,y) : d\left(E_{x'}v' \otimes \calL^{2}_{y} \right).
\end{align*}
By a density argument, we infer that
\begin{equation*}
    \int_{\Omega \times \calY} \chi(x,y) : d\left(\lambda(x,y) - E_{x'}v' \otimes \calL^{2}_{y} \right) = 0,
\end{equation*}
for every $\chi \in C_0(\Omega \times \calY;\M^{2 \times 2}_{\sym})$ with $\div_{y}\chi(x,y) = 0$ (in the sense of distributions). 
In view of \Cref{duality lemma 1 - regime inf} we conclude that there exists $\mu\in \calXinf{\Omega}$ such that 
\begin{equation*}
    \lambda - E_{x'}v' \otimes \calL^{2}_{y} = E_{y}\mu.
\end{equation*}
\CCC This \EEE yields \BBB the claim. \BBB
\end{proof}

\subsection{Two-scale limits of scaled symmetrized gradients}
We are now ready to prove the \CCC main \BBB result of this section.

\begin{theorem} \label{two-scale weak limit of scaled strains}
\BBB Let $\{u^h\}_{h>0} \subset BD(\Omega)$ be a sequence such that there exists a constant $C>0$ for which $$\|u^h\|_{L^1(\Omega;\R^3)}+\|\Lambda_h Eu^h\|_{\mathcal{M}_b(\Omega; \mathbb M^{3 \times 3}_{sym}) } \leq C. $$ \BBB
Then there exist 
\begin{equation*}
    \bar{u} = (\bar{u}_1, \bar{u}_2) \in BD(\omega), \quad u_3 \in BH(\omega), \quad \widetilde{E} \in \Mb(\Omega \times \calY;\M^{3 \times 3}_{\sym}),
\end{equation*}
and a (not relabeled) subsequence of $\{u^h\}_{h>0}$ which satisfy
\begin{eqnarray*}
    \Lambda_h Eu^h 
    \weakstartwoscale 
    \begin{pmatrix} E\bar{u} - x_3 D^2u_3 & 0 \\ 0 & 0 \end{pmatrix} \otimes \calL^{2}_{y} 
    + 
    \Corrector 
    \quad \text{two-scale weakly* in $\Mb(\Omega \times \calY;\M^{3 \times 3}_{\sym})$}.
\end{eqnarray*}
\begin{enumerate}[label=(\alph*)]
    \item \label{two-scale weak limit of scaled strains - regime zero}
    If $\gamma = 0$, then there exist $\mu \in \calXzero{\omega}$, $\kappa \in \calYzero{\omega}$ and $\zeta \in \Mb(\Omega \times \calY;\R^3)$ such that
    \begin{eqnarray*}
        \Corrector
        \,=\, 
        \begin{pmatrix} \begin{matrix} E_{y}\mu(x',y) - x_3 D^2_{y}\kappa(x',y) \end{matrix} & \zeta'(x,y) \\ (\zeta'(x,y))^T & \zeta_3(x,y) \end{pmatrix}.
    \end{eqnarray*}
    \item \label{two-scale weak limit of scaled strains - regime inf}
    If $\gamma = +\infty$, then there exist $\mu \in \calXinf{\Omega}$, $\kappa \in \calYinf{\Omega}$ and $\zeta \in \Mb(\Omega;\R^3)$ such that
    \begin{eqnarray*}
        \Corrector
        \,=\, 
        \begin{pmatrix} \begin{matrix} E_{y}\mu(x,y) \end{matrix} & \zeta'(x) + D_{y}\kappa(x,y) \\ (\zeta'(x) + D_{y}\kappa(x,y))^T & \zeta_3(x) \end{pmatrix}.
    \end{eqnarray*}
\end{enumerate}
\end{theorem}

\begin{proof}
Owing to \cite[Chapter II, Remark 3.3]{Temam.1985}, we can assume without loss of generality that the maps $u^h$ are smooth functions for every $h > 0$.
Further, the uniform boundedness of the sequence $\{Ev^h\}$ implies that
\begin{align}
    \label{recall C h} &\int_{\Omega} |\partial_{x_\alpha}u^h_3+\partial_{x_3}u^h_\alpha| \,dx \leq C h, \quad \text{ for } \alpha = 1, 2,\\
    \label{recall C h^2} &\int_{\Omega} |\partial_{x_3}u^h_3| \,dx \leq C h^2.
\end{align}

In the following, we will consider $\lambda \in \Mb(\Omega \times \calY;\M^{3 \times 3}_{\sym})$ such that 
\begin{eqnarray*}
    \Lambda_h Eu^h \weakstartwoscale \lambda \quad \text{two-scale weakly* in $\Mb(\Omega \times \calY;\M^{3 \times 3}_{\sym})$}.
\end{eqnarray*}

\noindent{\bf Step 1.}
We consider the case $\gamma = 0$, i.e. $\frac{h}{\epsh} \to 0$.

By the Poincar\'{e} inequality in $L^1(I)$, there is a constant $C$ independent of $h$ such that
\begin{equation*}
    \int_{I} |u^h_3-\averageI{u}^h_3| \,dx_3 \leq C \int_{I} |\partial_{x_3}u^h_3| \,dx_3,
\end{equation*}
for a.e. $x' \in \omega$. 
Integrating over $\omega$ we obtain that
\begin{equation} \label{estimate 1 - regime zero}
    \int_{\Omega} |u^h_3-\averageI{u}^h_3| \,dx \leq C \int_{\Omega} |\partial_{x_3}u^h_3| \,dx \leq C h^2.
\end{equation}
\EEE Set \BBB
\begin{equation*} \label{theta^h_3 - regime zero}
    \vartheta^h_3(x) := \frac{u^h_3(x) - \averageI{u}^h_3(x')}{h^2}.
\end{equation*}
\EEE We have that $\{\vartheta^h_3\}_{h>0}$ \BBB is uniformly bounded in $L^1(\Omega)$. 
\EEE Correspondingly, we \BBB construct a sequence of antiderivatives $\{\theta^h_3\}_{h>0}$ 
by
\begin{equation*}
    \theta^h_3(x) := \int_{-\frac{1}{2}}^{x_3} \vartheta^h_3(x',z_3) \,dz_3 - C_{\vartheta^h_3},
\end{equation*}
where we choose $C_{\vartheta^h_3}$ such that $\averageI{\theta}^h_3 = 0$. 
Note that the constructed sequence is also uniformly bounded in $L^1(\Omega)$. 
Next, for $\alpha\in\{1,2\}$, we construct sequences $\{\theta^h_\alpha\}_{h>0}$ by
\begin{equation*} \label{theta^h_alpha - regime zero}
    \theta^h_\alpha(x) := \frac{u^h_\alpha(x) - \averageI{u}^h_\alpha(x') + x_3 \partial_{x_\alpha}\averageI{u}^h_3(x')}{h} + h \partial_{x_\alpha}\theta^h_3(x).
\end{equation*}
Then $\averageI{\theta}^h_\alpha = 0$ and  
\begin{align*}
    \partial_{x_3}\theta^h_\alpha = \frac{\partial_{x_3}u_\alpha^h + \partial_{x_\alpha}\averageI{u}^h_3}{h} + h \partial_{x_\alpha}\vartheta^h_3 = \frac{\partial_{x_3}u_\alpha^h + \partial_{x_\alpha}u^h_3}{h},
\end{align*}
since $\partial_{x_3}\theta^h_3 = \vartheta^h_3$.
Thus, by the Poincar\'{e} inequality in $L^1(I)$ and integrating over $\omega$, we obtain that 
\begin{equation} \label{estimate 2 - regime zero}
    \int_{\Omega} | \theta^h_\alpha | \,dx \leq C \int_{\Omega} |\partial_{x_3}\theta^h_\alpha| \,dx \leq C.
\end{equation}

From the above constructions, we infer
\begin{equation} \label{u^h_alpha structure - regime zero}
    u^h_\alpha(x) = \averageI{u}^h_\alpha(x') - x_3 \partial_{x_\alpha}\averageI{u}^h_3(x') + h^2 \partial_{x_\alpha}\theta^h_3(x) + h \theta^h_\alpha(x), \quad \alpha=1,2.
\end{equation}
For the $2 \times 2$ minors of the scaled symmetrized gradients, a direct calculation shows
\begin{align} \label{limit 0 - regime zero}
    &\nonumber \int_{\Omega \times \calY} \chi(x,y) : d\lambda^{\prime\prime}(x,y) \\
    &= \lim_{h \to 0} \int_{\Omega} \chi\hspace{-0.25em}\left(x,\tfrac{x'}{\epsh}\right) : \left( E(\averageI{u}^h)^{\prime}(x') - x_3 D^2\averageI{u}^h_3(x') + h^2 D^2_{x'}\theta^h_3(x) + h E_{x'}(\theta^h)^{\prime}(x) \right) \,dx,
\end{align}
for every $\chi \in C_c^{\infty}(\omega;C^{\infty}(I \times \calY;\M^{2 \times 2}_{\sym}))$. 
Notice that
the last two terms in \eqref{limit 0 - regime zero} are negligible in the limit. 
Indeed, we have
{\allowdisplaybreaks
\begin{align} \label{limit 3a - regime zero}
    \nonumber & \lim_{h \to 0} \int_{\Omega} \chi\hspace{-0.25em}\left(x,\tfrac{x'}{\epsh}\right) : h^2 D^2_{x'}\theta^h_3(x) \,dx\\
    \nonumber &= \lim_{h \to 0} h^2 \int_{\Omega} \theta^h_3(x)\,\div_{x'}\div_{x'}\left(\chi\hspace{-0.25em}\left(x,\tfrac{x'}{\epsh}\right)\right) \,dx\\ 
    \nonumber &= \lim_{h \to 0} h^2 \sum_{\alpha,\beta=1,2} \int_{\Omega} \theta^h_3(x)\,\partial_{x_\alpha}\left( \partial_{x_\beta}\chi_{\alpha \beta}\hspace{-0.25em}\left(x,\tfrac{x'}{\epsh}\right) + \frac{1}{\epsh}\partial_{y_\beta}\chi_{\alpha \beta}\hspace{-0.25em}\left(x,\tfrac{x'}{\epsh}\right) \right) \,dx\\ 
    \nonumber &= \lim_{h \to 0} \sum_{\alpha,\beta=1,2} \int_{\Omega} \theta^h_3(x)\,\bigg( h^2\partial_{x_\alpha x_\beta}\chi_{\alpha \beta}\hspace{-0.25em}\left(x,\tfrac{x'}{\epsh}\right) + \frac{h^2}{\epsh}\partial_{y_\alpha x_\beta}\chi_{\alpha \beta}\hspace{-0.25em}\left(x,\tfrac{x'}{\epsh}\right)\\\nonumber&\hspace{10em} + \frac{h^2}{\epsh}\partial_{x_\alpha y_\beta}\chi_{\alpha \beta}\hspace{-0.25em}\left(x,\tfrac{x'}{\epsh}\right) + \frac{h^2}{\epsh^2}\partial_{y_\alpha y_\beta}\chi_{\alpha \beta}\hspace{-0.25em}\left(x,\tfrac{x'}{\epsh}\right) \bigg) \,dx\\
    &= 0.
\end{align}
}
Similarly we compute
\begin{align} \label{limit 3b - regime zero}
    \nonumber & \lim_{h \to 0} \int_{\Omega} \chi\hspace{-0.25em}\left(x,\tfrac{x'}{\epsh}\right) : h E_{x'}(\theta^h)^{\prime}(x) \,dx\\
    \nonumber &= -\lim_{h \to 0} h \int_{\Omega} (\theta^h)^{\prime}(x) \cdot \div_{x'}\left(\chi\hspace{-0.25em}\left(x,\tfrac{x'}{\epsh}\right)\right) \,dx\\ 
    \nonumber &= -\lim_{h \to 0} \sum_{\alpha,\beta=1,2} \int_{\Omega} \theta^h_\alpha(x)\,\left( h\partial_{x_\beta}\chi_{\alpha \beta}\hspace{-0.25em}\left(x,\tfrac{x'}{\epsh}\right) + \frac{h}{\epsh}\partial_{y_\beta}\chi_{\alpha \beta}\hspace{-0.25em}\left(x,\tfrac{x'}{\epsh}\right) \right) \,dx\\ 
    &= 0.
\end{align}
Thus, considering an open set $\ext{I} \supset I$ which compactly contains $I$, we infer 
\begin{equation} \label{limit 3 - regime zero}
    \left(E_{\alpha \beta}(u^h)\right)^\perp \weakstartwoscale 0 \quad \text{two-scale weakly* in $\Mb(\omega \times \ext{I} \times \calY;\M^{2 \times 2}_{\sym})$}.
\end{equation}
Since $\{(\averageI{u}^h)^{\prime}\}$ is bounded in $BD(\omega)$ with $(\averageI{u}^h)^{\prime} \weakstar \bar{u} \;\text{ weakly* in } BD(\omega)$, by \cite[Proposition 4.10]{Francfort.Giacomini.2014} (\CCC the result follows by duality argument, using Proposition \ref{duality lemma 1 - regime zero}\BBB) there exists $\mu \in \calXzero{\omega}$ such that
\begin{equation} \label{limit 1 - regime zero}
    E(\averageI{u}^h)^{\prime} \weakstartwoscale E\bar{u} \otimes \calL^{2}_{y} + E_{y}\mu \quad \text{two-scale weakly* in $\Mb(\omega \times \calY;\M^{2 \times 2}_{\sym})$}.
\end{equation}
From  \CCC Proposition \BBB \ref{two-scale weak limit of scaled strains - 2x2 submatrix} there holds
\begin{align*}
    u^h_{\alpha} &\strong \bar{u}_{\alpha}-x_3 \partial_{x_\alpha}u_3, \quad \text{strongly in }L^1(\Omega), \quad \alpha=1,2,\\
    u^h_3 &\strong u_3, \quad \text{strongly in }L^1(\Omega).  
\end{align*}
thus we infer that
\begin{equation} \label{limit 2a - regime zero}
    \averageI{u}^h_3 \weakstartwoscale u_3(x') \,\calL^{2}_{x'} \otimes \calL^{2}_{y} \quad \text{two-scale weakly* in $\Mb(\omega \times \calY)$}
\end{equation}
Further, multiplying \eqref{u^h_alpha structure - regime zero} with $x_3$ and integrating over $\omega$, we obtain 
\begin{equation*} 
    \partial_{x_\alpha}\averageI{u}^h_3(x') = - \meanI{u}^h_\alpha(x') + h^2 \partial_{x_\alpha}\meanI{\theta}^h_3(x') + h \meanI{\theta}^h_\alpha(x'), \quad \alpha=1,2.
\end{equation*}
Using similar calculations as in \eqref{limit 3a - regime zero} and \eqref{limit 3b - regime zero}, we obtain that only the first term is not negligible in the limit, from which we conclude that, for any $\varphi \in C_c^{\infty}(\omega \times \calY)$
\begin{equation} \label{limit 2b - regime zero}
    \lim_{h \to 0} \int_{\omega} \partial_{x_\alpha}\averageI{u}^h_3(x')\,\varphi\hspace{-0.25em}\left(x',\tfrac{x'}{\epsh}\right) \,dx' = \int_{\omega \times \calY} \partial_{x_\alpha}u_3(x')\,\varphi\hspace{-0.25em}\left(x',y\right) \,dx' dy, \quad \alpha=1,2.
\end{equation}
Consider now $\chi \in C_c^{\infty}(\omega \times \calY;\M^{2 \times 2}_{\sym})$ such that $\div_{y}\div_{y}\chi(x',y) = 0$. 
Then
{\allowdisplaybreaks
\begin{align*}
    \nonumber & \lim_{h \to 0} \int_{\omega} \chi\hspace{-0.25em}\left(x',\tfrac{x'}{\epsh}\right) : D^2\averageI{u}^h_3(x') \,dx'\\
    \nonumber &= \lim_{h \to 0} \int_{\omega} \averageI{u}^h_3(x')\,\div_{x'}\div_{x'}\left(\chi\hspace{-0.25em}\left(x',\tfrac{x'}{\epsh}\right)\right) \,dx'\\ 
    \nonumber &= \lim_{h \to 0} \sum_{\alpha,\beta=1,2} \int_{\omega} \averageI{u}^h_3(x')\,\bigg( \partial_{x_\alpha x_\beta}\chi_{\alpha \beta}\hspace{-0.25em}\left(x',\tfrac{x'}{\epsh}\right) + \frac{1}{\epsh}\partial_{y_\alpha x_\beta}\chi_{\alpha \beta}\hspace{-0.25em}\left(x',\tfrac{x'}{\epsh}\right)\\\nonumber&\hspace{10em} + \frac{1}{\epsh}\partial_{x_\alpha y_\beta}\chi_{\alpha \beta}\hspace{-0.25em}\left(x',\tfrac{x'}{\epsh}\right) + \frac{1}{\epsh^2}\partial_{y_\alpha y_\beta}\chi_{\alpha \beta}\hspace{-0.25em}\left(x',\tfrac{x'}{\epsh}\right) \bigg) \,dx'\\
    \nonumber &= \lim_{h \to 0} \sum_{\alpha,\beta=1,2} \int_{\omega} \averageI{u}^h_3(x')\,\bigg( \partial_{x_\alpha x_\beta}\chi_{\alpha \beta}\hspace{-0.25em}\left(x',\tfrac{x'}{\epsh}\right) + \frac{2}{\epsh}\partial_{y_\alpha x_\beta}\chi_{\alpha \beta}\hspace{-0.25em}\left(x',\tfrac{x'}{\epsh}\right) \bigg) \,dx'\\
    \nonumber &= \lim_{h \to 0} \sum_{\alpha,\beta=1,2} \int_{\omega} \averageI{u}^h_3(x')\,\partial_{x_\alpha x_\beta}\chi_{\alpha \beta}\hspace{-0.25em}\left(x',\tfrac{x'}{\epsh}\right) \,dx' + 2 \int_{\omega} \bigg( \partial_{x_\alpha}\left( \averageI{u}^h_3(x')\,\partial_{x_\beta}\chi_{\alpha \beta}\hspace{-0.25em}\left(x',\tfrac{x'}{\epsh}\right) \right)\\\nonumber&\hspace{10em} - \partial_{x_\alpha}\averageI{u}^h_3(x')\,\partial_{x_\beta}\chi_{\alpha \beta}\hspace{-0.25em}\left(x',\tfrac{x'}{\epsh}\right) - \averageI{u}^h_3(x')\,\partial_{x_\alpha x_\beta}\chi_{\alpha \beta}\hspace{-0.25em}\left(x',\tfrac{x'}{\epsh}\right) \bigg) \,dx'\\
    \nonumber &= \lim_{h \to 0} \sum_{\alpha,\beta=1,2} \left( - \int_{\omega} \averageI{u}^h_3(x')\,\partial_{x_\alpha x_\beta}\chi_{\alpha \beta}\hspace{-0.25em}\left(x',\tfrac{x'}{\epsh}\right) \,dx' - 2 \int_{\omega} \partial_{x_\alpha}\averageI{u}^h_3(x')\,\partial_{x_\beta}\chi_{\alpha \beta}\hspace{-0.25em}\left(x',\tfrac{x'}{\epsh}\right) \,dx' \right),
\end{align*}
}
where in the last equality we used Green's theorem. 
Passing to the limit, by \eqref{limit 2a - regime zero} and \eqref{limit 2b - regime zero}, we have
{\allowdisplaybreaks
\begin{align} \label{limit 2 - regime zero}
    \nonumber & \lim_{h \to 0} \int_{\omega} \chi\hspace{-0.25em}\left(x',\tfrac{x'}{\epsh}\right) : D^2\averageI{u}^h_3(x') \,dx'\\
    \nonumber &= \sum_{\alpha,\beta=1,2} \left( - \int_{\omega \times \calY} u_3(x')\,\partial_{x_\alpha x_\beta}\chi_{\alpha \beta}\left(x',y\right) \,dx' dy - 2 \int_{\omega \times \calY} \partial_{x_\alpha}u_3(x')\,\partial_{x_\beta}\chi_{\alpha \beta}\left(x',y\right) \,dx' dy \right)\\
    \nonumber &= \sum_{\alpha,\beta=1,2} \bigg( - \int_{\omega \times \calY} u_3(x')\,\partial_{x_\alpha x_\beta}\chi_{\alpha \beta}\left(x',y\right) \,dx' dy \\\nonumber&\hspace{4em}- 2 \int_{\omega \times \calY} \Big( \partial_{x_\alpha}\left( u_3(x')\,\partial_{x_\beta}\chi_{\alpha \beta}\left(x',y\right) \right) -  u_3(x')\,\partial_{x_\alpha x_\beta}\chi_{\alpha \beta}\left(x',y\right) \Big) \,dx' dy \bigg)\\
    \nonumber &= \sum_{\alpha,\beta=1,2} \int_{\omega \times \calY} u_3(x')\,\partial_{x_\alpha x_\beta}\chi_{\alpha \beta}\left(x',y\right) \,dx' dy\\
    &= \int_{\omega \times \calY} \chi(x',y) : d\left(D^2u_3 \otimes \calL^{2}_{y} \right).
\end{align}
}
From \eqref{limit 3 - regime zero}, \eqref{limit 1 - regime zero}, \eqref{limit 2 - regime zero} and \Cref{auxiliary result - regime zero}, we conclude that
\begin{equation*}
    \lambda^{\prime\prime} = E\bar{u} \otimes \calL^{2}_{y} + E_{y}\mu - x_3 D^2u_3 \otimes \calL^{2}_{y} - x_3 D^2_{y}\kappa,
\end{equation*}
\CCC where $\mu \in \calXzero{\omega}$, $\kappa \in \calYzero{\omega}$. \BBB
Finally, consider the vector $\zeta^h(x)$ given by the third column of $\Lambda_h Eu^h$, for every $h > 0$. 
The boundedness of the sequence of functions $v^h \in BD(\Omega^h)$ implies that $\{\zeta^h\}_{h>0}$ is a uniformly bounded sequence in $L^1(\Omega;\R^3)$. 
Consequently, we can extract a subsequence which two-scale weakly* converges in $\Mb(\Omega \times \calY;\R^3)$ such that
\begin{align*}
    &\frac{1}{h} E_{\alpha 3}(u^h) \weakstartwoscale \zeta_\alpha \quad \text{two-scale weakly* in $\Mb(\Omega \times \calY)$}, \quad \alpha=1,2,\\
    &\frac{1}{h^2} E_{3 3}(u^h) \weakstartwoscale \zeta_3 \quad \text{two-scale weakly* in $\Mb(\Omega \times \calY)$},
\end{align*}
for a suitable $\zeta \in \Mb(\Omega \times \calY;\R^3)$.
\EEE This concludes the proof in the case $\gamma=0$. \BBB

\medskip
\noindent{\bf Step 2.}
Consider the case $\gamma = +\infty$, i.e. $\frac{\epsh}{h} \to 0$.


For the $2 \times 2$ minors of two-scale limit, by \Cref{auxiliary result - regime inf} and the proof \Cref{two-scale weak limit of scaled strains - 2x2 submatrix}, we have that there exists $\mu \in \calXinf{\Omega}$ such that 
\begin{equation*}
    \lambda''
    = \left( E\bar{u} - x_3 D^2u_3 \right) \otimes \calL^{2}_{y} + E_{y}\mu.
\end{equation*}

Let $\chi^{(1)} \in C_c^{\infty}(\Omega)$ and $\chi^{(2)} \in C^{\infty}(\calY;\CCC \M^{3 \times 3}_{\sym}\BBB)$ such that $\int_{\calY} \chi^{(2)} \,dy = 0$. 
We consider a test function $\chi(x,y) = \chi^{(1)}(x)\chi^{(2)}\hspace{-0.25em}\left(\tfrac{x'}{\epsh}\right)$, such that
\begin{equation*}
    \int_{\Omega \times \calY} \chi(x,y) : d\lambda(x,y)
    = \lim_{h \to 0} \int_{\Omega} \chi^{(1)}(x)\chi^{(2)}\hspace{-0.25em}\left(\tfrac{x'}{\epsh}\right) : d\left(\Lambda_h Eu^h(x)\right).
\end{equation*}
For each $i=1,2,3$, let $G_i$ denote the unique solution in $C^\infty(\calY)$ to the Poisson's equation
\begin{equation*}
    -\Laplace_y G_i = \chi^{(2)}_{3i}, \quad \int_{\calY} G_i \,dy = 0.
\end{equation*}
Then, observing that
\begin{equation*}
    \int_{\Omega \times \calY} \chi_{33}(x,y) : d\lambda_{33}(x,y)
    = \lim_{h \to 0} \frac{1}{h^2} \int_{\Omega} \partial_{x_3}u^h_3(x)\,\chi^{(1)}(x)\chi^{(2)}_{33}\hspace{-0.25em}\left(\tfrac{x'}{\epsh}\right) \,dx,
\end{equation*}
we find 
{\allowdisplaybreaks
\begin{align*} 
    & \int_{\Omega \times \calY} \chi_{33}(x,y) : d\lambda_{33}(x,y)\\ 
    &= -\lim_{h \to 0} \frac{1}{h^2} \sum_{\alpha=1,2} \int_{\Omega} \partial_{x_3}u^h_3(x)\,\chi^{(1)}(x)\partial_{y_\alpha y_\alpha}G_3\hspace{-0.25em}\left(\tfrac{x'}{\epsh}\right) \,dx\\
    &= \lim_{h \to 0} \frac{1}{h^2} \sum_{\alpha=1,2} \int_{\Omega} u^h_3(x)\,\partial_{x_3}\chi^{(1)}(x)\partial_{y_\alpha y_\alpha}G_3\hspace{-0.25em}\left(\tfrac{x'}{\epsh}\right) \,dx\\
    &= \lim_{h \to 0} \frac{\epsh}{h^2} \sum_{\alpha=1,2} \left( \int_{\Omega} u^h_3(x)\,\partial_{x_\alpha}\left(\partial_{x_3}\chi^{(1)}(x)\partial_{y_\alpha}G_3\hspace{-0.25em}\left(\tfrac{x'}{\epsh}\right)\right) \,dx - \int_{\Omega} u^h_3(x)\,\partial_{x_\alpha x_3}\chi^{(1)}(x)\partial_{y_\alpha}G_3\hspace{-0.25em}\left(\tfrac{x'}{\epsh}\right) \,dx \right)\\
    &= \lim_{h \to 0} \frac{\epsh}{h^2} \sum_{\alpha=1,2} \left( - \int_{\Omega} \partial_{x_\alpha}u^h_3(x)\,\partial_{x_3}\chi^{(1)}(x)\partial_{y_\alpha}G_3\hspace{-0.25em}\left(\tfrac{x'}{\epsh}\right) \,dx + \int_{\Omega} \partial_{x_3}u^h_3(x)\,\partial_{x_\alpha}\chi^{(1)}(x)\partial_{y_\alpha}G_3\hspace{-0.25em}\left(\tfrac{x'}{\epsh}\right) \,dx \right).
\end{align*}
}
Recalling \eqref{recall C h} and \eqref{recall C h^2}, we deduce
\begin{align} \label{limit 3,3 - regime inf}
    \nonumber& \int_{\Omega \times \calY} \chi_{33}(x,y) : d\lambda_{33}(x,y)\\ 
    \nonumber&= \lim_{h \to 0} \frac{\epsh}{h^2} \sum_{\alpha=1,2} \int_{\Omega} \partial_{x_3}u^h_\alpha(x)\,\partial_{x_3}\chi^{(1)}(x)\partial_{y_\alpha}G_3\hspace{-0.25em}\left(\tfrac{x'}{\epsh}\right) \,dx\\
    \nonumber&= -\lim_{h \to 0} \frac{\epsh}{h^2} \sum_{\alpha=1,2} \int_{\Omega} u^h_\alpha(x)\,\partial_{x_3 x_3}\chi^{(1)}(x)\partial_{y_\alpha}G_3\hspace{-0.25em}\left(\tfrac{x'}{\epsh}\right) \,dx\\
    \nonumber&= -\lim_{h \to 0} \frac{\epsh^2}{h^2} \sum_{\alpha=1,2} \left( \int_{\Omega} u^h_\alpha(x)\,\partial_{x_\alpha}\left(\partial_{x_3 x_3}\chi^{(1)}(x)G_3\hspace{-0.25em}\left(\tfrac{x'}{\epsh}\right)\right) \,dx - \int_{\Omega} u^h_3(x)\,\partial_{x_\alpha x_3 x_3}\chi^{(1)}(x)\partial_{y_\alpha}G_3\hspace{-0.25em}\left(\tfrac{x'}{\epsh}\right) \,dx \right)\\
    \nonumber&= \lim_{h \to 0} \frac{\epsh^2}{h^2} \sum_{\alpha=1,2} \int_{\Omega} \partial_{x_\alpha}u^h_\alpha(x)\,\partial_{x_3 x_3}\chi^{(1)}(x)G_3\hspace{-0.25em}\left(\tfrac{x'}{\epsh}\right) \,dx\\
    &= 0.
\end{align}
Thus, recalling that $\int_{\calY} \chi^{(2)}_{33} \,dy = 0$, and since for arbitrary test function we can subtract their mean value over $\calY$ to obtain a function with mean value zero, we infer that there exists $\zeta_3 \in \Mb(\Omega)$ such that
\begin{equation*}
    \lambda_{33} = \zeta_3 \otimes \calL^{2}_y.
\end{equation*}

Similarly, from the observation that
\begin{align*}
    &\int_{\Omega \times \calY} \chi_{13}(x,y) : d\lambda_{13}(x,y) + \int_{\Omega \times \calY} \chi_{23}(x,y) : d\lambda_{23}(x,y)\\
    &= \lim_{h \to 0} \frac{1}{2h} \sum_{\alpha=1,2} \int_{\Omega} \left(\partial_{x_\alpha}u^h_3(x)+\partial_{x_3}u^h_\alpha(x)\right)\,\chi^{(1)}(x)\chi^{(2)}_{3\alpha}\hspace{-0.25em}\left(\tfrac{x'}{\epsh}\right) \,dx,
\end{align*}
we deduce
\begin{align} \label{limit 3,alfa - regime inf}
    \nonumber&\int_{\Omega \times \calY} \chi_{13}(x,y) : d\lambda_{13}(x,y) + \int_{\Omega \times \calY} \chi_{23}(x,y) : d\lambda_{23}(x,y)\\
    &= \lim_{h \to 0} \frac{1}{2h} \sum_{\alpha,\beta=1,2} \left( \int_{\Omega} \partial_{x_\alpha}u^h_3(x)\,\chi^{(1)}(x)\partial_{y_\beta y_\beta}G_\alpha\hspace{-0.25em}\left(\tfrac{x'}{\epsh}\right) \,dx + \int_{\Omega} \partial_{x_3}u^h_\alpha(x)\,\chi^{(1)}(x)\partial_{y_\beta y_\beta}G_\alpha\hspace{-0.25em}\left(\tfrac{x'}{\epsh}\right) \,dx \right).
\end{align}
Suppose now that $\div_{y}\chi^{(2)}_{3\alpha} = 0$, i.e. $\sum_{\alpha,\beta=1,2} \partial_{y_\alpha y_\beta y_\beta}G_\alpha = 0$. 
Then we have
{\allowdisplaybreaks
\begin{align} \label{limit 3,alfa pt1 - regime inf}
    \nonumber&\lim_{h \to 0} \frac{1}{2h} \sum_{\alpha,\beta=1,2} \int_{\Omega} \partial_{x_\alpha}u^h_3(x)\,\chi^{(1)}(x)\partial_{y_\beta y_\beta}G_\alpha\hspace{-0.25em}\left(\tfrac{x'}{\epsh}\right) \,dx\\
    \nonumber&= \lim_{h \to 0} \frac{1}{2h} \sum_{\alpha,\beta=1,2} \left( - \int_{\Omega} u^h_3(x)\,\partial_{x_\alpha}\chi^{(1)}(x)\partial_{y_\beta y_\beta}G_\alpha\hspace{-0.25em}\left(\tfrac{x'}{\epsh}\right) \,dx - \frac{1}{\epsh} \int_{\Omega} u^h_3(x)\,\chi^{(1)}(x)\partial_{y_\alpha y_\beta y_\beta}G_\alpha\hspace{-0.25em}\left(\tfrac{x'}{\epsh}\right) \,dx \right)\\
    \nonumber&= - \lim_{h \to 0} \frac{1}{2h} \sum_{\alpha,\beta=1,2} \int_{\Omega} u^h_3(x)\,\partial_{x_\alpha}\chi^{(1)}(x)\partial_{y_\beta y_\beta}G_\alpha\hspace{-0.25em}\left(\tfrac{x'}{\epsh}\right) \,dx\\
    \nonumber&= \lim_{h \to 0} \frac{\epsh}{2h} \sum_{\alpha,\beta=1,2} \left( \int_{\Omega} \partial_{x_\beta}u^h_3(x)\,\partial_{x_\alpha}\chi^{(1)}(x)\partial_{y_\beta}G_\alpha\hspace{-0.25em}\left(\tfrac{x'}{\epsh}\right) \,dx + \int_{\Omega} u^h_3(x)\,\partial_{x_\alpha x_\beta}\chi^{(1)}(x)\partial_{y_\beta}G_\alpha\hspace{-0.25em}\left(\tfrac{x'}{\epsh}\right) \,dx \right)\\
    \nonumber&= \lim_{h \to 0} \frac{\epsh}{2h} \sum_{\alpha,\beta=1,2} \int_{\Omega} \partial_{x_\beta}u^h_3(x)\,\partial_{x_\alpha}\chi^{(1)}(x)\partial_{y_\beta}G_\alpha\hspace{-0.25em}\left(\tfrac{x'}{\epsh}\right) \,dx\\
    \nonumber& = - \lim_{h \to 0} \frac{\epsh}{2h} \sum_{\alpha,\beta=1,2} \int_{\Omega} \partial_{x_3}u^h_\beta(x)\,\partial_{x_\alpha}\chi^{(1)}(x)\partial_{y_\beta}G_\alpha\hspace{-0.25em}\left(\tfrac{x'}{\epsh}\right) \,dx\\
    \nonumber& = \lim_{h \to 0} \frac{\epsh}{2h} \sum_{\alpha,\beta=1,2} \int_{\Omega} u^h_\beta(x)\,\partial_{x_\alpha x_3}\chi^{(1)}(x)\partial_{y_\beta}G_\alpha\hspace{-0.25em}\left(\tfrac{x'}{\epsh}\right) \,dx\\
    & = 0.
\end{align}
}
Furthermore,
{\allowdisplaybreaks
\begin{align} \label{limit 3,alfa pt2 - regime inf} 
    \nonumber&\lim_{h \to 0} \frac{1}{2h} \sum_{\alpha,\beta=1,2} \int_{\Omega} \partial_{x_3}u^h_\alpha(x)\,\chi^{(1)}(x)\partial_{y_\beta y_\beta}G_\alpha\hspace{-0.25em}\left(\tfrac{x'}{\epsh}\right) \,dx\\
    \nonumber&= - \lim_{h \to 0} \frac{1}{2h} \sum_{\alpha,\beta=1,2} \int_{\Omega} u^h_\alpha(x)\,\partial_{x_3}\chi^{(1)}(x)\partial_{y_\beta y_\beta}G_\alpha\hspace{-0.25em}\left(\tfrac{x'}{\epsh}\right) \,dx\\
    \nonumber&= \lim_{h \to 0} \frac{\epsh}{2h} \sum_{\alpha,\beta=1,2} \left( \int_{\Omega} \partial_{x_\beta}u^h_\alpha(x)\,\partial_{x_3}\chi^{(1)}(x)\partial_{y_\beta}G_\alpha\hspace{-0.25em}\left(\tfrac{x'}{\epsh}\right) \,dx + \int_{\Omega} u^h_\alpha(x)\,\partial_{x_\beta x_3}\chi^{(1)}(x)\partial_{y_\beta}G_\alpha\hspace{-0.25em}\left(\tfrac{x'}{\epsh}\right) \,dx \right)\\ 
    \nonumber&= \lim_{h \to 0} \frac{\epsh}{2h} \sum_{\alpha,\beta=1,2} \int_{\Omega} \partial_{x_\beta}u^h_\alpha(x)\,\partial_{x_3}\chi^{(1)}(x)\partial_{y_\beta}G_\alpha\hspace{-0.25em}\left(\tfrac{x'}{\epsh}\right) \,dx\\
    \nonumber&= - \lim_{h \to 0} \frac{\epsh}{2h} \sum_{\alpha,\beta=1,2} \int_{\Omega} \partial_{x_\alpha}u^h_\beta(x)\,\partial_{x_3}\chi^{(1)}(x)\partial_{y_\beta}G_\alpha\hspace{-0.25em}\left(\tfrac{x'}{\epsh}\right) \,dx\\
    \nonumber&= \lim_{h \to 0} \frac{\epsh}{2h} \sum_{\alpha,\beta=1,2} \left( \int_{\Omega} u^h_\beta(x)\,\partial_{x_\alpha x_3}\chi^{(1)}(x)\partial_{y_\beta}G_\alpha\hspace{-0.25em}\left(\tfrac{x'}{\epsh}\right) + \frac{1}{\epsh} \int_{\Omega} u^h_\beta(x)\,\partial_{x_3}\chi^{(1)}(x)\partial_{y_\alpha y_\beta}G_\alpha\hspace{-0.25em}\left(\tfrac{x'}{\epsh}\right) \,dx \right)\\
    \nonumber&= \lim_{h \to 0} \frac{1}{2h} \sum_{\alpha,\beta=1,2}  \int_{\Omega} u^h_\beta(x)\,\partial_{x_3}\chi^{(1)}(x)\partial_{y_\alpha y_\beta}G_\alpha\hspace{-0.25em}\left(\tfrac{x'}{\epsh}\right) \,dx\\
    \nonumber&= - \lim_{h \to 0} \frac{\epsh}{2h} \sum_{\alpha,\beta=1,2} \left( \int_{\Omega} \partial_{x_\beta}u^h_\beta(x)\,\partial_{x_3}\chi^{(1)}(x)\partial_{y_\alpha}G_\alpha\hspace{-0.25em}\left(\tfrac{x'}{\epsh}\right) \,dx + \int_{\Omega} u^h_\beta(x)\,\partial_{x_\beta x_3}\chi^{(1)}(x)\partial_{y_\alpha}G_\alpha\hspace{-0.25em}\left(\tfrac{x'}{\epsh}\right) \,dx \right)\\
    &= 0.
\end{align}
}
From \eqref{limit 3,alfa - regime inf}, \eqref{limit 3,alfa pt1 - regime inf} and \eqref{limit 3,alfa pt2 - regime inf}, and \Cref{duality lemma 2 - regime inf}, and recalling that $\int_{\calY} \chi^{(2)}_{13} \,dy = 0$ and $\int_{\calY} \chi^{(2)}_{23} \,dy = 0$, we conclude that there exist $\kappa \in \calYinf{\Omega}$ and $\zeta' \in \Mb(\Omega;\R^2)$ such that
\begin{equation*}
    \begin{pmatrix} \lambda_{13} \\ \lambda_{23} \end{pmatrix} = \zeta' \otimes \calL^{2}_y + D_{y}\kappa.
\end{equation*}
\EEE This concludes the proof of the theorem. \BBB
\end{proof}

\section{Two-scale statics and duality}
\label{statics}
In this section we define a notion of stress-strain duality and analyze the two-scale behavior of our functionals.
\CCC The main goal is to prove the principle of maximum plastic work in \Cref{principleofmximumpl}, which we will use in \Cref{dynamics} to prove the global stability of the limiting model. 
In \Cref{torusstresstrain} we \EEE characterize \BBB the duality between stress and strain on the torus $\calY$, the admissible two-scale configurations are discussed in \Cref{subs:dis}, while the admissible two-scale stresses are \EEE the subject of
\Cref{admissiblestress}.
\BBB

\subsection{Stress-plastic strain duality on the cell}\label{torusstresstrain}

\subsubsection{Case \texorpdfstring{$\gamma = 0$}{γ = 0}}
\begin{definition} \label{definition K_0}
The set $\calK_{0}$ of admissible stresses is defined as the set of all elements $\Sigma \in L^2(I \times \calY;\M^{3 \times 3}_{\sym})$ satisfying:
\begin{enumerate}[label=(\roman*)]
    \item \label{definition K_0 (i)} $\Sigma_{i3}(x_3,y) = 0 \,\text{ for } i=1,2,3$,
    \item \label{definition K_0 (ii)} $\Sigma_{\dev}(x_3,y) \in K(y) \,\text{ for } \calL^{1}_{x_3} \otimes \calL^{2}_{y}\text{-a.e. } (x_3,y) \in I \times \calY$,
    \item \label{definition K_0 (iii)} $\div_{y}\bar{\Sigma} = 0 \text{ in } \calY$,
    \item \label{definition K_0 (iv)} $\div_{y}\div_{y}\hat{\Sigma} = 0 \text{ in } \calY$,
\end{enumerate}
where $\bar{\Sigma},\, \hat{\Sigma} \in L^2(\calY;\M^{2 \times 2}_{\sym})$ are the \CCC zero-th \BBB and \CCC the \BBB first order moments of the $2 \times 2$ minor of $\Sigma$.
\end{definition}

Recalling \eqref{K_r characherization}, by conditions \ref{definition K_0 (i)} and \ref{definition K_0 (ii)} we may identify $\Sigma \in \calK_{0}$ with an element of $L^{\infty}(I \times \calY;\M^{2 \times 2}_{\sym})$ such that $\Sigma(x_3,y) \in K_r(y) \,\text{ for } \calL^{1}_{x_3} \otimes \calL^{2}_{y}\text{-a.e. } (x_3,y) \in I \times \calY$.
Thus, in this regime it will be natural to define the family of admissible configurations \EEE by means of conditions formulated on \BBB $\M^{2 \times 2}_{\sym}$ .

\begin{definition} \label{definition A_0}
The family $\calA_{0}$ of admissible configurations is given by the set of quadruplets
\begin{equation*}
    \bar{u} \in BD(\calY), \qquad u_3 \in BH(\calY), \qquad E \in L^2(I \times \calY;\M^{2 \times 2}_{\sym}), \qquad P \in \Mb(I \times \calY;\M^{2 \times 2}_{\sym}),
\end{equation*}
such that
\begin{equation} \label{disss1}
    E_{y}\bar{u} - x_3 D^2_{y}u_3 = E \,\calL^{1}_{x_3} \otimes \calL^{2}_{y} + P \quad \textit{ in } I \times \calY.
\end{equation}
\end{definition}

Recalling the definitions of \CCC zero-th \BBB and first order moments of functions and measures (see \Cref{moments of functions} and \Cref{moments of measures}), we introduce the following analogue of the duality between moments of stresses and plastic strains.

\begin{definition}
Let $\Sigma \in \calK_{0}$ and let 
$(\bar{u}, u_3, E, P) \in \calA_{0}$.
We define the distributions $[ \bar{\Sigma} : \bar{P} ]$ and $[ \hat{\Sigma} : \hat{P} ]$ on $\calY$ by
\begin{gather}
\begin{split} \label{cell stress-strain duality zero moment - regime zero}
    [ \bar{\Sigma} : \bar{P} ](\varphi) :=
    - \int_{\calY} \varphi\,\bar{\Sigma} : \bar{E} \,dy 
    - \int_{\calY} \bar{\Sigma} : \big( \bar{u} \odot \nabla_{y}\varphi \big) \,dy,
\end{split}\\
\begin{split} \label{cell stress-strain duality first moment - regime zero}
    [ \hat{\Sigma} : \hat{P} ](\varphi) :=
    - \int_{\calY} \varphi\,\hat{\Sigma} : \hat{E} \,dy 
    + 2 \int_{\calY} \hat{\Sigma} : \big( \nabla_{y}u_3 \odot \nabla_{y}\varphi \big) \,dy 
    + \int_{\calY} u_3\,\hat{\Sigma} : \nabla^2_{y}\varphi \,dy,
\end{split}
\end{gather}
for every $\varphi \in C^{\infty}(\calY)$.
\end{definition}

\begin{remark}\label{remadd1} 
Note that the second integral in \eqref{cell stress-strain duality zero moment - regime zero} is well defined since $BD(\calY)$ is embedded into $L^2(\calY;\R^2)$. 
Similarly, the second and third integrals in \eqref{cell stress-strain duality first moment - regime zero} are well defined since $BH(\calY)$ is embedded into $H^1(\calY)$. 
Moreover, the definitions are independent of the choice of $(u, E)$, so \eqref{cell stress-strain duality zero moment - regime zero} and \eqref{cell stress-strain duality first moment - regime zero} define a meaningful distributions on $\calY$ \CCC (this is valid for arbitrary $\bar{\Sigma}, \hat{\Sigma} \in L^\infty(\calY;\mathbb{M}_{\sym}^{2 \times 2})$ that satisfy the properties (iii) and (iv) of Definition \ref{definition K_0}) \BBB.
\CCC
Arguing as in \cite[Section 7]{Davoli.Mora.2013}, one can prove that $[ \bar{\Sigma} : \bar{P} ]$ and $[ \hat{\Sigma} : \hat{P} ]$ are bounded Radon measures on $\calY$. 
For $\bar{\Sigma}$ of class $C^1$ and $\hat{\Sigma}$ of class $C^2$ it can be shown by integration by parts (see e.g. \cite{Francfort.Giacomini.2012} and \cite[Remark 7.1, Remark 7.4]{Davoli.Mora2015} that 
\begin{equation} \label{dualityadded1} 
\int_{\calY} \varphi d[\bar{\Sigma}:\bar P]=\int_{\calY} \varphi \bar{\Sigma} d\bar{P}, \quad \int_{\calY} \varphi d[\hat{\Sigma}:\hat P]=\int_{\calY} \varphi \hat{\Sigma} d\hat{P}. 
\end{equation} 
From this it follows that for $\bar{\Sigma}$ of class $C^1$ and $\hat{\Sigma}$ of class $C^2$ we have 
\begin{equation} \label{dualityadded2} 
 \left| [\bar{\Sigma}:\bar{P}]\right| \leq \|\bar{\Sigma}\|_{L^\infty}|\bar{P}|, \quad   \left| [\hat{\Sigma}:\hat{P}]\right| \leq \|\hat{\Sigma}\|_{L^\infty}|\hat{P}|, \quad \varphi \in C(\calY).  
 \end{equation} 
Through the approximation by convolution \eqref{dualityadded1} then extends to arbitrary continuous $\bar{\Sigma}$, $\hat{\Sigma}$ and  \eqref{dualityadded2} applies to arbitrary $\bar{\Sigma}, \hat{\Sigma} \in L^\infty(\calY;\mathbb{M}_{\sym}^{2 \times 2})$ satisfying the properties (iii) and (iv) of Definition \ref{definition K_0})
\BBB
\end{remark}
\CCC
\begin{remark} \label{remadd2} 
If $\alpha$ is a simple $C^2$ curve in $\calY$, then  
\begin{equation} \label{dualityadded3} 
[\bar{\Sigma}:\bar{P}]= \Sigma \nu^1_{\alpha}\cdot (\bar{u}_1-\bar{u}_2)\calH^1,
\end{equation} 
where $\nu^1_{\alpha}$ is a unit normal on the curve $\alpha$ while $\bar{u}_1$ and $\bar{u}_2$ are the traces on $\alpha$ of $\bar{u}$ ($\bar{u}_1$ is from the side toward which normal is pointing, $\bar{u}_2$ is from the opposite side). This can be obtained from \eqref{dualityadded1} and approximation by convolution, see e.g. \cite[Lemma 3.8]{Francfort.Giacomini.2012}. 

From \eqref{traces of the stress} it follows that if $U$ is an open set in $\calY$ whose boundary is of class $C^2$ and $\bar{\Sigma}_n \in  L^\infty(U;\mathbb{M}_{\sym}^{2 \times 2})$ a bounded sequence such that $\bar{\Sigma}_n \to \bar{\Sigma}$ almost everywhere (and thus in $L^p(U)$, for every $p<\infty$) and $\div_y \bar{\Sigma}_n\to 0$ strongly in $L^2(U)$, then $\bar{\Sigma}_n \nu^1_{\alpha} \weakstar \bar{\Sigma}\nu_{\alpha}^1$, weakly* in $L^{\infty}(K \cap \alpha)$ for any compact set $K \subset U$.
\end{remark} 
\begin{remark} 
It can be shown that if $\alpha \subset \calY$ is simple $C^2$ closed or non-closed   $C^2$ curve with endpoints $\{a,b\}$  that there exists $b_1(\hat{\Sigma}) \in \III L^{\infty}_{\rm loc} \BBB (\alpha) $ such that
\begin{equation} \label{curve1} 
[\hat{\Sigma}: \hat P]=b_1(\hat{\Sigma})\partial_{ \nu_{\alpha} }u_3^{1,2}\calH^1, \quad \textrm{ on } \alpha,   
\end{equation} 
where $\nu_{\alpha}$ is a unit normal of $\alpha$ and $\partial_{\nu_{\alpha} }u_3^{1,2}$ is a jump in the normal derivative of $u_3$ (from the side in the opposite direction of the normal), \III which is an $L^1_{\rm loc} (\alpha)$ function. \BBB This is a direct consequence of \eqref{cell stress-strain duality first moment - regime zero} and \cite[Th\'{e}oreme 2]{Demengel.1984}, see also \cite[Remark 7.4]{Davoli.Mora2015} and the fact that
$\left|[\hat{\Sigma}: \hat{P}]\right| \{a,b\}=0$ (see \ref{dualityadded2}).

From \cite[Th\'{e}oreme 2 and Appendice, Th\'{e}oreme 1]{Demengel.1984} it follows that if $U$ is an open set in $\calY$ whose boundary is of class $C^2$ and $\hat{\Sigma}_n \in  L^\infty(U;\mathbb{M}_{\sym}^{2 \times 2})$ a bounded sequence such that $\hat{\Sigma}_n \to \hat{\Sigma}$ almost everywhere (and thus in $L^p(U)$, for every $p<\infty$) and $\div_y \div_y \hat{\Sigma}_n\to 0$ strongly in $L^2(U)$, then $b_1(\hat{\Sigma}_n) \weakstar b_1(\hat{\Sigma})$, weakly* in $L^{\infty}(K \cap \alpha)$ for any compact set $K \subset U$. 
\end{remark}
\BBB
We are now in a position to introduce a duality pairing between admissible stresses and plastic strains. 

\begin{definition}
Let $\Sigma \in \calK_{0}$ and let 
$(\bar{u}, u_3, E, P) \in \calA_{0}$. 
Then we can define a bounded Radon measure $[ \Sigma : P ]$ on $I \times \calY$ by setting
\begin{equation*}
    [ \Sigma : P ] := 
    [ \bar{\Sigma} : \bar{P} ] \otimes \calL^{1}_{x_3} 
    + \frac{1}{12} [ \hat{\Sigma} : \hat{P} ] \otimes \calL^{1}_{x_3} 
    - \Sigma^\perp : E^\perp,
\end{equation*}
so that
\begin{align} \label{cell stress-strain duality - regime zero}
\begin{split}
    \int_{I \times \calY} \varphi \,d[ \Sigma : P ] 
    &= - \int_{I \times \calY} \varphi\,\Sigma : E \,dx_3 dy 
    - \int_{\calY} \bar{\Sigma} : \big( \bar{u} \odot \nabla_{y}\varphi \big) \,dy\\
    &\,\quad + \frac{1}{6} \int_{\calY} \hat{\Sigma} : \big( \nabla_{y}u_3 \odot \nabla_{y}\varphi \big) \,dy 
    + \frac{1}{12} \int_{\calY} u_3\,\hat{\Sigma} : \nabla^2_{y}\varphi \,dy,
\end{split}
\end{align}
for every $\varphi \in C^2(\calY)$.
\end{definition}
\CCC 
\begin{remark} \label{remmarin1} 
Notice that 
$$  \overline{[ \Sigma : P ]} := 
    [ \bar{\Sigma} : \bar{P} ] 
    + \frac{1}{12} [ \hat{\Sigma} : \hat{P} ] 
    - \overline{\Sigma^\perp : E^\perp}. $$
\end{remark}

The following proposition will be used in \Cref{principleofmximumpl} to prove the main result of this section.   \BBB
\begin{proposition} \label{cell Hill's principle - regime zero}
Let $\Sigma \in \calK_{0}$ and $(\bar{u}, u_3, E, P) \in \calA_{0}$. 
\CCC If $\calY$ is a geometrically admissible multi-phase torus, under the assumption on the ordering of the phases we have \BBB 
\CCC
\begin{equation} \label{inequality1}
    \overline{\,H_r\left(y, \frac{dP}{d|P|}\right) |P|} \geq \overline{ [\Sigma : P ]}.
\end{equation}
\BBB
\end{proposition}

\begin{proof}
\CCC
The proof is divided into two steps.

\noindent{\bf Step 1.} 
In this step we consider a phase $\calY_i$ for arbitrary $i$. 
 
Regularizing $\Sigma$ just by convolution with respect to $y$, we obtain a sequence $\{\Sigma_n\}$ satisfying
\begin{gather*}
    \Sigma_n \strong \Sigma \quad \text{strongly in } L^2(I \times \calY;\M^{2 \times 2}_{\sym}),\\
    \div_{y}\bar{\Sigma}_n = 0,\\
    \div_{y}\div_{y}\hat{\Sigma}_n = 0.
\end{gather*}
We also have that for every $\varepsilon>0$ there exists $n(\varepsilon)$ large enough such that $(\Sigma_n(x_3,y))_{\dev} \in K_i$ for a.e. $x_3 \in I$ and every $y \in \calY_i$ that are distanced from $\partial\calY_i$ more than $\varepsilon$, for every $n\geq n(\eps)$. 
Consider the the orthogonal decomposition
\begin{equation*}
    P = \bar{P} \otimes \calL^{1}_{x_3} + \hat{P} \otimes x_3 \calL^{1}_{x_3} + P^\perp,
\end{equation*}
where $\bar{P}, \hat{P} \in \Mb(\calY;\M^{2 \times 2}_{\sym})$ and $P^\perp \in L^2(I \times \calY;\M^{2 \times 2}_{\sym})$. 
We infer that 
$|P|$ is absolutely continuous with respect to the measure
\begin{equation*}
    \Pi := |\bar{P}| \otimes \calL^{1}_{x_3} + |\hat{P}| \otimes \calL^{1}_{x_3} + \calL^{3}_{x_3,y}.
\end{equation*}
As a consequence, for $|\Pi|$-a.e. $(x_3,y) \in I \times \calY_i$ such that $\dist(y,\partial \calY_i) >\eps$ we have
\begin{equation*}
    H_r\left(y, \frac{dP}{d|\Pi|}\right) \geq \Sigma_n : \frac{dP}{d|\Pi|},
\end{equation*}
for every $n \geq n(\varepsilon)$.
Thus for every $\varphi \in C_c(\calY_i)$, such that $\varphi \geq 0$, we obtain
\begin{align*}
    &\int_{I \times \calY_i} \varphi(y)\,H_r\left(y, \frac{dP}{d|P|}\right) \,d|P| 
    = \int_{I \times \calY_i} \varphi\,H_r\left(y, \frac{dP}{d|\Pi|}\right) \,d|\Pi|\\
    &\geq \int_{I \times \calY_i} \varphi\,\Sigma_n : \frac{dP}{d|\Pi|} \,d|\Pi|
    = \int_{I \times \calY_i} \varphi\,\Sigma_n : \frac{dP}{d|P|} \,d|P|
    = \int_{I \times \calY_i} \varphi \,d[ \Sigma_n : P ],
\end{align*}
for $n$ large enough. 
Since $\bar{\Sigma}_n$, $\hat{\Sigma}_n$ and $(\Sigma_n)^\perp$ are smooth with respect to  $y$, from \eqref{cell stress-strain duality zero moment - regime zero}, \eqref{cell stress-strain duality first moment - regime zero} and \eqref{dualityadded2} we conclude that
\begin{gather*}
    [\bar{\Sigma}_n : \bar{P}] \weakstar [\bar{\Sigma} : \bar{P}] \quad \text{weakly* in $\Mb(\calY)$},\\
    [\hat{\Sigma}_n : \hat{P}] \weakstar [\hat{\Sigma} : \hat{P}] \quad \text{weakly* in $\Mb(\calY)$},\\
    \int_{I \times \calY_i} \varphi\, (\Sigma_n)^\perp : P^\perp \,dx_3 dy \to \int_{I \times \calY_i} \varphi \, (\Sigma)^\perp : P^\perp \,dx_3 dy.
\end{gather*}
Passing to the limit, we have
\begin{equation*}
    \int_{I \times \calY_i} \varphi(y)\,H_r\left(y, \frac{dP}{d|P|}\right) \,d|P| \geq \int_{I \times \calY_i} \varphi \,d[ \Sigma : P ].
\end{equation*}
This proves \eqref{inequality1} on every phase. 
\\
\noindent{\bf Step 2.}
In this step we consider a curve $\alpha$ that is of class $C^2$ (together with its possible endpoints) and that is the connected component of $\Gamma \backslash S$. The points on $\alpha$
(with the exception of the possible endpoints) belong to the intersection of the boundary of exactly two phases $\partial \calY_i \cap \partial \calY_j$.
From the assumption on the ordering of the phases, without loss of generality we can assume that $K_i \subset K_j$. By  \eqref{disss1} (cf. \Cref{A_KL characherization}) as well as \EEE by the continuity of $u_3$, we find \CCC  
\begin{eqnarray} \label{disss2} & &\bar{P}=(\bar{u}_j-\bar{u}_i)\odot \nu^i_{\alpha}\mathcal{H}^1, \quad  \hat{P}=(\nabla u_3^i-\nabla u_3^j) \odot \nu_{\alpha}^i\mathcal{H}^1=\partial_{\nu^i_\alpha}u_3^{i,j} \nu^i_{\alpha} \odot \nu^i_{\alpha} \calH^1 \quad \textrm{ on }  \alpha, \\ & &
\hspace{+4ex} P= \bar{P}+x_3 \hat{P}, \quad \textrm{ on } \alpha, 
\end{eqnarray} 
where $\bar{u}_i$, $\bar{u}_j$ are traces of $\bar{u}$ on $\alpha$ from $\calY_i$ and $\calY_j$ respectively and $\partial_{\nu^i_\alpha}u_3^{i,j}$ is a jump in the normal derivative of $u_3$. 
From \eqref{dualityadded3} and \eqref{curve1} (cf. \Cref{remmarin1}) we deduce 
\begin{equation} 
\label{dualcurve} 
\overline{[\Sigma:P]}= \left( \Sigma\nu^i_{\alpha}\cdot (\bar{u}_j-\bar{u}_i)+b_1(\hat{\Sigma}) \partial_{\nu^i_\alpha}u_3^{i,j} \right)\calH^1, \quad \textrm{ on } \alpha. 
\end{equation}
Since, for each $i$, $\calY_i$ is a bounded open set with piecewise $C^2$ boundary (in particular, with Lipschitz boundary) by \cite[Proposition 2.5.4]{Carbone.DeArcangelis.2002} there exists a finite open covering $\{\calU^{(i)}_k\}$ of $\closure{\calY}_i$ such that $\calY_i \cap \calU^{(i)}_k$ is (strongly) star-shaped with Lipschitz boundary (the construction is simple and those $\calU^{(i)}_k$ that intersect the  boundary have cylindrical form up to rotation). We take only those members of the covering that have non-empty intersection with $\alpha$. We can easily modify these cylindrical sets $\calY_i \cap \calU^{(i)}_k$ to be of class $C^2$. 
 Let $\{\psi^{(i)}_k\}$ be a  partition of unity of $\alpha$ subordinate to the covering $\{\calU^{(i)}_k\}$, i.e. $\psi^{(i)}_k \in C(\alpha)$, with $0 \leq \psi^{(i)}_k\leq 1$, such that $\supp(\psi^{(i)}_k) \subset \calU^{(i)}_k$ and $\sum_{k} \psi^{(i)}_k = 1$ on $\alpha$ and let \III $\varphi \in C_c(\alpha)$ \BBB be an arbitrary non-negative function. For each $k$ we define an approximation of the stress $\Sigma$ on $\calY_i \cap \calU^{(i)}_k$ 
by
\begin{equation} \label{prop4}
   \Sigma^{(i)}_{n,k}(x_3,y) := \big(\left( \Sigma \circ d_{n,k}^{(i)} \right)(x_3,\cdot) \ast \rho_{\frac{1}{n+1}}\big)(y),
\end{equation}
where $d^{(i)}_{n,k}(x_3,y) = \left( x_3,\tfrac{n}{n+1}(y-y_k^{(i)})+y_k^{(i)}\right)$ and $y_k^{(i)}$ is the point with respect to which $\calY_i \cap \calU^{(i)}_k$ is star shaped. Obviously one has for every $k$ 
\begin{enumerate}[label=(\roman*)]
\item $\Sigma_{n,k}^{(i)} \in (K_i)_r$ for $|\Pi|$-a.e. $(x_3,y) \in I \times (\overline{\calY}_i\cap \calU^{(i)}_k)$,  
\item $\|\Sigma^{(i)}_{n,k}\|_{L^\infty} \leq \|\Sigma\|_{L^{\infty}(\calY_i \cap \calU^{(i)}_k)} $, 
\item $\Sigma^{(i)}_{n,k} \to \Sigma$, $\bar{\Sigma}^{(i)}_{n,k} \to \bar{\Sigma}$, $\hat{\Sigma}^{(i)}_{n,k} \to \hat{\Sigma}$ strongly in $L^2(\overline{\calY}_i\cap \calU^{(i)}_k;\mathbb{M}_{\sym}^{2 \times 2})$,
\item $\div_y \bar{\Sigma}_{n,k}^{(i)}=0$, $\div_y \div_y\hat{\Sigma}_{n,k}^{(i)}=0$.  
\end{enumerate}
From (i)-(iv) and by using Remark \ref{remadd1}, Remark \ref{remadd2} and \eqref{dualcurve} we conclude  for every $k$
\begin{align*}
    \int_{I \times \alpha} \psi_k^{i}(y)\varphi(y)\,H_r\left(y, \frac{dP}{d|P|}\right) \,d|P| 
    &= \int_{I \times \alpha} \psi_k^{i}(y) \varphi(y)\, H_r\left(y, \frac{dP}{d|\Pi|}\right) \,d|\Pi|\\
    &\geq \int_{I \times \alpha} \psi_k^{i} \varphi\,\Sigma^{(i)}_{n,k} : \frac{dP}{d|\Pi|} \,d|\Pi|\\
    &= \int_{\alpha} \psi_k^{i} \varphi\,\left( \Sigma^{(i)}_{n,k}\nu^i_{\alpha}\cdot (\bar{u}_j-\bar{u}_i)+b_1(\hat{\Sigma}^{(i)}_{n,k}) \partial_{\nu^i_\alpha}u_3^{i,j} \right)d \mathcal{H}^1 \\
    &\to\int_{\alpha} \psi_k^{i} \varphi\,\left( \Sigma\nu^i_{\alpha}\cdot (\bar{u}_j-\bar{u}_i)+b_1(\hat{\Sigma}) \partial_{\nu^i_\alpha}u_3^{i,j} \right)d \mathcal{H}^1.
\end{align*}
By summing over $k$ we infer \eqref{inequality1} on $\alpha$.

The final claim goes by combining Step 1 and Step 2 and using the fact that both measures in \eqref{inequality1} are zero on $\mathcal{S}$ as a consequence of \eqref{disss1} and \eqref{dualityadded2}. 
\BBB
\end{proof}

\subsubsection{Case \texorpdfstring{$\gamma = +\infty$}{γ = +∞}}
\CCC
  We first define the set of admissible stresses and configurations on the torus.
 \BBB
\begin{definition} \label{definition K_inf}
The set $\calK_{\infty}$ of admissible stresses is defined as the set of all elements $\Sigma \in L^2(\calY;\M^{3 \times 3}_{\sym})$ satisfying:
\begin{enumerate}[label=(\roman*)]
    \item \label{definition K_inf (i)} $\div_{y}\Sigma = 0 \text{ in } \calY$,
    \item \label{definition K_inf (ii)} $\Sigma_{\dev}(y) \in K(y) \,\text{ for } \calL^{2}_{y}\text{-a.e. } y \in \calY$.
\end{enumerate}
\end{definition}
\CCC Notice that in (i) we neglect the third column of $\Sigma$. \BBB
\begin{definition} \label{definition A_inf}
The family $\calA_{\infty}$ of admissible configurations is given by the set of quintuplets
\begin{equation*}
    \bar{u} \in BD(\calY), \qquad u_3 \in BV(\calY), \qquad v \in \R^3, \qquad E \in L^2(\calY;\M^{3 \times 3}_{\sym}), \qquad P \in \Mb(\calY;\M^{3 \times 3}_{\dev}),
\end{equation*}
such that
\begin{equation} \label{nakexp1} 
    \begin{pmatrix} \begin{matrix} E_{y}\bar{u} \end{matrix} & v' + D_{y}u_3 \\ (v' + D_{y}u_3)^T & v_3 \end{pmatrix} = E \,\calL^{2}_{y} + P \quad \textit{ in } \calY.
\end{equation}
\end{definition}
\CCC We also define a notion of stress-strain duality on the torus. \BBB
\begin{definition} 
Let $\Sigma \in \calK_{\infty}$ and let 
$(\bar{u}, u_3, v, E, P) \in \calA_{\infty}$.
We define the distribution $[ \Sigma_{\dev} : P ]$ on $\calY$ by
\begin{align} \label{cell stress-strain duality - regime inf}
\begin{split}
    [ \Sigma_{\dev} : P ](\varphi) 
    :=& - \int_{\calY} \varphi\,\Sigma : E \,dy - \int_{\calY} \Sigma^{\prime\prime} : \big( \bar{u} \odot \nabla_{y}\varphi \big) \,dy\\
    &- 2 \int_{\calY} u_3\,\begin{pmatrix} \Sigma_{13} \\ \Sigma_{23} \end{pmatrix} \cdot \nabla_{y}\varphi \,dy \\
    &+ 2\,v' \cdot\,\int_{\calY} \varphi\,\begin{pmatrix} \Sigma_{13} \\ \Sigma_{23} \end{pmatrix} \,dy + v_3\,\int_{\calY} \varphi\,\Sigma_{33} \,dy,
\end{split}
\end{align}
for every $\varphi \in C^{\infty}(\calY)$.
\end{definition}
\begin{remark}
Note that the integrals in \eqref{cell stress-strain duality - regime inf} are well defined since $BD(\calY)$ and $BV(\calY)$ are both embedded into $L^2(\calY;\R^2)$. 
Moreover, the definition is independent of the choice of $(\bar{u}, u_3, v, E)$, so \eqref{cell stress-strain duality - regime inf} defines a meaningful distribution on $\calY$.
\end{remark}
\CCC The following proposition provides an estimate on the total variation of $[\Sigma_{dev}:P]$. As a consequence, we find that $[\Sigma_{dev}:P]$ depends indeed only on the deviatoric part of $\Sigma$. \BBB
\begin{proposition} \label{duality distribution is actually a measure - regime inf}
Let 
$\Sigma \in \calK_{\infty}$ and $(\bar{u}, u_3, v, E, P) \in \calA_{\infty}$.
Then $[ \Sigma_{\dev} : P ]$ can be extended to a bounded Radon measure on $\calY$, whose variation satisfies
\begin{equation*}
    | [ \Sigma_{\dev} : P ] | \leq \| \Sigma_{\dev} \|_{L^{\infty}(\calY;\M^{3 \times 3}_{\sym})} |P| 
    \quad \text{ in } \Mb(\calY).
\end{equation*}
\end{proposition}
\begin{proof}
Using a convolution argument we construct a sequence $\{\Sigma_n\} \subset C^{\infty}(\calY;\M^{3 \times 3}_{\sym})$ such that 
\begin{align*}
    &\Sigma_n \strong \Sigma \quad \text{strongly in } L^2(\calY;\M^{3 \times 3}_{\sym}),\\
    &\div_{y}\Sigma_n = 0 \text{ in } \calY,\\
    &\| (\Sigma_n)_{\dev} \|_{L^{\infty}(\calY;\M^{3 \times 3}_{\dev})} \leq \| \Sigma_{\dev} \|_{L^{\infty}(\calY;\M^{3 \times 3}_{\dev})}.
\end{align*}
According to the integration by parts formulas for $BD(\calY)$ and $BV(\calY)$, we have for every $\varphi \in C^1(\calY)$
\begin{eqnarray*}
    \int_{\calY} \varphi\,\div_{y}(\Sigma_n)^{\prime\prime} \cdot \bar{u} \,dy + \int_{\calY} \varphi\,(\Sigma_n)^{\prime\prime} : dE_{y}\bar{u} + \int_{\calY} (\Sigma_n)^{\prime\prime} : \big( \bar{u} \odot \nabla_{y}\varphi \big) \,dy &=& 0,\\
    \int_{\calY} \varphi\,u_3\,\div_{y}\begin{pmatrix} (\Sigma_n)_{13} \\ (\Sigma_n)_{23} \end{pmatrix} \,dy + \int_{\calY} \varphi\,\begin{pmatrix} (\Sigma_n)_{13} \\ (\Sigma_n)_{23} \end{pmatrix} \cdot dD_{y}u_3 + \int_{\calY} u_3\,\begin{pmatrix} (\Sigma_n)_{13} \\ (\Sigma_n)_{23} \end{pmatrix} \cdot \nabla_{y}\varphi \,dy &=& 0.
\end{eqnarray*}
From these two equalities, together with the above convergence and the expression in \Cref{cell stress-strain duality - regime inf}, we compute
\begin{align*}
    &[ \Sigma_{\dev} : P ](\varphi) \\
    &= \lim_n \,\Big[ - \int_{\calY} \varphi\,\Sigma_n : E \,dy - \int_{\calY} (\Sigma_n)^{\prime\prime} : \big( \bar{u} \odot \nabla_{y}\varphi \big) \,dy\\
    &\,\quad\,\quad\,\quad - 2 \int_{\calY} u_3\,\begin{pmatrix} (\Sigma_n)_{13} \\ (\Sigma_n)_{23} \end{pmatrix} \cdot \nabla_{y}\varphi \,dy + 2\,v' \cdot\,\int_{\calY} \varphi\,\begin{pmatrix} (\Sigma_n)_{13} \\ (\Sigma_n)_{23} \end{pmatrix} \,dy + v_3\,\int_{\calY} \varphi\,(\Sigma_n)_{33} \,dy \Big]\\
    &= \lim_n \,\Big[ - \int_{\calY} \varphi\,\Sigma_n : E \,dy + \int_{\calY} \varphi\,\div_{y}(\Sigma_n)^{\prime\prime} \cdot \bar{u} \,dy + \int_{\calY} \varphi\,(\Sigma_n)^{\prime\prime} : dE_{y}\bar{u}\\
    &\,\quad\,\quad\,\quad + 2 \int_{\calY} \varphi\,u_3\,\div_{y}\begin{pmatrix} (\Sigma_n)_{13} \\ (\Sigma_n)_{23} \end{pmatrix} \,dy + 2 \int_{\calY} \varphi\,\begin{pmatrix} (\Sigma_n)_{13} \\ (\Sigma_n)_{23} \end{pmatrix} \cdot dD_{y}u_3\\
    &\,\quad\,\quad\,\quad + 2\,v' \cdot\,\int_{\calY} \varphi\,\begin{pmatrix} (\Sigma_n)_{13} \\ (\Sigma_n)_{23} \end{pmatrix} \,dy + v_3\,\int_{\calY} \varphi\,(\Sigma_n)_{33} \,dy \Big]\\
    &= \lim_n \,\Big[ \int_{\calY} \varphi\,\div_{y}(\Sigma_n) \cdot \begin{pmatrix} \bar{u} \\ u_3 \end{pmatrix} \,dy + \int_{\calY} \varphi\,\Sigma_n : dP \Big] \\
    &= \lim_n \, \int_{\calY} \varphi\,(\Sigma_n)_{\dev} : dP.
\end{align*}
In view of the $L^{\infty}$-bound on $\{(\Sigma_n)_{\dev}\}$, passing to the limit yields
\begin{equation*}
    | [ \Sigma_{\dev} : P ] |(\varphi) \leq \| \Sigma_{\dev} \|_{L^{\infty}(\calY;\M^{3 \times 3}_{\sym})} \int_{\calY} |\varphi|\,d|P|,
\end{equation*}
from which the claims follow.
\end{proof}
\CCC The following proposition characterizes $[\Sigma_{\dev}:P]$ on the interface. Before the statement we recall \Cref{remnak1}\BBB
\begin{proposition} \label{duality on the interface - regime inf} 
Let $\Sigma \in \calK_{\infty}$.
\CCC Assume that $\calY$ is a geometrically admissible multi-phase torus. \BBB 
Then, for \CCC $\calH^{1}$-a.e. $y \in \partial \calY_i \cap \partial \calY_j$, \BBB
\CCC
\begin{equation} \label{duality on the interface (1) - regime inf}
    [\Sigma\iota(\nu^i)]\tangentiallll(y) \in \left((K_i\cap K_j)\iota(\nu^i)\right) \tangentiallll.
 \end{equation}
\BBB
Furthermore, if $(\bar{u}, u_3, v, E, P) \in \calA_{\infty}$, then for every $i \neq j$,
\begin{equation} \label{duality on the interface (2) - regime inf}
    [ \Sigma_{\dev} : P ]\mres{\Gamma_{ij}} = \left( 
    [\Sigma^{\prime\prime}\nu^i]\tangentiali \cdot (\bar{u}^i - \bar{u}^j) 
    + 2 \left(\begin{pmatrix} \Sigma_{13} \\ \Sigma_{23} \end{pmatrix}\cdot\CCC\nu^i \BBB\right) (u_3^i - u_3^j)
    \right)\,\calH^{1}\mres{\Gamma_{ij}},
\end{equation}
where $\bar{u}^i,\,u_3^i$ and $\bar{u}^j,\,u_3^j$ are the traces on $\Gamma_{ij}$ of the restrictions of $\bar{u},\,u_3$ to $\calY_i$ and $\calY_j$ respectively, assuming that $\nu^i$ points from $\calY_j$ to $\calY_i$.
\end{proposition}
\begin{proof}
To prove \eqref{duality on the interface (2) - regime inf}, let $\varphi \in C^1(\calY)$ be such that its support is contained in $\calY_i \cup \calY_j \cup \Gamma_{ij}$.
Let $\calU \subset\subset \calY$ be a compact set containing $\supp(\varphi)$, and consider any smooth approximating sequence $\{\Sigma_n\} \subset C^{\infty}(\calU;\M^{3 \times 3}_{\sym})$ such that 
\begin{align} \label{prop1}
    &\Sigma_n \strong \Sigma \quad \text{strongly in } L^2(\calU;\M^{3 \times 3}_{\sym}),\\ \label{prop2}
    &\div_{y}\Sigma_n = 0 \text{ in } \calU,\\ \label{prop3}
    &\| (\Sigma_n)_{\dev} \|_{L^{\infty}(\calU;\M^{3 \times 3}_{\dev})} \leq \| \Sigma_{\dev} \|_{L^{\infty}(\calU;\M^{3 \times 3}_{\dev})}.
\end{align}
Note that 
\CCC$\left((\Sigma_n)^{\prime\prime}\nu^i\right)\tangentiali = \left((\Sigma_n)_{\dev}^{\prime\prime}\nu^i\right)\tangentiali$ \BBB and 
\CCC
\begin{equation*}
    \left((\Sigma_n)_{\dev}^{\prime\prime}\nu^i\right)\tangentiali \weakstar [\Sigma_{\dev}^{\prime\prime}\nu^i]\tangentiali \quad \text{weakly* in $L^\infty(\Gamma_{ij};\R^2)$}.
\end{equation*}\BBB
Since $\varphi\,\bar{u} \in BD(\calY)$ and $\varphi\,u_3 \in BD(\calY)$, with
\begin{gather*}
    E_{y}\left(\varphi\,\bar{u}\right) = \varphi\,E_{y}\bar{u} + \bar{u} \odot \nabla_{y}\varphi,\\
    D_{y}\left(\varphi\,u_3\right) = \varphi\,D_{y}u_3 + u_3\,\nabla_{y}\varphi,
\end{gather*}
we compute \CCC using \eqref{nakexp1} \BBB 
\begin{align*}
    &[ \Sigma_{\dev} : P ](\varphi) \\
    &= \lim_n \,\Big[ - \int_{\calY_i \cup \calY_j} \varphi\,\Sigma_n : E \,dy - \int_{\calY_i \cup \calY_j} (\Sigma_n)^{\prime\prime} : \big( \bar{u} \odot \nabla_{y}\varphi \big) \,dy\\
    &\,\quad\,\quad\,\quad - 2 \int_{\calY_i \cup \calY_j} u_3\,\begin{pmatrix} (\Sigma_n)_{13} \\ (\Sigma_n)_{23} \end{pmatrix} \cdot \nabla_{y}\varphi \,dy + 2\,v' \cdot\,\int_{\calY_i \cup \calY_j} \varphi\,\begin{pmatrix} (\Sigma_n)_{13} \\ (\Sigma_n)_{23} \end{pmatrix} \,dy + v_3\,\int_{\calY_i \cup \calY_j} \varphi\,(\Sigma_n)_{33} \,dy \Big]\\
    &= \lim_n \,\Big[ - \int_{\calY_i \cup \calY_j} \varphi\,\Sigma_n : E \,dy
    - \int_{\calY_i \cup \calY_j} (\Sigma_n)^{\prime\prime} : dE_{y}\left(\varphi\,\bar{u}\right)
    + \int_{\calY_i \cup \calY_j} \varphi\,(\Sigma_n)^{\prime\prime} : E_{y}\bar{u}\\
    &\,\quad\,\quad\,\quad - 2\int_{\calY_i \cup \calY_j} \begin{pmatrix} (\Sigma_n)_{13} \\ (\Sigma_n)_{23} \end{pmatrix} \cdot  dD_{y}\left(\varphi\,u_3\right) + 2\int_{\calY_i \cup \calY_j} \varphi\,\begin{pmatrix} (\Sigma_n)_{13} \\ (\Sigma_n)_{23} \end{pmatrix} \cdot dD_{y}u_3\\
    &\,\quad\,\quad\,\quad + 2\,v' \cdot\,\int_{\calY_i \cup \calY_j} \varphi\,\begin{pmatrix} (\Sigma_n)_{13} \\ (\Sigma_n)_{23} \end{pmatrix} \,dy + v_3\,\int_{\calY_i \cup \calY_j} \varphi\,(\Sigma_n)_{33} \,dy \Big]\\
    &= \lim_n \,\Big[ - \int_{\calY_i \cup \calY_j} (\Sigma_n)^{\prime\prime} : dE_{y}\left(\varphi\,\bar{u}\right) - 2 \int_{\calY_i \cup \calY_j} \begin{pmatrix} (\Sigma_n)_{13} \\ (\Sigma_n)_{23} \end{pmatrix} \cdot dD_{y}\left(\varphi\,u_3\right) + \int_{\calY_i \cup \calY_j} \varphi\,\Sigma_n : dP \Big]. \\
 \end{align*}
Owing to the assumption on $\supp(\varphi)$, we have that the only relevant part of the boundary of $\calY_i \cup \calY_j$ is $\Gamma_{ij}$. Thus, an integration by parts yields
\begin{align*}
    &[ \Sigma_{\dev} : P ](\varphi) \\
        &= \lim_n \,\Big[ \int_{\Gamma_{ij}} \varphi\,\left((\Sigma_n)^{\prime\prime}\CCC\nu^{i}\BBB\right) \cdot (\bar{u}^i - \bar{u}^j) \,d\calH^{1} \CCC+\BBB 2 \int_{\Gamma_{ij}} \varphi\,\left(\begin{pmatrix} (\Sigma_n)_{13} \\ (\Sigma_n)_{23} \end{pmatrix}\cdot\CCC\nu^{i}\BBB\right) (u_3^i - u_3^j) \,d\calH^{1}\\
    &\,\quad\,\quad\,\quad + \int_{\calY_i \cup \calY_j} \varphi\,(\Sigma_n)_{\dev} : dP \Big].
\end{align*}
Now
\begin{equation*}
    P\mres{\Gamma_{ij}} 
    = \begin{pmatrix} \begin{matrix} E_{y}\bar{u} \end{matrix} & D_{y}u_3 \\ (D_{y}u_3)^T & 0 \end{pmatrix} \mres{\Gamma_{ij}}
    = \CCC\begin{pmatrix} \begin{matrix} (\bar{u}^i - \bar{u}^j) \odot \nu^i \end{matrix} & (u_3^i - u_3^j)\,\nu^i \\ (u_3^i - u_3^j)\,(\nu^i)^T & 0 \end{pmatrix} \,\calH^{1} \BBB
\end{equation*}
and $\tr{P} = 0$ imply that $\bar{u}^i(y) - \bar{u}^j(y) \perp \CCC \nu^i\BBB(y)$ for $\calH^{1}$-a.e. $y \in \Gamma_{ij}$.
The above computation then yields
\begin{align} \label{duality limit - regime inf}
\begin{split}
    [ \Sigma_{\dev} : P ](\varphi) 
    =& \int_{\Gamma_{ij}} \varphi\,[\Sigma^{\prime\prime}\CCC\nu^i]\tangentiali \BBB\cdot (\bar{u}^i - \bar{u}^j) \,d\calH^{1} \CCC +\BBB 2 \int_{\Gamma_{ij}} \varphi\,\left(\begin{pmatrix} \Sigma_{13} \\ \Sigma_{23} \end{pmatrix}\cdot\CCC \nu^i \BBB\right) (u_3^i - u_3^j) \,d\calH^{1}\\
    &+ \lim_n \,\int_{\calY_i \cup \calY_j} \varphi\,(\Sigma_n)_{\dev} : dP.
\end{split}
\end{align}
Defining $\lambda_n \in \Mb(\calY_i \cup \calY_j \cup \Gamma_{ij})$ as
\begin{equation*}
    \lambda_n(\varphi) := \int_{\calY_i \cup \calY_j} \varphi\,(\Sigma_n)_{\dev} : dP,
\end{equation*}
then the $L^{\infty}$-bound on $\{(\Sigma_n)_{\dev}\}$ ensures that it satisfies 
\begin{equation*}
    |\lambda_n| \leq C\,|P|\mres{(\calY_i \cup \calY_j)},
\end{equation*}
and we infer from \eqref{duality limit - regime inf} that 
\begin{equation*}
    \lambda_n \weakstar \lambda \quad \text{weakly* in $\Mb(\calY_i \cup \calY_j \cup \Gamma_{ij})$}
\end{equation*}
for a suitable $\lambda \in \Mb(\calY_i \cup \calY_j \cup \Gamma_{ij})$ with 
\begin{equation} \label{total variation bound - regime inf}
    |\lambda| \leq C\,|P|\mres{(\calY_i \cup \calY_j)},
\end{equation}
and
\CCC
\begin{align*}
    [ \Sigma_{\dev} : P ](\varphi) 
    =& \int_{\Gamma_{ij}} \varphi\,[\Sigma^{\prime\prime}\nu^i]\tangentiallll \cdot (\bar{u}^i - \bar{u}^j) \,d\calH^{1} +2 \int_{\Gamma_{ij}} \varphi\,\left(\begin{pmatrix} \Sigma_{13} \\ \Sigma_{23} \end{pmatrix}\cdot\nu^i\right) (u_3^i - u_3^j) \,d\calH^{1}\\
    &+ \lambda(\varphi).
\end{align*}
\BBB
Since \eqref{total variation bound - regime inf} implies $\lambda\mres{\Gamma_{ij}} = 0$, the result directly follows. \CCC To prove \eqref{duality on the interface (1) - regime inf}
we first notice that as a consequence of \cite[Section 1.2]{Francfort.Giacomini.2012} there holds $[\Sigma\iota(\nu^i)]\tangentiallll \in L^{\infty}(\Gamma)$. We locally approximate $\Sigma$ at every point $y \in \partial \calY_i$  by dilation and convolution as in the proof of \Cref{cell Hill's principle - regime zero}, see \eqref{prop4}, so that \EEE the approximating sequence $\{\Sigma_n\}$ consequently satisfies \eqref{prop1}-\eqref{prop3} and also $\Sigma_n \in K_i$.  Since we  have that  $[\Sigma_n\iota(\nu^i)]\tangentiallll \weakstar
[\Sigma\iota(\nu^i)]\tangentiallll$ the claim follows from the convexity of $K_i$. 
\BBB
\end{proof}
\CCC The following proposition is analogous to \Cref{cell Hill's principle - regime zero} and will also  be used in \Cref{principleofmximumpl} to prove the main result of this section. \BBB
\begin{proposition} \label{cell Hill's principle - regime inf}
Let $\Sigma \in \calK_{\infty}$ and $(\bar{u}, u_3, v, E, P) \in \calA_{\infty}$. 
If $\calY$ is a geometrically admissible multi-phase torus and the \CCC assumption on the ordering of the phases is satisfied we have \BBB
\begin{equation*}
    H\left(y, \frac{dP}{d|P|}\right)\,|P| \geq [ \Sigma_{\dev} : P ] 
    \quad \text{ in } \Mb(\calY).
\end{equation*}
\end{proposition}
\begin{proof}
To establish the stated inequality, we consider the behavior of the measures on each phase $\calY_i$ and inteface $\Gamma_{ij}$ respectively.
First, consider an opet set $\calU$ such that $\closure{\calU} \subset \calY_i$ for some $i$. 
Regularizing by convolution, we obtain a sequence $\Sigma_n \in C^\infty(\calU;\M^{3 \times 3}_{\sym})$ such that
\begin{gather*}
    \Sigma_n \strong \Sigma \quad \text{strongly in } L^2(\calU;\M^{3 \times 3}_{\sym}),\\
    \div_{y}\Sigma_n = 0 \text{ in } \calU.
\end{gather*}
Furthermore, $(\Sigma_n(y))_{\dev} \in K_i$ for every $y \in \calU$. 
As a consequence, for $|P|$-a.e. $y \in \calU$ we have
\begin{equation*}
    H\left(y, \frac{dP}{d|P|}\right) 
    = H_i\left(\frac{dP}{d|P|}\right) 
    \geq \Sigma_n : \frac{dP}{d|P|}.
\end{equation*}
Thus for every $\varphi \in C(\calU)$, such that $\varphi \geq 0$, we obtain
\begin{align*}
    \int_{\calU} \varphi\,H\left(y, \frac{dP}{d|P|}\right) \,d|P| 
    \geq \int_{\calU} \varphi\,\Sigma_n : \frac{dP}{d|P|} \,d|P|
    = \int_{\calU} \varphi \,d[ \Sigma_n : P ].
\end{align*}
Since $\Sigma_n$ is smooth, we conclude that
\begin{gather*}
    [\Sigma_n : \bar{P}] \weakstar [\Sigma : \bar{P}] \quad \text{weakly* in $\Mb(\calU)$}.
\end{gather*}
Passing to the limit we have
\begin{equation*}
    \int_{\calU} \varphi\,H\left(y, \frac{dP}{d|P|}\right) \,d|P| \geq \int_{\calU} \varphi \,d[ \Sigma : P ].
\end{equation*}
The inequality on the phase $\calY_i$ now follows by considering a collection of open subsets that increases to $\calY_i$.
Next, 
for every $i \neq j$,
\begin{equation*}
    H\left(y, \frac{dP}{d|P|}\right)\,|P|\mres{\Gamma_{ij}} 
    = \CCC\min\{H_i,H_j\}\BBB\left(\begin{pmatrix} \begin{matrix} (\bar{u}^j - \bar{u}^i) \odot \nu \end{matrix} & (u_3^j - u_3^i)\,\nu \\ (u_3^j - u_3^i)\,\nu^T & 0 \end{pmatrix}\right) \,\calH^{1}\mres{\Gamma_{ij}}.
\end{equation*}
where $\bar{u}^i,\,u_3^i$ and $\bar{u}^j,\,u_3^j$ are the traces on $\Gamma_{ij}$ of the restrictions of $\bar{u},\,u_3$ to $\calY_i$ and $\calY_j$ respectively, assuming that $\nu$ points from $\calY_j$ to $\calY_i$.
The claim then directly follows in view of Proposition \ref{duality on the interface - regime inf}.
\end{proof}
\subsection{Disintegration of admissible configurations} \label{subs:dis} 
Let $\ext{\omega} \subseteq \R^2$ be an open and bounded set such that $\omega \subset \ext{\omega}$ and $\ext{\omega} \cap \partial{\omega} = \gamma_\Dir$. 
We also denote by $\ext{\Omega} = \ext{\omega} \times I$ the associated reference domain.
In order to make sense of the duality between the two-scale limits of stresses and plastic strains, we will need to disintegrate the two-scale limits of the
kinematically admissible fields in such a way to obtain elements 
\RED
of $\calA_{0}$ and $\calA_{\infty}$, respectively. 
\BLACK
\subsubsection{Case \texorpdfstring{$\gamma = 0$}{γ = 0}}
\begin{definition} \label{definition A^hom_0}
Let $w \in H^1(\ext{\Omega};\R^3) \cap KL(\ext{\Omega})$. 
We define the class $\calA^{hom}_{0}(w)$ of admissible two-scale configurations relative to the boundary datum $w$ as the set of triplets $(u,E,P)$ with
\begin{equation*}
    u \in KL(\ext{\Omega}), \qquad E \in L^2(\ext{\Omega} \times \calY;\M^{2 \times 2}_{\sym}), \qquad P \in \Mb(\ext{\Omega} \times \calY;\M^{2 \times 2}_{\sym}),
\end{equation*}
such that
\begin{equation*}
    u = w, \qquad E = Ew, \qquad P = 0 \qquad \text{ on } (\ext{\Omega} \setminus \closure{\Omega}) \times \calY,
\end{equation*}
and also such that there exist $\mu \in \calXzero{\ext{\omega}}$, $\kappa \in \calYzero{\ext{\omega}}$ with
\begin{equation} \label{admissible two-scale configurations - regime zero}
    Eu \otimes \calL^{2}_{y} + E_{y}\mu - x_3 D^2_{y}\kappa = E \,\calL^{3}_{x} \otimes \calL^{2}_{y} + P \qquad \text{ in } \ext{\Omega} \times \calY.
\end{equation}
\end{definition}
\CCC The following lemma gives the disintegration result that will be used in the proof of \Cref{two-scale stress-strain duality - regime zero}. \BBB
\begin{lemma} \label{disintegration result - regime zero}
Let $(u,E,P) \in \calA^{hom}_{0}(w)$ with the associated $\mu \in \calXzero{\ext{\omega}}$, $\kappa \in \calYzero{\ext{\omega}}$, and let $\bar{u} \in BD(\ext{\omega})$ and $u_3 \in BH(\ext{\omega})$ be the Kirchhoff-Love components of $u$. 
Then \CCC there exists $\eta \in \Mb^+(\ext{\omega})$ such that the following disintegrations hold true:\BBB
\begin{align}
    \label{disintegration result 1 - regime zero} Eu \otimes \calL^{2}_{y} &= \left( A_1(x') + x_3 A_2(x') \right) \eta \otimes \calL^{1}_{x_3} \otimes \calL^{2}_{y},\\
    \label{disintegration result 2 - regime zero} E \,\calL^{3}_{x} \otimes \calL^{2}_{y} &= C(x') E(x,y) \,\eta \otimes \calL^{1}_{x_3} \otimes \calL^{2}_{y}\\
    \label{disintegration result 3 - regime zero} P &= \eta \genprod P_{x'}.
\end{align}
Above, $A_1, A_2 : \ext{\omega} \to \M^{2 \times 2}_{\sym}$ and $C : \ext{\omega} \to [0, +\infty]$ are respective Radon-Nikodym derivatives of $E\bar{u}$, $-D^2u_3$ and $\calL^{2}_{x'}$ with respect to $\eta$, $E(x,y)$ is a Borel representative of $E$, and $P_{x'} \in \Mb(I \times \calY;\M^{2 \times 2}_{\sym})$ for $\eta$-a.e. $x' \in \ext{\omega}$.
Furthermore, we can choose Borel maps $(x',y) \in \ext{\omega} \times \calY \mapsto \mu_{x'}(y) \in \R^2$ and $(x',y) \in \ext{\omega} \times \calY \mapsto \kappa_{x'}(y) \in \R$ such that, for $\eta$-a.e. $x' \in \ext{\omega}$,
\begin{equation} \label{disintegration result 4a - regime zero} 
    \mu = \mu_{x'}(y) \,\eta \otimes \calL^{2}_{y}, \quad E_{y}\mu = \eta \genprod E_{y}\mu_{x'},
\end{equation}
\begin{equation} \label{disintegration result 4b - regime zero} 
    \kappa = \kappa_{x'}(y) \,\eta \otimes \calL^{2}_{y}, \quad D^2_{y}\kappa = \eta \genprod D^2_{y}\kappa_{x'},
\end{equation}
where $\mu_{x'} \in BD(\calY)$, $\int_{\calY} \mu_{x'}(y) \,dy = 0$ and $\kappa_{x'} \in BH(\calY)$, $\int_{\calY} \kappa_{x'}(y) \,dy = 0$.
\end{lemma}
\CCC
\begin{proof}
The proof is a consequence of \Cref{corrector main property - regime zero} and follows along the lines of \cite[Lemma 5.8]{BDV}. 
\end{proof}
\BBB
\begin{remark} \label{admissible configurations and disintegration - regime zero}
From the above disintegration, we have that, for $\eta$-a.e. $x' \in \ext{\omega}$, 
\begin{equation*}
    E_{y}\mu_{x'} - x_3 D^2_{y}\kappa_{x'} = \left[ C(x') E(x,y) - \left( A_1(x') + x_3 A_2(x') \right) \right] \calL^{1}_{x_3} \otimes \calL^{2}_{y} + P_{x'} \quad \textit{ in } I \times \calY.
\end{equation*}
Thus, the quadruplet
\begin{equation*}
    \left( \mu_{x'}, \kappa_{x'}, \left[ C(x') E(x,y) - \left( A_1(x') + x_3 A_2(x') \right) \right], P_{x'} \right)
\end{equation*}
is an element of $\calA_{0}$.
\end{remark}
\subsubsection{Case \texorpdfstring{$\gamma = +\infty$}{γ = +∞}}
\begin{definition} \label{definition A^hom_inf}
Let $w \in H^1(\ext{\Omega};\R^3) \cap KL(\ext{\Omega})$. 
We define the class $\calA^{hom}_{\infty}(w)$ of admissible two-scale configurations relative to the boundary datum $w$ as the set of triplets $(u,E,P)$ with
\begin{equation*}
    u \in KL(\ext{\Omega}), \qquad E \in L^2(\ext{\Omega} \times \calY;\M^{3 \times 3}_{\sym}), \qquad P \in \Mb(\ext{\Omega} \times \calY;\M^{3 \times 3}_{\dev}),
\end{equation*}
such that
\begin{equation*}
    u = w, \qquad E = Ew, \qquad P = 0 \qquad \text{ on } (\ext{\Omega} \setminus \closure{\Omega}) \times \calY,
\end{equation*}
and also such that there exist $\mu \in \calXinf{\ext{\Omega}}$, $\kappa \in \calXinf{\ext{\Omega}}$, $\zeta \in \Mb(\Omega;\R^3)$ with
\begin{equation} \label{admissible two-scale configurations - regime inf}
    Eu \otimes \calL^{2}_{y} + \begin{pmatrix} \begin{matrix} E_{y}\mu \end{matrix} & \zeta' + D_{y}\kappa \\ (\zeta' + D_{y}\kappa)^T & \zeta_3 \end{pmatrix} = E \,\calL^{3}_{x} \otimes \calL^{2}_{y} + P \qquad \text{ in } \ext{\Omega} \times \calY.
\end{equation}
\end{definition}
\CCC The following lemma provides a disintegration result in this regime and will be instrumental for \CCC \Cref{two-scale stress-strain duality - regime inf}. \BBB
\begin{lemma} \label{disintegration result - regime inf}
Let $(u,E,P) \in \calA^{hom}_{\infty}(w)$ with the associated $\mu \in \calXinf{\ext{\Omega}}$, $\kappa \in \calXinf{\ext{\Omega}}$, $\zeta \in \Mb(\Omega;\R^3)$ and let $\bar{u} \in BD(\ext{\omega})$ and $u_3 \in BH(\ext{\omega})$ be the Kirchhoff-Love components of $u$. 
Then \CCC there exists $\eta \in  \Mb^+(\ext{\Omega})$ such that the  following disintegrations hold true: \BBB
\begin{align}
    \label{disintegration result 1 - regime inf} Eu \otimes \calL^{2}_{y} &= \left( A_1(x') + x_3 A_2(x') \right) \eta \otimes \calL^{2}_{y},\\
    \label{disintegration result 4c - regime inf} \zeta \otimes \calL^{2}_{y} &= z(x) \,\eta \otimes \calL^{2}_{y},\\
    \label{disintegration result 2 - regime inf} E \,\calL^{3}_{x} \otimes \calL^{2}_{y} &= C(x) E(x,y) \,\eta \otimes \calL^{2}_{y}\\
    \label{disintegration result 3 - regime inf} P &= \eta \genprod P_{x}.
\end{align}
Above, $A_1, A_2 : \ext{\omega} \to \M^{2 \times 2}_{\sym}$, $z : \ext{\omega} \to \R^3$ and $C : \ext{\Omega} \to [0, +\infty]$ are the respective Radon-Nikodym derivatives of $E\bar{u}$, $-D^2u_3$, $\zeta$ and $\calL^{3}_{x}$ with respect to $\eta$, $E(x,y)$ is a Borel representative of $E$, and $P_{x} \in \Mb(\calY;\M^{3 \times 3}_{\dev})$ for $\eta$-a.e. $x \in \ext{\Omega}$.

Furthermore, we can choose Borel maps $(x,y) \in \ext{\Omega} \times \calY \mapsto \mu_{x}(y) \in \R^2$ and $(x,y) \in \ext{\Omega} \times \calY \mapsto \kappa_{x}(y) \in \R$ such that, for $\eta$-a.e. $x \in \ext{\Omega}$,
\begin{equation} \label{disintegration result 4a - regime inf} 
    \mu = \mu_{x}(y) \,\eta \otimes \calL^{2}_{y}, \quad E_{y}\mu = \eta \genprod E_{y}\mu_{x},
\end{equation}
\begin{equation} \label{disintegration result 4b - regime inf} 
    \kappa = \kappa_{x}(y) \,\eta \otimes \calL^{2}_{y}, \quad D^2_{y}\kappa = \eta \genprod D^2_{y}\kappa_{x},
\end{equation}
where $\mu_{x} \in BD(\calY)$, $\int_{\calY} \mu_{x}(y) \,dy = 0$ and $\kappa_{x} \in BV(\calY)$, $\int_{\calY} \kappa_{x}(y) \,dy = 0$.
\end{lemma}
\CCC
\begin{proof} 
The proof builds upon \Cref{corrector main property - regime inf} and follows along \cite[Lemma 5.8]{BDV}. 
\end{proof} 
\BBB

\begin{remark} \label{admissible configurations and disintegration - regime inf}
From the above disintegration, we have that, for $\eta$-a.e. $x \in \ext{\Omega}$, 
\begin{equation*}
    \begin{pmatrix} \begin{matrix} E_{y}\mu_{x} \end{matrix} & z' + D_{y}\kappa_{x} \\ (z' + D_{y}\kappa_{x})^T & z_3 \end{pmatrix} = \left[ C(x) E(x,y) - \begin{pmatrix} A_1(x') + x_3 A_2(x') & 0 \\ 0 & 0 \end{pmatrix} \right] \calL^{2}_{y} + P_{x} \quad \textit{ in } \calY.
\end{equation*}
Thus, the quintuplet
\begin{equation*}
    \left( \mu_{x}, \kappa_{x}, z, \left[ C(x) E(x,y) - \begin{pmatrix} A_1(x') + x_3 A_2(x') & 0 \\ 0 & 0 \end{pmatrix} \right], P_{x} \right)
\end{equation*}
is an element of $\calA_{\infty}$.
\end{remark}

\subsection{Admissible stress configurations and approximations}\label{admissiblestress} 
For every $e^h \in L^2(\Omega;\M^{3 \times 3}_{\sym})$ we define $\sigma^h(x) := \C\left(\frac{x'}{\epsh}\right) \Lambda_h e^h(x)$. \CCC We \BBB introduce the set \RED of stresses for the rescaled $h$ problems: \BLACK
\begin{align*}
    \calK_h = \bigg\{\sigma^h &\in L^2(\Omega;\M^{3 \times 3}_{\sym}) : \div_{h}\sigma^h = 0 \text{ in } \Omega,\ \sigma^h\,\nu = 0 \text{ in } \partial\Omega \setminus {\closure{\Gamma}_\Dir},\\
    &\sigma^h_{\dev}(x',x_3) \in K\left(\frac{x'}{\epsh}\right) \,\text{ for a.e. } x' \in \omega,\, x_3 \in I\bigg\}.
\end{align*}
\CCC We recall some properties of the limiting stress that can be found in \cite{Davoli.Mora.2013}. \BBB\\


If we consider the weak limit $\sigma \in L^2(\Omega;\M^{3 \times 3}_{\sym})$ of the sequence $\sigma^h \in \calK_h$ as $h \to 0$, then 
$\sigma_{i3} = 0$ for $i=1,2,3$. 
Furthermore, since the uniform boundedness of the sets $K(y)$ implies that the deviatoric part of the weak limit, i.e. $\sigma_{\dev} = \sigma - \frac{1}{3} \tr{\sigma} I_{3 \times 3}$, is bounded in $L^\infty(\Omega;\M^{3 \times 3}_{\sym})$, we have that
the components $\sigma_{\alpha \beta}$ are all bounded in $L^\infty(\Omega)$.

Lastly, 
\begin{equation*}
    \div_{x'}\bar{\sigma} = 0 \text{ in } \omega,  \;\text{ and }\;  \div_{x'}\div_{x'}\hat{\sigma} = 0 \text{ in } \omega.
\end{equation*}

\RED
In the following, we further characterize the sets of two-scale limits of sequences of elastic stresses $\{\sigma^h\}$, depending on the regime.
\BLACK

\subsubsection{Case \texorpdfstring{$\gamma = 0$}{γ = 0}}
\EEE We first introduce the set of limiting two-scale stress. \BBB
\begin{definition} \label{definition K^hom_0}
The set $\calK^{hom}_{0}$ is the set of all elements $\Sigma \in L^{\infty}(\Omega \times \calY;\M^{3 \times 3}_{\sym})$ satisfying:
\begin{enumerate}[label=(\roman*)]
    \item $\Sigma_{i3}(x,y) = 0 \,\text{ for } i=1,2,3$,
    \item $\Sigma_{\dev}(x,y) \in K(y) \,\text{ for } \calL^{3}_{x} \otimes \calL^{2}_{y}\text{-a.e. } (x,y) \in \Omega \times \calY$,
    \item $\div_{y}\bar{\Sigma}(x',\cdot) = 0 \text{ in } \calY \,\text{ for a.e. } x' \in \omega$,
    \item $\div_{y}\div_{y}\hat{\Sigma}(x',\cdot) = 0 \text{ in } \calY \,\text{ for a.e. } x' \in \omega$,
    \item $\div_{x'}\bar{\sigma} = 0 \text{ in } \omega$,
    \item $\div_{x'}\div_{x'}\hat{\sigma} = 0 \text{ in } \omega$,
\end{enumerate}
where $\bar{\Sigma},\, \hat{\Sigma} \in L^{\infty}(\omega \times \calY;\M^{2 \times 2}_{\sym})$ are the zero-th and first order moments of the $2 \times 2$ minor of $\Sigma$, \
$\sigma := \int_{\calY} \Sigma(\cdot,y) \,dy$, and $\bar{\sigma},\, \hat{\sigma} \in L^{\infty}(\omega;\M^{2 \times 2}_{\sym})$ are the zero-th and first order moments of the $2 \times 2$ minor of $\sigma$.
\end{definition}
\CCC The following proposition \EEE motivates \BBB the above definition. \BBB 
\begin{proposition} \label{two-scale weak limit of admissible stress - regime zero}
Let $\{\sigma^h\}$ be a bounded family in $L^2(\Omega;\M^{3 \times 3}_{\sym})$ such that $\sigma^h \in \calK_h$ and
\begin{equation*}
    \sigma^h \weaktwoscale \Sigma \quad \text{two-scale weakly in $L^2(\Omega \times \calY;\M^{3 \times 3}_{\sym})$}.
\end{equation*}
Then $\Sigma \in \calK^{hom}_{0}$.
\end{proposition}

\begin{proof}
\CCC
Properties (v) and (vi) follow from \Cref{admissiblestress}. 

To prove (i) let $\psi \in C_c^{\infty}(\omega;C^{\infty}(\closure{I} \times \calY;\R^3))$ and consider the test function $h\,\psi\left(x,\frac{x'}{\epsh}\right)$.  
We find that
\begin{equation*}
    \nabla_{h}\left(h\,\psi\left(x,\frac{x'}{\epsh}\right)\right) = \left[\; h\,\nabla_{x'}\psi\left(x,\frac{x'}{\epsh}\right) + \frac{h}{\epsh}\,\nabla_{y}\psi\left(x,\frac{x'}{\epsh}\right) \;\right|\left.\; \partial_{x_3}\psi\left(x,\frac{x'}{\epsh}\right) \;\right]
\end{equation*}
converges strongly in $L^2(\Omega \times \calY;\M^{3 \times 3})$.
Hence, taking such a test function in \III $\div_h \sigma^h=0$ \BBB and passing to the limit, we get
\begin{equation*}
    \int_{\Omega \times \calY} \Sigma(x,y) : \left(\begin{array}{ccc} 0 & 0 & \partial_{x_3}\psi_1(x,y) \\ 0 & 0 & \partial_{x_3}\psi_2(x,y) \\ \partial_{x_3}\psi_1(x,y) & \partial_{x_3}\psi_2(x,y) & \partial_{x_3}\psi_3(x,y) \end{array} \right) \,dx dy = 0,
\end{equation*}
which is sufficient to conclude that $\Sigma_{i3}(x,y) = 0 \,\text{ for } i=1,2,3$.

To prove (ii) we define
\begin{equation} \label{approximating sequence for K^hom}
    \Sigma^h(x,y) = \sum_{i \in I_\epsh(\ext{\omega})} \charfun{Q_\epsh^i}(x')\,\sigma^h(\epsh i + \epsh\calI(y),x_3),
\end{equation}
and consider the set
\begin{equation*}
    S = \{ \Xi \in L^2(\Omega \times \calY;\M^{3 \times 3}_{\sym}) : \Xi_{\dev}(x,y) \in K(y) \,\text{ for } \calL^{3}_{x} \otimes \calL^{2}_{y}\text{-a.e. } (x,y) \in \Omega \times \calY \}.
\end{equation*}
The construction of $\Sigma^h$ from $\sigma^h \in \calK_h$ ensures that $\Sigma^h \in S$ and that $\Sigma^h \weak \Sigma \;\text{ weakly in } L^2(\Omega \times \calY;\M^{3 \times 3}_{\sym})$. 
\III (i) and (ii) imply that $\Sigma \in L^{\infty}$.\BBB

Since compactness of $K(y)$ implies that $S$ is convex and weakly closed in $L^2(\Omega \times \calY;\M^{3 \times 3}_{\sym})$, we have that $\Sigma \in S$, which concludes the proof.

Finally to prove (iii) and (iv) let $\phi \in C_c^{\infty}(\omega;C^{\infty}(\calY;\R^3))$ and consider the test function 
\begin{equation*}
    \varphi(x) 
    = \epsh \left(\begin{array}{c} \phi_1(x',\frac{x'}{\epsh}) \\ \phi_2(x',\frac{x'}{\epsh}) \\ 0 \end{array}\right) 
    + \epsh^2 \left(\begin{array}{c} - x_3\,\partial_{x_1}\phi_3(x',\frac{x'}{\epsh}) - \frac{x_3}{\epsh}\partial_{y_1}\phi_3(x',\frac{x'}{\epsh}) \\ - x_3\,\partial_{x_2}\phi_3(x',\frac{x'}{\epsh}) - \frac{x_3}{\epsh}\partial_{y_2}\phi_3(x',\frac{x'}{\epsh}) \\ \frac{1}{h}\,\phi_3(x',\frac{x'}{\epsh}) \end{array}\right).
\end{equation*}
By a direct computation we infer
\begin{equation*}
    E_{h}\varphi(x) \strong \begin{pmatrix} \begin{matrix} E_{y}\phi'(x',y) - x_3 D^2_{y}\phi_3(x',y) \end{matrix} & \begin{matrix} 0 \\ 0\end{matrix} \\ \begin{matrix} 0 & 0 \end{matrix} & 0 \end{pmatrix} \quad \text{strongly in } L^2(\Omega \times \calY;\M^{3 \times 3}_{\sym}).
\end{equation*}
Hence, taking such a test function in  \III $\div_h \sigma^h=0$ \BBB and passing to the limit, we get
\begin{equation*}
    \int_{\Omega \times \calY} \Sigma(x,y) : \begin{pmatrix} E_{y}\phi' - x_3 D^2_{y}\phi_3 & 0 \\ 0 & 0 \end{pmatrix} \,dx dy = 0.
\end{equation*}
Suppose now that $\phi\left(x',y\right) = \psi^{(1)}(x')\,\psi^{(2)}(y)$ for $\psi^{(1)} \in C_c^{\infty}(\omega)$ and $\psi^{(2)} \in C^{\infty}(\calY;\R^3)$.
Then
\begin{equation*}
    \int_{\omega} \psi^{(1)}(x') \left(\int_{I \times \calY} \Sigma(x,y) : \begin{pmatrix} E_{y}(\psi^{(2)})'(y) - x_3 D^2_{y}\psi^{(2)}_3(y) & 0 \\ 0 & 0 \end{pmatrix} \,dx_3 dy\right) \,dx' = 0,
\end{equation*}
from which we deduce that, for a.e. $x' \in \omega$,
\begin{align*}
    0 &= \int_{I \times \calY} \Sigma(x,y) : \begin{pmatrix} E_{y}(\psi^{(2)})'(y) - x_3 D^2_{y}\psi^{(2)}_3(y) & 0 \\ 0 & 0 \end{pmatrix} \,dx_3 dy\\
    &= \int_{\calY} \bar{\Sigma}(x',y) : E_{y}(\psi^{(2)})'(y) \,dy - \frac{1}{12} \int_{\calY} \hat{\Sigma}(x',y) : D^2_{y}\psi^{(2)}_3(y) \,dy\\
    &= - \int_{\calY} \div_{y}\bar{\Sigma}(x',y) \cdot (\psi^{(2)})'(y) \,dy - \frac{1}{12} \int_{\calY} \div_{y}\div_{y}\hat{\Sigma}(x',y) \cdot \psi^{(2)}_3(y) \,dy.
\end{align*}
Thus, $\div_{y}\bar{\Sigma}(x',\cdot) = 0 \text{ in } \calY$ and $\div_{y}\div_{y}\hat{\Sigma}(x',\cdot) = 0 \text{ in } \calY$.

\end{proof}
The following lemma approximates the limiting stresses with respect to the macroscopic variable and will be used in \Cref{two-scale stress-strain duality - regime zero}. It is proved under the assumption that the domain is star-shaped. 
\begin{lemma} \label{approximation of stresses - regime zero} 
Let $\omega \subset \R^2$ be an open bounded set that is star-shaped with respect to one of its points and let $\Sigma \in \calK^{hom}_{0}$. Then, there exists a sequence $\Sigma_n \in L^{\infty}(\R^2 \times I \times \calY;\M^{3 \times 3}_{\sym})$ such that the following holds:
\begin{enumerate}[label=(\alph*)]
    \item \label{approximation of stresses (a) - regime zero} $\Sigma_n \in C^\infty(\R^2;L^{\infty}(I \times \calY;\M^{3 \times 3}_{\sym}))$ and $\Sigma_n \strong \Sigma$ strongly in $L^p(\omega \times I \times \calY;\M^{3 \times 3}_{\sym})$, for $1 \leq p < +\infty$.
    \item \label{approximation of stresses (b) - regime zero} $(\Sigma_n)_{i3}(x,y) = 0 \,\text{ for } i=1,2,3$,
    \item \label{approximation of stresses (c) - regime zero} $(\Sigma_n(x,y))_{\dev} \in K(y) \,\text{ for every } x' \in \R^2 \text{ and } \calL^{1}_{x_3} \otimes \calL^{2}_{y}\text{-a.e. } (x_3,y) \in I \times \calY$,
    \item \label{approximation of stresses (d) - regime zero} $\div_{y}\bar{\Sigma}_n(x',\cdot) = 0 \text{ in } \calY \,\text{ for  every } x' \in \omega$,
    \item \label{approximation of stresses (e) - regime zero} $\div_{y}\div_{y}\hat{\Sigma}_n(x',\cdot) = 0 \text{ in } \calY \,\text{ for  every } x' \in \omega$,
\end{enumerate}
    where $\bar{\Sigma}_n,\, \hat{\Sigma}_n \in L^{\infty}(\omega \times \calY;\M^{2 \times 2}_{\sym})$ are the zero-th and first order moments of the $2 \times 2$ minor of $\Sigma_n$.
    Further, if we set $\sigma_n(x) := \int_{\calY} \Sigma_n(x,y) \,dy$, and $\bar{\sigma}_n,\, \hat{\sigma}_n \in L^{\infty}(\omega;\M^{2 \times 2}_{\sym})$ are the zero-th and first order moments of the $2 \times 2$ minor of $\sigma_n$, then:
\begin{enumerate}[label=(\alph*), resume]
    \item \label{approximation of stresses (f) - regime zero} $\sigma_n \in C^\infty(\R^2 \times I;\M^{3 \times 3}_{\sym})$ and $\sigma_n \strong \sigma$ strongly in $L^p(\omega \times I;\M^{3 \times 3}_{\sym})$, for $1 \leq p < +\infty$.
    \item \label{approximation of stresses (g) - regime zero} $\div_{x'}\bar{\sigma}_n = 0 \text{ in } \omega$,
    \item \label{approximation of stresses (h) - regime zero} $\div_{x'}\div_{x'}\hat{\sigma}_n = 0 \text{ in } \omega$.
\end{enumerate}
\end{lemma}

\begin{proof}
\CCC
The approximation is done by dilation and convolution and is analogous to \cite[Lemma 5.13]{BDV}.  
\BBB
\end{proof}

\subsubsection{Case \texorpdfstring{$\gamma = +\infty$}{γ = +∞}}
\EEE In this regime, the set of limiting two-scale stresses is defined as follows. \BBB
\begin{definition} \label{definition K^hom_inf}
The set $\calK^{hom}_{\infty}$ is the set of all elements $\Sigma \in L^2(\Omega \times \calY;\M^{3 \times 3}_{\sym})$ satisfying:
\begin{enumerate}[label=(\roman*)]
    \item $\div_{y}\Sigma(x,\cdot) = 0 \text{ in } \calY \,\text{ for  a.e. } x \in \Omega$,
    \item $\Sigma_{\dev}(x,y) \in K(y) \,\text{ for } \calL^{3}_{x} \otimes \calL^{2}_{y}\text{-a.e. } (x,y) \in \Omega \times \calY$,
    \item $\sigma_{i3}(x) = 0 \,\text{ for } i=1,2,3$,
    \item $\div_{x'}\bar{\sigma} = 0 \text{ in } \omega$,
    \item $\div_{x'}\div_{x'}\hat{\sigma} = 0 \text{ in } \omega$,
\end{enumerate}
where $\sigma := \int_{\calY} \Sigma(\cdot,y) \,dy$, and $\bar{\sigma},\, \hat{\sigma} \in L^2(\omega;\M^{2 \times 2}_{\sym})$ are the zero-th and first order moments of the $2 \times 2$ minor of $\sigma$.
\end{definition}
\EEE The previous definition is motivated by the following. \BBB
\begin{proposition} \label{two-scale weak limit of admissible stress - regime inf}
Let $\{\sigma^h\}$ be a bounded family in $L^2(\Omega;\M^{3 \times 3}_{\sym})$ such that $\sigma^h \in \calK_h$ and
\begin{equation*}
    \sigma^h \weaktwoscale \Sigma \quad \text{two-scale weakly in $L^2(\Omega \times \calY;\M^{3 \times 3}_{\sym})$}.
\end{equation*}
Then $\Sigma \in \calK^{hom}_{\infty}$.
\end{proposition}

\begin{proof}
\CCC Properties (iii), (iv) and (v) follow in view of \Cref{admissiblestress}. \BBB
\CCC To prove (i) \BBB we consider the test function $\epsh\,\phi\left(x,\frac{x'}{\epsh}\right)$, for $\phi \in C_c^{\infty}(\omega;C^{\infty}(\closure{I} \times \calY;\R^3))$.
We see that
\begin{equation*}
    \nabla_{h}\left(\epsh\,\phi\left(x,\frac{x'}{\epsh}\right)\right) = \left[\; \epsh\,\nabla_{x'}\phi\left(x,\frac{x'}{\epsh}\right) + \nabla_{y}\phi\left(x,\frac{x'}{\epsh}\right) \;\right|\left.\; \frac{\epsh}{h}\,\partial_{x_3}\phi\left(x,\frac{x'}{\epsh}\right) \;\right]
\end{equation*}
converges strongly in $L^2(\Omega \times \calY;\M^{3 \times 3})$.
Hence, taking such a test function in \III $\div_h \sigma^h=0 $ \BBB  and passing to the limit, we get
\begin{equation*}
    \int_{\Omega \times \calY} \Sigma(x,y) : E_{y}\phi\left(x,y\right) \,dx dy = 0.
\end{equation*}
Suppose now that $\phi\left(x,y\right) = \psi^{(1)}(x)\,\psi^{(2)}(y)$ for $\psi^{(1)} \in C_c^{\infty}(\omega;C^{\infty}(\closure{I}))$ and $\psi^{(2)} \in C^{\infty}(\calY;\R^3)$.
Then
\begin{equation*}
    \int_{\Omega} \psi^{(1)}(x) \left(\int_{\calY} \Sigma(x,y) : E_{y}\psi^{(2)}(y) \,dy\right) \,dx = 0,
\end{equation*}
from which we can deduce that $\div_{y}\Sigma(x,\cdot) = 0 \text{ in } \calY$ for a.e. $x \in \Omega$.

To conclude the proof, it remains to show the stress constraint $\Sigma_{\dev}(x,y) \in K(y) \,\text{ for } \calL^{3}_{x} \otimes \calL^{2}_{y}\text{-a.e. } (x,y) \in \Omega \times \calY$. 
To do this we can define the approximating sequence \eqref{approximating sequence for K^hom} and argue as in the proof of Proposition \ref{two-scale weak limit of admissible stress - regime zero}.
\end{proof}
\CCC The following lemma is analogous to \Cref{approximation of stresses - regime inf}. \BBB
\begin{lemma} \label{approximation of stresses - regime inf} 
Let $\omega \subset \R^2$ be an open bounded set that is star-shaped with respect to one of its points and let $\Sigma \in \calK^{hom}_{\infty}$. Then, there exists a sequence $\Sigma_n \in L^2(\R^2 \times I \times \calY;\M^{3 \times 3}_{\sym})$ such that the following holds:
\begin{enumerate}[label=(\alph*)]
    \item \label{approximation of stresses (a) - regime inf} $\Sigma_n \in C^\infty(\R^3;L^2(\calY;\M^{3 \times 3}_{\sym}))$ and $\Sigma_n \strong \Sigma$ strongly in $L^2(\omega \times I \times \calY;\M^{3 \times 3}_{\sym})$,
    \item \label{approximation of stresses (b) - regime inf} $\div_{y}\Sigma_n(x,\cdot) = 0 \text{ on } \calY$ for every $x \in \R^3$,
    \item \label{approximation of stresses (c) - regime inf} $(\Sigma_n(x,y))_{\dev} \in K(y) \,\text{ for every } x \in \R^3 \text{ and } \calL^{2}_{y}\text{-a.e. } y \in \calY$.
    \end{enumerate}
    Further, if we set $\sigma_n(x) := \int_{\calY} \Sigma_n(x,y) \,dy$, and $\bar{\sigma}_n,\, \hat{\sigma}_n \in L^2(\omega;\M^{2 \times 2}_{\sym})$ are the zero-th and first order moments of the $2 \times 2$ minor of $\sigma_n$, then:
\begin{enumerate}[label=(\alph*), resume]
    \item \label{approximation of stresses (d) - regime inf} $\sigma_n \in C^\infty(\R^2 \times I;\M^{3 \times 3}_{\sym})$ and $\sigma_n \strong \sigma$ strongly in $L^2(\omega \times I;\M^{3 \times 3}_{\sym})$,
    \item \label{approximation of stresses (e) - regime inf} $\div_{x'}\bar{\sigma}_n = 0 \text{ in } \omega$,
    \item \label{approximation of stresses (f) - regime inf} $\div_{x'}\div_{x'}\hat{\sigma}_n = 0 \text{ in } \omega$.
\end{enumerate}
\end{lemma}

\begin{proof}
\CCC The proof is again analogous to \cite[Lemma 5.13]{BDV}. \BBB
The only difference is that the convolution and dilation used to define $\Sigma_n$ are taken in $\R^3$ instead of $\R^2$.
\end{proof}

\subsection{The principle of maximum plastic work} \label{principleofmximumpl} 
\CCC We introduce the following functionals: Let $\gamma \in \{0,+\infty\}$. 
For $(u,E,P) \in \calA^{hom}_{\gamma}(w)$ we define
\begin{equation} \label{definition Q^hom} 
    \calQ_0^{hom}(E) := \int_{\Omega \times \calY} Q_r\left(y, E\right) \,dx dy, \quad   \calQ_{+\infty}^{hom}(E) := \int_{\Omega \times \calY} Q \left(y, E\right) \,dx dy
\end{equation}
and
\begin{equation} \label{definition H^hom}
    \calH_0^{hom}(P) := \int_{\closure{\Omega} \times \calY} H_r\left(y, \frac{dP}{d|P|} \right) \,d|P|, \quad  \calH_{+\infty}^{hom}(P) := \int_{\closure{\Omega} \times \calY} H\left(y, \frac{dP}{d|P|} \right) \,d|P|.
\end{equation}

\BBB

The aim of this subsection is to prove the following inequality between two-scale dissipation and plastic work, which in turn will be essential to prove the global stability condition of two-scale quasistatic evolutions.
\CCC Its proof is a direct consequence of  \Cref{two-scale Hill's principle - regime zero} for the case $\gamma=0$, and of  \Cref{two-scale Hill's principle - regime inf} for the case $\gamma=+\infty$. \BBB
\begin{corollary} \label{two-scale dissipation and plastic work inequality}
Let \RED $\gamma \in \{0, +\infty\}$ \BLACK.
Then
\begin{equation*}
    \CCC \calH^{hom}_{\gamma}\BBB(P) \geq -\int_{\Omega \times \calY} \Sigma : E \,dx dy + \int_{\omega} \bar{\sigma} : E\bar{w} \,dx' - \frac{1}{12} \int_{\omega} \hat{\sigma} : D^2w_3 \,dx',
\end{equation*}
for every $\Sigma \in \calK^{hom}_{\gamma}$ and $(u,E,P) \in \calA^{hom}_{\gamma}(w)$.
\end{corollary}

\RED 
The proof relies on the approximation argument given in \Cref{approximation of stresses - regime zero} and \Cref{approximation of stresses - regime inf} and on two-scale duality, which can be established only for smooth stresses by disintegration and duality pairings between admissible stresses and plastic strains (given by \eqref{cell stress-strain duality - regime zero} and \eqref{cell stress-strain duality - regime inf}). 
The problem is that the measure $\eta$ defined in \Cref{disintegration result - regime zero} and \Cref{disintegration result - regime inf} can concentrate on the points where the stress (which is only in $L^2$) is not well-defined. 
The difference with respect to \cite[Proposition 5.11]{Francfort.Giacomini.2014} is that one can rely only on the approximation given by \Cref{approximation of stresses - regime zero} and \Cref{approximation of stresses - regime inf}, which are given for star-shaped domains. 
To prove \EEE the corresponding result for general domains we rely on \BBB the localization argument (see the proof of Step 2 of \Cref{two-scale stress-strain duality - regime zero} and the proof of \Cref{two-scale Hill's principle - regime zero}, as well as a \Cref{two-scale stress-strain duality - regime inf} and \Cref{two-scale Hill's principle - regime inf}).
\BLACK

\subsubsection{Case \texorpdfstring{$\gamma = 0$}{γ = 0}}
\CCC The following proposition defines the measure $\lambda$ through  two-scale stress-strain duality based on the approximation argument. \BBB
\begin{proposition} \label{two-scale stress-strain duality - regime zero}
Let $\Sigma \in \calK^{hom}_{0}$ and $(u,E,P) \in \calA^{hom}_{0}(w)$ with the associated $\mu \in \calXzero{\ext{\omega}}$, $\kappa \in \calYzero{\ext{\omega}}$. 
There exists an element $\lambda \in \Mb(\ext{\Omega} \times \calY)$ such that
for every $\varphi \in C_c^2(\ext{\omega})$
\begin{align*}
    \langle \lambda, \varphi \rangle =& - \int_{\Omega \times \calY} \varphi(x')\,\Sigma : E \,dx dy + \int_{\omega} \varphi\,\bar{\sigma} : E\bar{w} \,dx' - \frac{1}{12}\int_{\omega} \varphi\,\hat{\sigma} : D^2w_3 \,dx' \\
    & - \int_{\omega} \bar{\sigma} : \left( (\bar{u}-\bar{w}) \odot \nabla\varphi \right) \,dx' - \frac{1}{6} \int_{\omega} \hat{\sigma} : \big( \nabla (u_3 - w_3) \odot \nabla \varphi \big) \,dx' \\
    & - \frac{1}{12} \int_{\omega} (u_3 - w_3)\,\hat{\sigma} : \nabla^2 \varphi \,dx'.
\end{align*}
Furthermore, the mass of $\lambda$ is given by
\begin{equation} \label{mass of lambda - regime zero}
    \lambda(\ext{\Omega} \times \calY) = -\int_{\Omega \times \calY} \Sigma : E \,dx dy + \int_{\omega} \bar{\sigma} : E\bar{w} \,dx' - \frac{1}{12} \int_{\omega} \hat{\sigma} : D^2w_3 \,dx'.
\end{equation}
\end{proposition}

\begin{proof}
The proof is divided into two steps.

\noindent{\bf Step 1.}
Suppose that $\omega$ is star-shaped with respect to one of its points.

Let $\{ \Sigma_n \} \subset C^\infty(\R^2;L^2(I \times \calY;\M^{3 \times 3}_{\sym}))$ be sequence given by Lemma \ref{approximation of stresses - regime zero}. 
We define the sequence
\begin{equation*}
    \lambda_n := \eta \genprod [ \Sigma_n(x',\cdot) : P_{x'} ] \in \Mb(\ext{\Omega} \times \calY),
\end{equation*}
where \CCC $\eta$ is given by \cref{disintegration result - regime zero} and \BBB the duality $[ \Sigma_n(x',\cdot) : P_{x'} ]$ is a well defined bounded measure on $I \times \calY$ for $\eta$-a.e. $x' \in \ext{\omega}$. 
Further, in view of \Cref{admissible configurations and disintegration - regime zero}, \eqref{cell stress-strain duality - regime zero} gives
\begin{align*}
    &\int_{I \times \calY} \psi\,d[ \Sigma_n(x',\cdot) : P_{x'} ]  \\
    &= - \int_{I \times \calY} \psi(y)\,\Sigma_n(x,y) : \left[ C(x') E(x,y) - \left( A_1(x') + x_3 A_2(x') \right) \right] \,dx_3 dy\\
    &\,\quad - \int_{\calY} \bar{\Sigma}_n(x',y) : \big( \mu_{x'}(y) \odot \nabla_{y}\psi(y) \big) \,dy\\
    &\,\quad + \frac{1}{6} \int_{\calY} \hat{\Sigma}_n(x',y) : \big( \nabla_{y}\kappa_{x'}(y) \odot \nabla_{y}\psi(y) \big) \,dy + \frac{1}{12} \int_{\calY} \kappa_{x'}(y)\,\hat{\Sigma}_n(x',y) : \nabla^2_{y}\psi(y) \,dy,
\end{align*}
for every $\psi \in C^2(\calY)$, and
\begin{equation*}
    | [ \Sigma_n(x',\cdot) : P_{x'} ] | \leq \| \Sigma_n(x',\cdot) \|_{L^{\infty}(I \times \calY;\M^{2 \times 2}_{\sym})} |P_{x'}| \leq C\,|P_{x'}|,
\end{equation*}
where the last inequality stems from item \ref{approximation of stresses (c) - regime zero} in Lemma \ref{approximation of stresses - regime zero}.
This in turn implies that
\begin{equation*}
    |\lambda_n | = \eta \genprod | [ \Sigma_n(x',\cdot) : P_{x'} ] | \leq C \,\eta \genprod |P_{x'}| = C\,|P|,
\end{equation*}
from which we conclude that is $\{ \lambda_n \}$ is a bounded sequence.

Let now $\ext{I} \supset I$ be an open set which compactly contains $I$. \CCC We extend these measures by zero on $\ext{\omega} \times \ext{I} \times \calY$.  \BBB
Let $\xi$ be a smooth cut-off function with $\xi \equiv 1$ on $I$, with support contained in $\ext{I}$. 
Finally, we consider a test function $\phi(x,y) := \varphi(x')\xi(x_3)$, for $\varphi \in C_c^{\infty}(\ext{\omega})$. 
Then, since $\nabla_{y}\phi(x,y) = 0$ and $\nabla^2_{y}\phi(x,y) = 0$, we have
\begin{align*}
    \langle \lambda_n, \phi \rangle 
    &= \int_{\ext{\omega}} \left( \int_{I \times \calY} \phi(x,y) \,d[ \Sigma_n(x',\cdot) : P_{x'} ] \right) \,d\eta(x')\\
    &= - \int_{\ext{\Omega} \times \calY} \varphi(x')\,\Sigma_n(x,y) : \left[ C(x') E(x,y) - \left( A_1(x') + x_3 A_2(x') \right) \right] \,d\left(\eta \otimes \calL^{1}_{x_3} \otimes \calL^{2}_{y}\right)\\
    &= - \int_{\ext{\Omega} \times \calY} \varphi(x')\,\Sigma_n(x,y) : E(x,y) \,dx dy + \int_{\ext{\Omega}} \varphi(x')\,\sigma_n(x) : \left( A_1(x') + x_3 A_2(x') \right) \,d\left(\eta \otimes \calL^{1}_{x_3}\right)\\
    &= - \int_{\ext{\Omega} \times \calY} \varphi(x')\,\Sigma_n(x,y) : E(x,y) \,dx dy + \int_{\ext{\Omega}} \varphi(x')\,\sigma_n(x) : \,dEu(x)
\end{align*}
Since $u \in KL(\ext{\Omega})$, we have
\begin{align*}
    \int_{\ext{\Omega}} \varphi(x')\,\sigma_n(x) : \,dEu(x) 
    &= \int_{\ext{\omega}} \varphi(x')\,\bar{\sigma}_n(x') : \,dE\bar{u}(x') - \frac{1}{12}\int_{\ext{\omega}} \varphi(x')\,\hat{\sigma}_n(x') : \,dD^2u_3(x'),
\end{align*}
where $\bar{u} \in BD(\ext{\omega})$ and $u_3 \in BH(\ext{\omega})$ are the Kirchhoff-Love components of $u$. 
From the characterization given in \Cref{A_KL characherization}, we can thus conclude that
\begin{align*}
    \int_{\ext{\Omega}} \varphi(x')\,\sigma_n(x) : \,dEu(x) 
    =& \int_{\ext{\omega}} \varphi(x')\,\bar{\sigma}_n(x') : \bar{e}(x') \,dx' + \int_{\ext{\omega}} \varphi(x')\,\bar{\sigma}_n(x') : \,d\bar{p}(x')\\
    & + \frac{1}{12}\int_{\ext{\omega}} \varphi(x')\,\hat{\sigma}_n(x') : \hat{e}(x') \,dx' + \frac{1}{12}\int_{\ext{\omega}} \varphi(x')\,\hat{\sigma}_n(x') : \,d\hat{p}(x')\\
    =& \int_{\ext{\omega}} \varphi(x')\,\bar{\sigma}_n(x') : \bar{e}(x') \,dx' + \int_{\ext{\omega}} \varphi(x') \,d[\bar{\sigma}_n : \bar{p}](x')\\
    & + \frac{1}{12}\int_{\ext{\omega}} \varphi(x')\,\hat{\sigma}_n(x') : \hat{e}(x') \,dx' + \frac{1}{12}\int_{\ext{\omega}} \varphi(x') \,d[\hat{\sigma}_n : \hat{p}](x'),
\end{align*}
where in the last equality we used that $\bar{\sigma}_n$ and $\hat{\sigma}_n$ are smooth functions. 
Notice that, since $\bar{p} \equiv 0$ and $\hat{p} \equiv 0$ outside of $\omega \cup \gamma_\Dir$, we have
\begin{equation*}
    \int_{\ext{\omega}} \varphi \,d[\bar{\sigma}_n : \bar{p}] = \int_{\omega \cup \gamma_\Dir} \varphi \,d[\bar{\sigma}_n : \bar{p}], \quad 
    \int_{\ext{\omega}} \varphi \,d[\hat{\sigma}_n : \hat{p}] = \int_{\omega \cup \gamma_\Dir} \varphi \,d[\hat{\sigma}_n : \hat{p}].
\end{equation*}
Furthermore, since $e = E = E\bar{w} - x_3 D^2w_3$ on $\ext{\Omega} \setminus \Omega$, we can conclude that
\begin{align*}
    \langle \lambda_n, \phi \rangle 
    &= - \int_{\ext{\Omega} \times \calY} \varphi(x')\,\Sigma_n : E \,dx dy + \int_{\ext{\omega}} \varphi\,\bar{\sigma}_n : \bar{e} \,dx' + \frac{1}{12}\int_{\ext{\omega}} \varphi\,\hat{\sigma}_n : \hat{e} \,dx'\\
    &\,\quad + \int_{\omega \cup \gamma_\Dir} \varphi \,d[\bar{\sigma}_n : \bar{p}] + \frac{1}{12}\int_{\omega \cup \gamma_\Dir} \varphi \,d[\hat{\sigma}_n : \hat{p}]\\
    &= - \int_{\Omega \times \calY} \varphi(x')\,\Sigma_n : E \,dx dy + \int_{\omega} \varphi\,\bar{\sigma}_n : \bar{e} \,dx' + \frac{1}{12}\int_{\omega} \varphi\,\hat{\sigma}_n : \hat{e} \,dx'\\
    &\,\quad + \int_{\omega \cup \gamma_\Dir} \varphi \,d[\bar{\sigma}_n : \bar{p}] + \frac{1}{12}\int_{\omega \cup \gamma_\Dir} \varphi \,d[\hat{\sigma}_n : \hat{p}].
\end{align*}
\CCC Taking into account that $\div_{x'}\bar{\sigma}_n = 0 \text{ in } \omega$, by integration by parts (see also \cite[Proposition 7.2]{Davoli.Mora.2013}) \BBB we have for every $\varphi \in C^1(\closure{\omega})$
\begin{align*}
    \int_{\omega \cup \gamma_\Dir} \varphi \,d[\bar{\sigma}_n : \bar{p}] + \int_{\omega} \varphi\,\bar{\sigma}_n : (\bar{e} - E\bar{w}) \,dx' +  \int_{\omega} \bar{\sigma}_n : \left( (\bar{u} - \bar{w}) \odot \nabla\varphi \right) \,dx' = 0.
\end{align*}
Likewise \CCC taking into account $\div_{x'}\div_{x'}\hat{\sigma}_n = 0 \text{ in } \omega$ 
and $u_3 = w_3 \text{ on } \gamma_\Dir$, by integration by parts (see also \cite[Proposition 7.6]{Davoli.Mora.2013}), \BBB we have for every $\varphi \in C^2(\closure{\omega})$
\begin{align*}
    &\int_{\omega \cup \gamma_\Dir} \varphi \,d[\hat{\sigma}_n : \hat{p}] + \int_{\omega} \varphi\,\hat{\sigma}_n : (\hat{e} + D^2w_3) \,dx' \\
    &+ 2 \int_{\omega} \hat{\sigma}_n : \big( \nabla (u_3 - w_3) \odot \nabla \varphi \big) \,dx' + \int_{\omega} (u_3 - w_3)\,\hat{\sigma}_n : \nabla^2 \varphi \,dx'
    = 0.
\end{align*}

Let now $\lambda \in \Mb(\ext{\Omega} \times \calY)$ be such that (up to a subsequence)
\begin{equation*}
    \lambda_n \weakstar \lambda \quad \text{weakly* in $\Mb(\ext{\Omega} \times \calY)$}.
\end{equation*}
By items \ref{approximation of stresses (a) - regime zero} and \ref{approximation of stresses (f) - regime zero} in Lemma \ref{approximation of stresses - regime zero}, we have in the limit
\begin{align*}
    \langle \lambda, \phi \rangle &= \lim_n \,\langle \lambda_n, \phi \rangle\\
    &= \lim_n \,\Big[ - \int_{\Omega \times \calY} \varphi(x')\,\Sigma_n : E \,dx dy + \int_{\omega} \varphi\,\bar{\sigma}_n : E\bar{w} \,dx' - \frac{1}{12}\int_{\omega} \varphi\,\hat{\sigma}_n : D^2w_3 \,dx' \\
    &\,\quad\,\quad\,\quad - \int_{\omega} \bar{\sigma}_n : \left( (\bar{u}-\bar{w}) \odot \nabla\varphi \right) \,dx' - \frac{1}{6} \int_{\omega} \hat{\sigma}_n : \big( \nabla (u_3 - w_3) \odot \nabla \varphi \big) \,dx' \\
    &\,\quad\,\quad\,\quad - \frac{1}{12} \int_{\omega} (u_3 - w_3)\,\hat{\sigma}_n : \nabla^2 \varphi \,dx' \Big]\\
    &= - \int_{\Omega \times \calY} \varphi(x')\,\Sigma : E \,dx dy + \int_{\omega} \varphi\,\bar{\sigma} : E\bar{w} \,dx' - \frac{1}{12}\int_{\omega} \varphi\,\hat{\sigma} : D^2w_3 \,dx' \\
    &\,\quad - \int_{\omega} \bar{\sigma} : \left( (\bar{u}-\bar{w}) \odot \nabla\varphi \right) \,dx' - \frac{1}{6} \int_{\omega} \hat{\sigma} : \big( \nabla (u_3 - w_3) \odot \nabla \varphi \big) \,dx' \\
    &\,\quad - \frac{1}{12} \int_{\omega} (u_3 - w_3)\,\hat{\sigma} : \nabla^2 \varphi \,dx'.
\end{align*}
Taking $\varphi \nearrow \charfun{\ext{\omega}}$, we deduce \eqref{mass of lambda - regime zero}.

\noindent{\bf Step 2.}
If $\omega$ is not star-shaped, then since $\omega$ is a bounded $C^2$ domain (in particular, with Lipschitz boundary) by \cite[Proposition 2.5.4]{Carbone.DeArcangelis.2002} there exists a finite open covering $\{U_i\}$ of $\closure{\omega}$ such that $\omega \cap U_i$ is (strongly) star-shaped with Lipschitz boundary. \CCC Again, since the sets which are intersecting $\partial \omega$  are cylindrical up to a rotation, we can slightly change them such that they become $C^2$. \BBB

Let $\{\psi_i\}$ be a smooth partition of unity subordinate to the covering $\{U_i\}$, i.e. $\psi_i \in C^{\infty}(\closure{\omega})$, with $0 \leq \psi_i \leq 1$, such that $\supp(\psi_i) \subset U_i$ and $\sum_{i} \psi_i = 1$ on $\closure{\omega}$. 

For each $i$, let
\begin{equation*}
    \Sigma^i(x,y)
    :=
    \begin{cases}
    \Sigma(x,y) & \ \text{ if } x' \in \omega \cap U_i,\\
    0 & \ \text{ otherwise}.
    \end{cases}
\end{equation*}
Since $\Sigma^i \in \RED \calK^{hom}_{0} \BLACK$, the construction in Step 1 yields that there exist sequences $\{ \Sigma^i_n \} \subset C^\infty(\R^2;L^2(I \times \calY;\M^{3 \times 3}_{\sym}))$ and 
\begin{equation*}
    \lambda^i_n := \eta \genprod [ (\Sigma^i_n)_{\dev}(x',\cdot) : P_{x'} ] \in \Mb((\omega \cap U_i) \times I \times \calY),
\end{equation*}
such that
\begin{equation*}
    \lambda^i_n \weakstar \lambda^i \quad \text{weakly* in $\Mb((\omega \cap U_i) \times I \times \calY)$},
\end{equation*}
with
\begin{align*}
    \langle \lambda^i, \varphi \rangle 
    &= - \int_{(\omega \cap U_i) \times I \times \calY} \varphi(x')\,\Sigma : E \,dx dy + \int_{\omega \cap U_i} \varphi\,\bar{\sigma} : E\bar{w} \,dx' - \frac{1}{12}\int_{\omega \cap U_i} \varphi\,\hat{\sigma} : D^2w_3 \,dx' \\
    &\,\quad - \int_{\omega \cap U_i} \bar{\sigma} : \left( (\bar{u}-\bar{w}) \odot \nabla\varphi \right) \,dx' - \frac{1}{6} \int_{\omega \cap U_i} \hat{\sigma} : \big( \nabla (u_3 - w_3) \odot \nabla \varphi \big) \,dx' \\
    &\,\quad - \frac{1}{12} \int_{\omega \cap U_i} (u_3 - w_3)\,\hat{\sigma} : \nabla^2 \varphi \,dx'.
\end{align*}
for every $\varphi \in C_c^2(\closure{\omega} \cap U_i)$.
This allows us to define measures on $\ext{\Omega} \times \calY$ by letting, for every $\phi \in C_0(\ext{\Omega} \times \calY)$,
\begin{equation*}
    \langle \lambda_n, \phi \rangle := \sum_{i} \langle \lambda^i_n, \psi_i(x')\,\phi \rangle,
\end{equation*}
and
\begin{equation*}
    \langle \lambda, \phi \rangle := \sum_{i} \langle \lambda^i, \psi_i(x')\,\phi \rangle.
\end{equation*}
Then we can see that $\lambda_n \weakstar \lambda$ weakly* in $\Mb(\ext{\Omega} \times \calY)$, and $\lambda$ satisfies all the required properties.
\end{proof}
\CCC The following theorem provides a two-scale Hill's principle (cf. \cite[Theorem 5.12]{Francfort.Giacomini.2014}). \BBB
\begin{theorem} \label{two-scale Hill's principle - regime zero}
Let $\Sigma \in \calK^{hom}_{0}$ and $(u,E,P) \in \calA^{hom}_{0}(w)$ with the associated $\mu \in \calXzero{\ext{\omega}}$, $\kappa \in \calYzero{\ext{\omega}}$. 
\CCC If $\calY$ is a geometrically admissible multi-phase torus, under the assumption on the ordering of phases we have \BBB
\CCC
\begin{equation*}
   \overline{H_r\left(y, \frac{dP}{d|P|}\right)\,|P|} \geq \bar{\lambda}, 
\end{equation*}
\BBB
where $\lambda \in \Mb(\ext{\Omega} \times \calY)$ is given by \Cref{two-scale stress-strain duality - regime zero}.
\end{theorem}

\begin{proof}
Take $\varphi \in C_{c}(\ext{\omega} \times \calY)$ non-negative. 
Let $\{\Sigma^i_n\}$, $\{\lambda^i_n\}$ and $\lambda^i$ be defined as in Step 2 of the proof of \Cref{two-scale stress-strain duality - regime zero}. 
Item \ref{approximation of stresses (c) - regime zero} in \Cref{approximation of stresses - regime zero} implies that
\begin{equation*}
    (\Sigma^i_n)_{\dev}(x,y) \in K(y) \,\text{ for every } x' \in \omega \text{ and } \calL^{1}_{x_3} \otimes \calL^{2}_{y}\text{-a.e. } (x_3,y) \in I \times \calY.   
\end{equation*}
By \Cref{cell Hill's principle - regime zero}, we have for $\eta$-a.e. $x' \in \ext{\omega}$
\begin{equation*}
    \int_{I \times \calY} \CCC \varphi(x',y)\BBB\,H_r\left(y, \frac{dP_{x'}}{d|P_{x'}|}\right) \,d|P_{x'}| \geq \int_{I \times \calY} \CCC \varphi(x',y)\BBB \,d[ \Sigma^i_n : P_{x'} ], \ \text{ for every } \varphi \in C(\calY), \varphi \geq 0.
\end{equation*}
Since $\frac{dP}{d|P|}(x,y) = \frac{dP_{x'}}{d|P_{x'}|}(x_3,y)$ for $|P_{x'}|$-a.e. $(x_3,y) \in I \times \calY$ by \cite[Proposition 2.2]{BDV}, we can conclude that
\begin{align*}
    H_r\left(y, \frac{dP}{d|P|}\right)\,|P| &= \eta \genprod H_r\left(y, \frac{dP}{d|P|}\right)\,|P_{x'}| = \eta \genprod H_r\left(y, \frac{dP_{x'}}{d|P_{x'}|}\right)\,|P_{x'}|\\
    &= \sum_{i} \psi_i \eta \genprod H_r\left(y, \frac{dP_{x'}}{d|P_{x'}|}\right)\,|P_{x'}|.
\end{align*}
Consequently,
\begin{align*}
   \int_{\ext{\Omega} \times \calY} \CCC \varphi(\cdot,y)\BBB\,H_r\left(y, \frac{dP}{d|P|}\right)\,d|P| &= \sum_{i} \int_{\ext{\omega}} \psi_i(x') \left( \int_{I \times \calY} \CCC\varphi(x',y)\BBB\,H_r\left(y, \frac{dP_{x'}}{d|P_{x'}|}\right)\,|P_{x'}| \right) \,d\eta(x') \\
    &\geq \sum_{i} \int_{\ext{\omega}} \psi_i(x') \left( \int_{I \times \calY} \CCC\varphi(x',y)\BBB \,d[ \Sigma^i_n : P_{x'} ] \right) \,d\eta(x') \\
    &= \sum_{i} \int_{\ext{\Omega} \times \calY} \psi_i(x') \CCC \varphi(x',y)\BBB \,d\lambda^i_n(x,y) = \int_{\ext{\Omega} \times \calY} \CCC \varphi\BBB \,d\lambda_n.
\end{align*}
By passing to the limit, we infer the desired inequality.
\end{proof}

\subsubsection{Case \texorpdfstring{$\gamma = +\infty$}{γ = +∞}}
\CCC The following proposition is the analogue of \Cref{two-scale stress-strain duality - regime zero}. \BBB
\begin{proposition} \label{two-scale stress-strain duality - regime inf}
Let $\Sigma \in \calK^{hom}_{\infty}$ and $(u,E,P) \in \calA^{hom}_{\infty}(w)$ with the associated $\mu \in \calXinf{\ext{\Omega}}$, $\kappa \in \calXinf{\ext{\Omega}}$, $\zeta \in \Mb(\Omega;\R^3)$.
There exists an element $\lambda \in \Mb(\ext{\Omega} \times \calY)$ such that
for every $\varphi \in C_c^2(\ext{\omega})$
\begin{align*}
    \langle \lambda, \varphi \rangle =& - \int_{\Omega \times \calY} \varphi(x')\,\Sigma : E \,dx dy + \int_{\omega} \varphi\,\bar{\sigma} : E\bar{w} \,dx' - \frac{1}{12}\int_{\omega} \varphi\,\hat{\sigma} : D^2w_3 \,dx' \\
    & - \int_{\omega} \bar{\sigma} : \left( (\bar{u}-\bar{w}) \odot \nabla\varphi \right) \,dx' - \frac{1}{6} \int_{\omega} \hat{\sigma} : \big( \nabla (u_3 - w_3) \odot \nabla \varphi \big) \,dx' \\
    & - \frac{1}{12} \int_{\omega} (u_3 - w_3)\,\hat{\sigma} : \nabla^2 \varphi \,dx'.
\end{align*}
Furthermore, the mass of $\lambda$ is given by
\begin{equation} \label{mass of lambda - regime inf}
    \lambda(\ext{\Omega} \times \calY) = -\int_{\Omega \times \calY} \Sigma : E \,dx dy + \int_{\omega} \bar{\sigma} : E\bar{w} \,dx' - \frac{1}{12} \int_{\omega} \hat{\sigma} : D^2w_3 \,dx'.
\end{equation}
\end{proposition}

\begin{proof}
Suppose that $\omega$ is star-shaped with respect to one of its points.

Let $\{ \Sigma_n \} \subset C^\infty(\R^3;L^2(\calY;\M^{3 \times 3}_{\sym}))$ be sequence given by Lemma \ref{approximation of stresses - regime inf}. 
We define the sequence
\begin{equation*}
    \lambda_n := \eta \genprod [ (\Sigma_n)_{\dev}(x,\cdot) : P_{x} ] \in \Mb(\ext{\Omega} \times \calY),
\end{equation*}
where \CCC $\eta$ is given by \cref{disintegration result - regime inf} and \BBB the duality $[ (\Sigma_n)_{\dev}(x,\cdot) : P_{x} ]$ is a well defined bounded measure on $\calY$ for $\eta$-a.e. $x \in \ext{\Omega}$. 
Further, in view of \Cref{admissible configurations and disintegration - regime inf},  \eqref{cell stress-strain duality - regime inf} gives
\begin{align*}
    &\int_{\calY} \psi\,d[ (\Sigma_n)_{\dev}(x,\cdot) : P_{x} ]  \\
    &= - \int_{\calY} \psi\,\Sigma_n : \left[ C(x) E(x,y) - \begin{pmatrix} A_1(x') + x_3 A_2(x') & 0 \\ 0 & 0 \end{pmatrix} \right] \,dy\\
    &\,\quad - \int_{\calY} (\Sigma_n)''(x,y) : \big( \mu_{x}(y) \odot \nabla_{y}\psi(y) \big) \,dy\\
    &\,\quad - \sum_{\alpha=1,2} \int_{\calY} \kappa_{x}(y)\,(\Sigma_n)_{\alpha 3}(x,y)\,\partial_{y_\alpha}\psi(y) \,dy + \sum_{i=1,2,3} z_i\,\int_{\calY} \psi(y)\,(\Sigma_n)_{i3}(x,y) \,dy,
\end{align*}
for every $\psi \in C^1(\calY)$, and
\begin{equation*}
    | [ (\Sigma_n)_{\dev}(x,\cdot) : P_{x} ] | \leq \| (\Sigma_n)_{\dev}(x,\cdot) \|_{L^{\infty}(\calY;\M^{3 \times 3}_{\sym})} |P_{x}| \leq C\,|P_{x}|,
\end{equation*}
where the last inequality stems from item \ref{approximation of stresses (c) - regime inf} in Lemma \ref{approximation of stresses - regime inf}.
This in turn implies that
\begin{equation*}
    |\lambda_n | = \eta \genprod | [ (\Sigma_n)_{\dev}(x,\cdot) : P_{x} ] | \leq C \,\eta \genprod |P_{x}| = C\,|P|,
\end{equation*}
from which we conclude that is $\{ \lambda_n \}$ is a bounded sequence.

Let now $\ext{I} \supset I$ be an open set which compactly contains $I$ \CCC and extend the above measures by zero on $\ext{\omega} \times \ext{I} \times \calY$. \BBB 
Let $\xi$ be a smooth cut-off function with $\xi \equiv 1$ on $I$, with support contained in $\ext{I}$. 
Finally, we consider a test function \CCC $\phi(x) := \varphi(x')\xi(x_3)$ \BBB, for $\varphi \in C_c^{\infty}(\ext{\omega})$. 
Then, since \CCC$\nabla_{y}\phi(x) = 0$ \BBB, \CCC $\partial_{y_\alpha}\phi(x) = 0$ \BBB and $\int_{\calY} (\Sigma_n)_{i3}(x,y) \,dy = 0$, we have
\begin{align*}
    \langle \lambda_n, \phi \rangle 
    &= \int_{\ext{\Omega}} \left( \int_{\calY} \phi(x,y) \,d[  (\Sigma_n)_{\dev}(x,\cdot) : P_{x} ] \right) \,d\eta(x)\\
    &= - \int_{\ext{\Omega} \times \calY} \varphi(x')\,\Sigma_n(x,y) : \left[ C(x) E(x,y) - \begin{pmatrix} A_1(x') + x_3 A_2(x') & 0\\ 0 & 0 \end{pmatrix} \right] \,d\left(\eta \otimes \calL^{2}_{y}\right)\\
    &= - \int_{\ext{\Omega} \times \calY} \varphi(x')\,\Sigma_n(x,y) : E(x,y) \,dx dy + \int_{\ext{\Omega}} \varphi(x')\,\sigma_n(x) : \left( A_1(x') + x_3 A_2(x') \right) \,d\eta\\
    &= - \int_{\ext{\Omega} \times \calY} \varphi(x')\,\Sigma_n(x,y) : E(x,y) \,dx dy + \int_{\ext{\Omega}} \varphi(x')\,\sigma_n(x) : \,dEu(x)
\end{align*}
From this point on, the proof is exactly the same as the proof of \Cref{two-scale stress-strain duality - regime zero} by defining in the analogous way $\Sigma_n^i$, $\lambda_n^i$, i.e. $\Sigma_i$, $\lambda_i$. 
\end{proof}
\CCC The following theorem is analogous to \Cref{two-scale Hill's principle - regime zero}. \BBB
\begin{theorem} \label{two-scale Hill's principle - regime inf}
Let $\Sigma \in \calK^{hom}_{\infty}$ and $(u,E,P) \in \calA^{hom}_{\infty}(w)$ with the associated $\mu \in \calXinf{\ext{\Omega}}$, $\kappa \in \calXinf{\ext{\Omega}}$, $\zeta \in \Mb(\Omega;\R^3)$. \CCC If $\calY$ is a geometrically admissible multi-phase torus,  under the assumption on the ordering of phases we have \BBB  
\begin{equation*}
    H\left(y, \frac{dP}{d|P|}\right)\,|P| \geq \lambda,
\end{equation*}
where $\lambda \in \Mb(\ext{\Omega} \times \calY)$ is given by \Cref{two-scale stress-strain duality - regime inf}.
\end{theorem}

\begin{proof}
Let $\{\Sigma^i_n\}$, $\{\lambda^i_n\}$ and $\lambda^i$ be defined as in the proof of \Cref{two-scale stress-strain duality - regime inf}. 
Item \ref{approximation of stresses (c) - regime inf} in \Cref{approximation of stresses - regime inf} implies that
\begin{equation*}
    (\Sigma^i_n)_{\dev}(x,y) \in K(y) \,\text{ for every } x \in \Omega \text{ and } \calL^{2}_{y}\text{-a.e. } y \in \calY.   
\end{equation*}
By \Cref{cell Hill's principle - regime inf}, we have for $\eta$-a.e. $x \in \ext{\Omega}$
\begin{equation*}
    H\left(y, \frac{dP_{x}}{d|P_{x}|}\right)\,|P_{x}| \geq [ (\Sigma^i_n)_{\dev}(x,\cdot) : P_{x} ] \ \text{ as measures on } \calY.
\end{equation*}
Since $\frac{dP}{d|P|}(x,y) = \frac{dP_{x}}{d|P_{x}|}(y)$ for $|P_{x}|$-a.e. $y \in \calY$ by \cite[Proposition 2.2]{BDV}, we can conclude that
\begin{align*}
    H\left(y, \frac{dP}{d|P|}\right)\,|P| &= \eta \genprod H\left(y, \frac{dP}{d|P|}\right)\,|P_{x}| = \eta \genprod H\left(y, \frac{dP_{x}}{d|P_{x}|}\right)\,|P_{x}|\\
    &= \sum_{i} \psi_i(x') \eta \genprod H\left(y, \frac{dP_{x}}{d|P_{x}|}\right)\,|P_{x}| \\
    &\geq \sum_{i} \psi_i(x') \eta \genprod [ (\Sigma^i_n)_{\dev}(x,\cdot) : P_{x} ] \\
    &= \sum_{i} \psi_i(x') \lambda^i_n = \lambda_n.
\end{align*}
By passing to the limit, we have the desired inequality.
\end{proof}

\section{Two-scale quasistatic evolutions}
\label{dynamics}


The associated $\calH^{hom}$-variation of a function $P : [0,T] \to \Mb(\ext{\Omega} \times \calY;\M^{3 \times 3}_{\dev})$ on $[a,b]$ is then defined as
\begin{equation*}
    \calD_{\CCC\calH^{hom}_{\gamma} \BBB}(P; a, b) := \sup\left\{ \sum_{i = 1}^{n-1} \CCC \calH^{hom}_{\gamma} \BBB\left(P(t_{i+1}) - P(t_i)\right) : a = t_1 < t_2 < \ldots < t_n = b,\ n \in \N \right\}.
\end{equation*}
In this section we prescribe for every $t \in [0, T]$ a boundary datum $w(t) \in H^1(\ext{\Omega};\R^3) \cap KL(\ext{\Omega})$ and we assume the map $t\mapsto w(t)$ to be absolutely continuous from $[0, T]$ into $H^1(\ext{\Omega};\R^3)$.

We now give the notion of the limiting quasistatic elasto-plastic evolution.

\begin{definition}
A \emph{two-scale quasistatic evolution} for the boundary datum $w(t)$ is a function $t \mapsto (u(t), E(t), P(t))$ from $[0,T]$ into $KL(\ext{\Omega}) \times L^2(\ext{\Omega} \times \calY;\M^{3 \times 3}_{\sym}) \times \Mb(\ext{\Omega} \times \calY;\M^{3 \times 3}_{\dev})$ which satisfies the following conditions:
\begin{enumerate}[label=(qs\arabic*)$^{hom}_{\gamma}$]
    \item \label{hom-qs S} for every $t \in [0,T]$ we have $(u(t), E(t), P(t)) \in \calA^{hom}_{\gamma}(w(t))$ and
    \begin{equation*}
        \calQ^{hom}_{\CCC \gamma \BBB} (E(t)) \leq \calQ^{hom}_{\CCC \gamma \BBB}(H) + \calH^{hom}_{\CCC \gamma \BBB}(\Pi-P(t)),
    \end{equation*}
    for every $(\upsilon,H,\Pi) \in \calA^{hom}_{\gamma}(w(t))$.
    \item \label{hom-qs E} the function $t \mapsto P(t)$ from $[0, T]$ into $\Mb(\ext{\Omega} \times \calY;\M^{3 \times 3}_{\dev})$ has bounded variation and for every $t \in [0, T]$
    \begin{equation*}
        \calQ^{hom}_{\CCC 0 \BBB}(E(t)) + \calD_{\calH^{hom}_{\CCC 0 \BBB}}(P; 0, t) = \calQ^{hom}_{\CCC 0 \BBB}(E(0)) 
        + \int_0^t \int_{\Omega \times \calY} \C_r(y) E(s) : E\dot{w}(s) \,dx dy ds,
    \end{equation*}
    \CCC for $\gamma=0$ and 
    \begin{equation*}
        \calQ^{hom}_{\CCC +\infty \BBB}(E(t)) + \calD_{\calH^{hom}_{\CCC +\infty \BBB}}(P; 0, t) = \calQ^{hom}_{\CCC +\infty \BBB}(E(0)) 
        + \int_0^t \int_{\Omega \times \calY} \C(y) E(s) : E\dot{w}(s) \,dx dy ds,
    \end{equation*}
for $\gamma=+\infty$. \BBB
\end{enumerate}
\end{definition}

Recalling the definition of a $h$-quasistatic evolution for the boundary datum $w(t)$ given in Definition \ref{h-quasistatic evolution}, we are in a position to formulate the main result of the \CCC paper\BBB.

\begin{theorem} \label{main result}
Let $t \mapsto w(t)$ be absolutely continuous from $[0,T]$ into $H^1(\ext{\Omega};\R^3) \cap KL(\ext{\Omega})$. 
\CCC Let $\calY$ be a geometrically admissible multi-phase torus and let the assumption on the ordering of phases be satisfied. 
\BBB \CCC Assume also \eqref{tensorassumption}, \eqref{tensorassumption2} \BBB and \eqref{coercivity of H_i} and \BBB
 that there exists a sequence of triples $(u^h_0, e^h_0, p^h_0) \in \calA_h(w(0))$ such that
\begin{align}
    &u^h_0 \weakstar u_0 \quad \text{weakly* in $BD(\ext{\Omega})$}, \label{main result u^h_0 condition}\\
    &\Lambda_h e^h_0 \strongtwoscale E_0 \quad \text{two-scale strongly in $L^2(\ext{\Omega} \times \calY;\M^{3 \times 3}_{\sym})$}, \label{main result e^h_0 condition}\\
    &\Lambda_h p^h_0 \weakstartwoscale P_0 \quad \text{two-scale weakly* in $\Mb(\ext{\Omega} \times \calY;\M^{3 \times 3}_{\dev})$}, \label{main result p^h_0 condition}
\end{align}
for \RED $(u_0,E_0,P_0) \in \calA^{hom}_{\infty}(w(0))$ if $\gamma = +\infty$, \BLACK and $(u_0,E_0^{\prime\prime},P_0^{\prime\prime}) \in \calA^{hom}_{0}(w(0))$ with \CCC $E_0=\A_y E_0''$ \BBB if $\gamma = 0$.\\ 
For every $h > 0$, let 
\begin{equation*}
    t \mapsto (u^h(t), e^h(t), p^h(t))
\end{equation*}
be a $h$-quasistatic evolution \BBB in the sense of Definition \ref{h-quasistatic evolution} \BBB for the boundary datum $w$ such that $u^h(0) = u^h_0$, $e^h(0) = e^h_0$, and $p^h(0) = p^h_0$. 
Then, there exists a two-scale quasistatic evolution
\begin{equation*}
    t \mapsto (u(t), E(t), P(t)) 
\end{equation*}
for the boundary datum $w(t)$ such that $u(0) = u_0$,\, $E(0) = E_0$, and $P(0) = P_0$, and such that (up to subsequence)
for every $t \in [0,T]$
\begin{align}
    &u^h(t) \weakstar u(t) \quad \text{weakly* in $BD(\ext{\Omega})$}, \label{main result u^h(t)}\\
    &\Lambda_h e^h(t) \weaktwoscale E(t) \quad \text{two-scale weakly in $L^2(\ext{\Omega} \times \calY;\M^{3 \times 3}_{\sym})$}, \label{main result e^h(t)}\\
    &\Lambda_h p^h(t) \weakstartwoscale P(t) \quad \text{two-scale weakly* in $\Mb(\ext{\Omega} \times \calY;\M^{3 \times 3}_{\dev})$}, \label{main result p^h(t)}
\end{align}
in case \RED $\gamma = +\infty$ \BLACK, and
\begin{align}
    &u^h(t) \weakstar u(t) \quad \text{weakly* in $BD(\ext{\Omega})$}, \label{main result 0 u^h(t) - regime zero}\\
    &\CCC \Lambda_h e^h(t) \BBB \weaktwoscale \A_y E(t) \quad \text{two-scale weakly in $L^2(\ext{\Omega} \times \calY;\M^{3 \times 3}_{\sym})$}, \label{main result e^h(t) - regime zero}\\
    &p^h(t) \weakstartwoscale \begin{pmatrix} P(t) & 0 \\ 0 & 0 \end{pmatrix} \quad \text{two-scale weakly* in $\Mb(\ext{\Omega} \times \calY;\M^{3 \times 3}_{\sym})$}, \label{main result p^h(t) - regime zero}
\end{align}
in case $\gamma = 0$.
\end{theorem}

\begin{proof}
The proof is divided into several steps, in the spirit of evolutionary $\Gamma$-convergence \CCC and it  follows the lines of \cite[Theorem 6.2]{BDV}. \BBB
\RED
We present the proof in the case $\gamma = 0$, while the argument for the case $\gamma = +\infty$ is identical upon replacing the appropriate structures in the statement of \Cref{two-scale weak limit of scaled strains} and definition of $\calA^{hom}_{\gamma}(w)$.
\BLACK

\noindent{\bf Step 1: \em Compactness.}

First, we prove that that there exists a constant $C$, depending only on the initial and boundary data, such that
\begin{equation} \label{boundness in time 1}
    \sup_{t \in [0,T]} \left\|\Lambda_h e^h(t)\right\|_{L^2(\ext{\Omega} \times \calY;\M^{3 \times 3}_{\sym})} \leq C \;\text{ and }\; \calD_{\calH_h}(\Lambda_h p^h; 0, T) \leq C,
\end{equation}
for every $h>0$. 
Indeed, the energy balance of the $h$-quasistatic evolution \ref{h-qs E} and \eqref{coercivity of Q} imply
\begin{align*}
    &r_c \left\|\Lambda_h e^h(t)\right\|_{L^2(\ext{\Omega} \times \calY;\M^{3 \times 3}_{\sym})} + \calD_{\calH_h}(\Lambda_h p^h; 0, t) \\
    &\leq R_c \left\|\Lambda_h e^h(0)\right\|_{L^2(\ext{\Omega} \times \calY;\M^{3 \times 3}_{\sym})} + 2 R_c \sup_{t \in [0,T]} \left\|\Lambda_h e^h(t)\right\|_{L^2(\ext{\Omega} \times \calY;\M^{3 \times 3}_{\sym})} \int_0^T \left\|E\dot{w}(s)\right\|_{L^2(\ext{\Omega};\M^{3 \times 3}_{\sym})} \,ds,
\end{align*}
where the last integral is well defined as $t \mapsto E\dot{w}(t)$ belongs to $L^1([0,T];L^2(\ext{\Omega};\M^{3 \times 3}_{\sym}))$. 
In view of the boundedness of $\Lambda_h e^h_0$ that is implied by \eqref{main result e^h_0 condition}, property \eqref{boundness in time 1} now follows by the Cauchy-Schwarz inequality.

Second, from the latter inequality in \eqref{boundness in time 1} and \eqref{coercivity of H_i}, we infer that
\begin{equation*}
    r_k \left\|\Lambda_h p^h(t) - \Lambda_h p^h_0\right\|_{\Mb(\ext{\Omega} \times \calY;\M^{3 \times 3}_{\dev})} \leq \calH_h\left(\Lambda_h p^h(t) - \Lambda_h p^h_0\right) \leq \calD_{\calH_h}(\Lambda_h p^h; 0, t) \leq C,
\end{equation*}
for every $t \in [0,T]$, which together with \eqref{main result p^h_0 condition} implies 
\begin{equation} \label{boundness in time 2}
    \sup_{t \in [0,T]} \left\|\Lambda_h p^h(t)\right\|_{\Mb(\ext{\Omega} \times \calY;\M^{3 \times 3}_{\dev})} \leq C.
\end{equation}

Next, we note that $\left\|\cdot\right\|_{L^1(\ext{\Omega} \setminus \closure{\Omega};\M^{3 \times 3}_{\sym})}$ is a continuous seminorm on $BD(\ext{\Omega})$ which is also a norm on the set of rigid motions. 
Then, using a variant of Poincar\'{e}-Korn's inequality (see \cite[Chapter II, Proposition 2.4]{Temam.1985}) and the fact $(u^h(t), e^h(t), p^h(t)) \in \calA_h(w(t))$, we conclude that, for every $h > 0$ and $t \in [0,T]$,
\begin{align*}
    \left\|u^h(t)\right\|_{BD(\ext{\Omega})} &\leq C \left(\left\|u^h(t)\right\|_{L^1(\ext{\Omega} \setminus \closure{\Omega};\R^3)} + \left\|Eu^h(t)\right\|_{\Mb(\ext{\Omega};\M^{3 \times 3}_{\sym})}\right)\\
    &\leq C \left(\left\|w(t)\right\|_{L^1(\ext{\Omega} \setminus \closure{\Omega};\R^3)} + \left\|e^h(t)\right\|_{L^2(\ext{\Omega};\M^{3 \times 3}_{\sym})} + \left\|p^h(t)\right\|_{\Mb(\ext{\Omega};\M^{3 \times 3}_{\dev})}\right)\\
    &\leq C \left(\left\|w(t)\right\|_{L^2(\ext{\Omega};\R^3)} + \left\|\Lambda_h e^h(t)\right\|_{L^2(\ext{\Omega};\M^{3 \times 3}_{\sym})} + \left\|\Lambda_h p^h(t)\right\|_{\Mb(\ext{\Omega};\M^{3 \times 3}_{\dev})}\right).
\end{align*}
In view of the assumption $w \in H^1(\ext{\Omega};\R^3)$, from \eqref{boundness in time 2} and the former inequality in \eqref{boundness in time 1} it follows that the sequences $\{u^h(t)\}$ are bounded in $BD(\ext{\Omega})$ uniformly with respect to $t$.

Owing to \eqref{equivalence of variations}, we infer that $\calD_{\calH_h}$ and $\calV$ are equivalent norms, which immediately implies
\begin{equation} \label{boundness in time 3}
    \calV(\Lambda_h p^h; 0, T) \leq C,
\end{equation}
for every $h>0$. 
Hence, by a generalized version of Helly's selection theorem (see \cite[Lemma 7.2]{DalMaso.DeSimone.Mora.2006}), there exists a (not relabeled) subsequence, independent of $t$, and $P \in BV(0,T;\Mb(\ext{\Omega} \times \calY;\M^{3 \times 3}_{\dev}))$ such that
\begin{equation*}
    \Lambda_h p^h(t) \weakstartwoscale P(t) \quad \text{two-scale weakly* in $\Mb(\ext{\Omega} \times \calY;\M^{3 \times 3}_{\dev})$},
\end{equation*}
for every $t \in [0,T]$, and 
$\calV(P; 0, T) \leq C$.
We extract a further subsequence (possibly depending on $t$),
\begin{align*}
    &u^{h_t}(t) \weakstar u(t) \quad \text{weakly* in $BD(\ext{\Omega})$},\\
    &\Lambda_{h_t} e^{h_t}(t) \weaktwoscale E(t) \quad \text{two-scale weakly in $L^2(\ext{\Omega} \times \calY;\M^{3 \times 3}_{\sym})$},
\end{align*}
for every $t \in [0,T]$. 
From \CCC \Cref{two-scale weak limit of scaled strains - 2x2 submatrix} \BBB, we can conclude for every $t \in [0,T]$ that $u(t) \in KL(\ext{\Omega})$. 
Furthermore, according to \Cref{two-scale weak limit of scaled strains}, one can choose the above subsequence in a way such that there \RED exist $\mu(t) \in \calXzero{\ext{\omega}}$, $\kappa(t) \in \calYzero{\ext{\omega}}$ and $\zeta(t) \in \Mb(\ext{\Omega} \times \calY;\R^3)$ \BLACK such that
\begin{equation*}
    \Lambda_h Eu^{h_t}(t) \weakstartwoscale Eu(t) \otimes \calL^{2}_{y} + \RED \begin{pmatrix} \begin{matrix} E_{y}\mu(t) - x_3 D^2_{y}\kappa(t) \end{matrix} & \zeta'(t) \\ (\zeta'(t))^T & \zeta_3(t) \end{pmatrix} \BLACK.
\end{equation*}

Since, $\Lambda_{h_t} Eu^{h_t}(t) = \Lambda_{h_t} e^{h_t}(t) + \Lambda_{h_t} p^{h_t}(t)$ in $\ext{\Omega}$ for every $h > 0$ and $t \in [0,T]$, we deduce that \RED  $(u(t),E^{\prime\prime}(t),P^{\prime\prime}(t)) \in \calA^{hom}_{0}(w(t))$ \BLACK.

Lastly, we consider for every $t \in [0,T]$
\begin{equation*}
    \sigma^{h_t}(t) := \C\left(\tfrac{x'}{\epsh_t}\right) \Lambda_{h_t} e^{h_t}(t).
\end{equation*}
Then we can choose a (not relabeled) subsequence, such that
\begin{equation} \label{main result sigma^h(t)}
   \sigma^{h_t}(t) \weaktwoscale \Sigma(t) \quad \text{two-scale weakly in $L^2(\ext{\Omega} \times \calY;\M^{3 \times 3}_{\sym})$},
\end{equation}
where $\Sigma(t) := \C(y) E(t)$. 
Since $\sigma^{h_t}(t) \in \calK_{h_t}$ for every $t \in [0,T]$, by \Cref{two-scale weak limit of admissible stress - regime zero} we can conclude $\Sigma(t) \in \RED \calK^{hom}_{0} \BLACK$. \CCC From this it follows that $E(t)=\A_y E''(t)$. \BBB

\noindent{\bf Step 2: \em Global stability.}

Since from Step 1 we have \CCC$(u(t),E''(t),P''(t)) \in  \calA^{hom}_{0}(w(t)) $ \BBB with the associated \RED $\mu(t) \in \calXzero{\ext{\omega}}$, $\kappa(t) \in \calYzero{\ext{\omega}}$  \BLACK, then for every $(\upsilon,H,\Pi) \in \RED \calA^{hom}_{0}(w(t)) \BLACK$ with the associated \RED $\nu(t) \in \calXzero{\ext{\omega}}$, $\lambda(t) \in \calYzero{\ext{\omega}}$  \BLACK we have
\CCC
\begin{equation*}
    (\upsilon-u(t),H-E''(t),\Pi-P''(t)) \in \RED \calA^{hom}_{0}(0).
\end{equation*}
\BBB
Furthermore, since from the first step of the proof \CCC $\C_r(y) E''(t) \BBB \in \RED \calK^{hom}_{0} \BLACK$, by \Cref{two-scale dissipation and plastic work inequality} we have 
\CCC
\begin{align*}
    \calH^{hom}_{\CCC 0 \BBB} (\Pi-P''(t)) &\geq -\int_{\omega \times I \times \calY} \C_r(y) E''(t) : (H-E''(t)) \,dx dy\\
    &= \calQ^{hom}_{\CCC 0 \BBB}(E''(t)) + \calQ^{hom}_{\CCC 0 \BBB}(H-E''(t)) - \calQ^{hom}_{\CCC 0 \BBB}(H),
\end{align*}
\BBB
where the last equality is a straightforward computation. 
From the above, we immediately deduce
\CCC
\begin{equation*}
    \calH^{hom}_0(\Pi-P''(t)) + \calQ^{hom}_0(H) \geq \calQ^{hom}(E''(t)) + \calQ^{hom}_0(H-E''(t)) \geq \calQ^{hom}_0(E''(t)),
\end{equation*}
\BBB
hence the global stability of the two-scale quasistatic evolution \ref{hom-qs S}.

We proceed by proving that the limit functions $u(t)$ and $E(t)$ do not depend on the subsequence. 
\CCC Since $E(t)=\A_y E''(t)$, it is enough to conclude that $E''(t)$ is unique. \BBB
Assume $(\upsilon(t),H(t),P(t)) \in \RED \calA^{hom}_{0}(w(t)) \BLACK$ with the associated \RED $\nu(t) \in \calXzero{\ext{\omega}}$, $\lambda(t) \in \calYzero{\ext{\omega}}$  \BLACK also satisfy the global stability of the two-scale quasistatic evolution. 
By the strict convexity of $\calQ^{hom}_{\CCC 0 \BBB}$, we immediately obtain that
\begin{equation*}
    H(t) = \CCC E''\BBB(t).
\end{equation*}
Identifing $Eu(t), E\upsilon(t)$ with elements of $\Mb(\ext{\Omega};\M^{2 \times 2}_{\sym})$ and using \eqref{admissible two-scale configurations - regime zero}, we have that
\begin{align*}
    E\upsilon(t) \otimes \calL^{2}_{y} + \RED E_{y}\nu(t) - x_3 D^2_{y}\kappa(t) \BLACK &= H(t) \,\calL^{3}_{x} \otimes \calL^{2}_{y} + P(t)\\
    &= E(t) \,\calL^{3}_{x} \otimes \calL^{2}_{y} + P(t)\\
    &= Eu(t) \otimes \calL^{2}_{y} + \RED E_{y}\mu(t) - x_3 D^2_{y}\kappa(t) \BLACK.
\end{align*}
Integrating over $\calY$, 
we obtain
\begin{equation*}
    E\upsilon(t) = Eu(t).
\end{equation*}
Using the variant of Poincar\'{e}-Korn's inequality as in Step 1, we can infer that $\upsilon(t) = u(t)$ on $\ext{\Omega}$.

This implies that the whole sequences converge without depening on $t$, i.e.
\begin{align*}
    &u^{h}(t) \weakstar u(t) \quad \text{weakly* in $BD(\ext{\Omega})$},\\
    &\Lambda_h e^h(t) \weaktwoscale E(t)\CCC=\A_y E''(t) \BBB \quad \text{two-scale weakly in $L^2(\ext{\Omega} \times \calY;\M^{3 \times 3}_{\sym})$}.
\end{align*}

\noindent{\bf Step 3: \em Energy balance.}

In order to prove energy balance of the two-scale quasistatic evolution \ref{hom-qs E}, it is enough (by arguing as in, e.g. \cite[Theorem 4.7]{DalMaso.DeSimone.Mora.2006} and \cite[Theorem 2.7]{Francfort.Giacomini.2012}) to prove the energy inequality
\begin{align} \label{step 3 inequality}
\begin{split}
    &\calQ^{hom}_{\CCC 0\BBB} (E''(t)) + \calD_{\calH^{hom}_{\CCC 0\BBB}}(P''; 0, t) \\
    &\leq \calQ^{hom}_{\CCC 0\BBB}(\CCC E'' \BBB(0)) 
    + \int_0^t \int_{\Omega \times \calY} \C_{\CCC r \BBB} (y) E''(s) : E\dot{w}(s) \,dx dy ds.
\end{split}
\end{align}

For a fixed $t \in [0,T]$, let us consider a subdivision $0 = t_1 < t_2 < \ldots < t_n = t$ of $[0,t]$. 
In view of the lower semicontinuity of $\calQ^{hom}$ and $\calH^{hom}$ \CCC as a consequence of the convexity of $Q$ and Reshetnyak lower-semicontinuity (see \cite[Theorem 2.38]{ambrosio2000functions} and \Cref{transfertwoscale} , see also \cite[Lemma 4.6]{Francfort.Giacomini.2014}) \BBB from \ref{h-qs E} we have 
\begin{align*}
    &\calQ^{hom}_{\CCC 0\BBB}(E(t)) + \sum_{i = 1}^{n} \calH^{hom}_{\CCC 0\BBB}\left(P(t_{i+1}) - P(t_i)\right)\\
    &\leq \liminf\limits_{h}\left( \calQ_h(\Lambda_h e^h(t)) + \sum_{i = 1}^{n} \calH_h\left(\Lambda_h p^h(t_{i+1}) - \Lambda_h p^h(t_i)\right) \right)\\
    &\leq \liminf\limits_{h}\left( \calQ_h(\Lambda_h e^h(t)) + \calD_{\calH_h}(\Lambda_h p^h; 0, t) \right)\\
    &= \liminf\limits_{h}\left( \calQ_h(\Lambda_h e^h(0)) 
    + \int_0^t \int_{\Omega} \C\left(\tfrac{x'}{\epsh}\right) \Lambda_h e^h(s) : E\dot{w}(s) \,dx ds \right).
\end{align*}
In view of the strong convergence assumed in \eqref{main result e^h_0 condition} and \eqref{main result sigma^h(t)}, by the Lebesgue's dominated convergence theorem we infer
\begin{align*}
    &\lim\limits_{h}\left( \calQ_h(\Lambda_h e^h(0)) 
    + \int_0^t \int_{\Omega} \C\left(\tfrac{x'}{\epsh}\right) \Lambda_h e^h(s) : E\dot{w}(s) \,dx ds \right)\\
    &= \calQ^{hom}_{\CCC 0\BBB}(E(0)) 
    + \int_0^t \int_{\Omega \times \calY} \C_{\CCC r \BBB} (y)  E''(s) : E\dot{w}(s) \,dx dy ds.
\end{align*}
Hence, we have
\CCC
\begin{align*}
    &\calQ^{hom}_0(E(t)) + \sum_{i = 1}^{n} \calH^{hom}_0\left(P''(t_{i+1}) - P''(t_i)\right) \\
    &\leq \calQ^{hom}_0(E''(0)) 
    + \int_0^t \int_{\Omega \times \calY} \C_r(y)  E''(s) : E\dot{w}(s) \,dx dy ds.
\end{align*}
\BBB
Taking the supremum over all partitions of $[0,t]$ yields \eqref{step 3 inequality}, which concludes the proof, \CCC after replacement of $E$ with $E''$ and $P$ with $P''$. \BBB
\end{proof}

\section*{Acknowledgements}
M.~Bu\v{z}an\v{c}i\'{c} and I.~Vel\v{c}i\'{c} were supported by the Croatian Science Foundation under
Grant Agreement no. IP-2018-01-8904 (Homdirestroptcm). 
The research of E.Davoli was supported by the Austrian Science Fund (FWF) projects F65, V 662, Y1292, and I 4052. All authors are thankful for the support from the OeAD-WTZ project HR 08/2020.
\printbibliography

\end{document}